\documentclass[a4paper, 10pt, final]{amsart}

\usepackage{amsthm}

\theoremstyle{plain}

\usepackage[latin1]{inputenc}
\usepackage[T1]{fontenc}
\usepackage[english]{babel}
\usepackage{amsmath}
\usepackage{amsfonts}
\usepackage{amssymb}
\usepackage{amscd}
\usepackage{mathrsfs}

\usepackage{lmodern}

\usepackage[OT2,T1]{fontenc}

\usepackage{amsmath, amsthm, amsfonts}

\DeclareMathOperator{\Lie}{Lie}
\DeclareMathOperator{\weight}{weight}
\DeclareMathOperator{\id}{id}

\DeclareMathOperator{\an}{an}
\DeclareMathOperator{\dR}{dR}

\DeclareMathOperator{\Wd}{Wd}
\DeclareMathOperator{\Ad}{Ad}

\DeclareMathOperator{\depth}{depth}

\DeclareMathOperator{\un}{un}

\DeclareMathOperator{\har}{har}
\DeclareMathOperator{\Spec}{Spec}
\DeclareMathOperator{\B}{B}

\DeclareMathOperator{\Li}{Li}

\DeclareMathOperator{\fix}{fix}

\DeclareMathOperator{\crys}{crys}

\DeclareMathOperator{\iter}{iter}
\DeclareMathOperator{\comp}{comp}

\DeclareMathOperator{\KZ}{KZ}

\theoremstyle{plain}

\newtheorem{Theoreme}{Theoreme}[subsection]
\newtheorem{Proposition}[Theoreme]{Proposition}

\newtheorem{Proposition-Definition}[Theoreme]{Proposition-Definition}
\newtheorem{Lemma-Notation}[Theoreme]{Lemma-Notation}
\newtheorem{Lemma-Definition}[Theoreme]{Lemma-Definition}

\newtheorem{Nota Bene}[Theoreme]{Nota Bene}
\newtheorem{Corollary}[Theoreme]{Corollary}
\theoremstyle{definition}

\newtheorem{Lemma}[Theoreme]{Lemma}
\newtheorem{Definition}[Theoreme]{Definition}
\newtheorem{Remark}[Theoreme]{Remark}

\newtheorem{Example}[Theoreme]{Example}

\newtheorem{Notation}[Theoreme]{Notation}

\newtheorem{Sub-lemma}[Theoreme]{Sub-lemma}

\newcommand{\simlra}{\buildrel \sim \over \longrightarrow}

\numberwithin{equation}{section}

\input cyracc.def
\DeclareFontFamily{U}{russian}{}
\DeclareFontShape{U}{russian}{m}{n}
        { <5><6> wncyr5
        <7><8><9> wncyr7
        <10><10.95><12><14.4><17.28><20.74><24.88> wncyr10 }{}
\DeclareSymbolFont{Russian}{U}{russian}{m}{n}
\DeclareSymbolFontAlphabet{\mathcyr}{Russian}
\makeatletter
\let\@math@cyr\mathcyr
\renewcommand{\mathcyr}[1]{\@math@cyr{\cyracc #1}}
\makeatother
\newcommand{\sh}{\mathcyr{sh}} 

\makeatletter
\newcounter{subsubsubsection}[subsubsection]
\renewcommand\thesubsubsubsection{\thesubsubsection .\@alph\c@subsubsubsection}
\newcommand\subsubsubsection{\@startsection{subsubsubsection}{4}{\z@}%
                                     {-3.25ex\@plus -1ex \@minus -.2ex}%
                                     {1.5ex \@plus .2ex}%
                                     {\normalfont\normalsize\bfseries}}
\newcommand*\l@subsubsubsection{\@dottedtocline{3}{10.0em}{4.1em}}
\newcommand*{\subsubsubsectionmark}[1]{}
\makeatother

\usepackage[top=2.5cm, bottom=2.5cm, left=2.5cm, right=2.5cm]{geometry}

\author{David Jarossay}

\address{Department of Mathematics, Ben Gurion University of the Negev, Be'er-Sheva`, Israel}

\email{jarossay@post.bgu.ac.il}

\begin{document}

\title{Pro-unipotent harmonic actions and dynamical properties of $p$-adic cyclotomic multiple zeta values}

\renewcommand{\shorttitle}{Pro-unipotent harmonic actions and dynamical properties of $p$-adic cyclotomic MZV's}

\maketitle

\begin{abstract}
$p$-adic cyclotomic multiple zeta values depend on the choice of a number of iterations of the crystalline Frobenius of the pro-unipotent fundamental groupoid of $\mathbb{P}^{1} - \{0,\mu_{N},\infty\}$. In this paper we study how the iterated Frobenius depends on the number of iterations, in relation with the computation of $p$-adic cyclotomic multiple zeta values in terms of cyclotomic multiple harmonic sums. This provides new results on that computation and the definition of a new pro-unipotent harmonic action.

This is Part I-3 of \emph{$p$-adic cyclotomic multiple zeta values and $p$-adic pro-unipotent harmonic actions}.
\end{abstract}

\tableofcontents

\setcounter{section}{-1}

\section{Introduction}

\subsection{$p$-adic cyclotomic multiple zeta values, computation and iteration of the Frobenius}

Cyclotomic multiple zeta values are the following iterated integrals: for any positive integers, $d$ and $n_{i}$ $(1\leq i \leq d)$ and roots of unity $\xi_{i}$ $(1\leq i \leq d)$, such that $(n_{d},\xi_{d})\not=(1,1)$, 
\begin{equation} \label{eq:iterated integrals}
\zeta\big((n_{i})_{d};(\xi_{i})_{d}\big) = (-1)^{d} \int_{t_{n}=0}^{1} \frac{dt_{n}}{t_{n}-\epsilon_{n}} \int_{t_{n-1}=0}^{t_{n}} \cdots \int_{t_{1}=0}^{t_{2}} \frac{dt_{1}}{t_{1}-\epsilon_{1}} ,
\end{equation}

where $n=\sum\limits_{i=1}^{d}n_{i}$ and $(\epsilon_{n},\ldots,\epsilon_{1}) = (\overbrace{0,\ldots,0}^{n_{d}-1},1,\ldots,\overbrace{0,\ldots,0}^{n_{1}-1},1)$.
We choose $N$ such that the $\epsilon_{i}$'s are $N$-th roots of unity. Let $p$ be a prime number prime to $N$. $p$-adic cyclotomic multiple zeta values are defined as $p$-adic analogues of the above iterated integrals. They are elements of the extension $K$ of $\mathbb{Q}_{p}$ generated by a primitive $N$-th root of unity. There are two types of $p$-adic cyclotomic multiple zeta values; both of the notions rely on the Frobenius of the crystalline pro-unipotent fundamental groupoid of $\mathbb{P}^{1} - \{0,\mu_{N},\infty\}$ at the base-points $\vec{1}_{0}$ and $\vec{1}_{1}$, as follows (see \S1.1.3 for details):

(i) numbers $\zeta_{q}^{\KZ}\big((n_{i})_{d};(\xi_{i})_{d}\big) \in K $ defined by Coleman integration i.e. by using a Frobenius-invariant path \cite{Furusho 1, Furusho 2, Yamashita} (here, $q$ is the cardinality of the residue field of $K$).

(ii) for each $\alpha \in \mathbb{Z} - \{0\}$, numbers $\zeta_{p,\alpha}\big((n_{i})_{d};(\xi_{i})_{d}\big) \in K$ defined by the image by Frobenius iterated $\alpha$ times of the canonical path in the de Rham fundamental groupoid of $\mathbb{P}^{1} - \{0,\mu_{N},\infty\}$ (\cite{I-1}, \S1, and for particular values of $\alpha$, \cite{Yamashita, Unver 2, Deligne Goncharov, Unver 1}).

$p$-adic cyclotomic multiple zeta values can be considered as canonical coefficients of the Frobenius, and conversely one can recover the Frobenius by knowing only $p$-adic cyclotomic multiple zeta values.

Cyclotomic multiple harmonic sums are the following numbers:

\begin{equation} \label{eq:harmonic sum} h_{m} \big((n_{i})_{d};(\xi_{i})_{d+1}\big)  =
\sum_{0<m_{1}<\ldots<m_{d}<m}
\frac{\big( \frac{\xi_{2}}{\xi_{1}} \big)^{n_{1}} \ldots \big(\frac{\xi_{d+1}}{\xi_{d}}\big)^{n_{d}} \big(\frac{1}{\xi_{d+1}}\big)^{m}}{m_{1}^{n_{1}}\ldots m_{d}^{n_{d}}} .
\end{equation}

In the complex case we have: $\zeta\big((n_{i})_{d};(\xi_{i})_{d}\big) = \underset{m\rightarrow \infty}{\lim} h_{m} \big((n_{i})_{d};((\xi_{i})_{d},1)\big)$. Similarly it is possible to compute $p$-adic cyclotomic multiple zeta values in terms of cyclotomic multiple harmonic sums \cite{I-2}, thanks to a big combinatorial simplification proved by the main result of \cite{I-1}. (A different computation in the $N=1$ and $\alpha=-1$ resp. the $d\leq 2$ and $\alpha=-1$ case, which does not use the simplification of \cite{I-1} and whose results are more complicated and seem difficult to use, appears in \cite{Unver 3} resp. \cite{Unver 2}.)
In this paper we are going to study the following question: how does the iterated Frobenius depend on its number of iterations ? More specifically, we are going to connect this question and the framework of \cite{I-2}.

\subsection{Principles of the study}

Most of the time, we are not going to consider directly the Frobenius but, instead, the \emph{harmonic Frobenius}, defined in \cite{I-2}, definition 2.3.5, (we will reproduce it in \S1.4). It is a variant of the Frobenius which is much simpler and more natural from the point of view of multiple harmonic sums, and computing it suffices to compute the Frobenius.

Whereas the Frobenius is an isomorphism of bundles with connection, the harmonic Frobenius is a map on a space which contains the non-commutative generating series of weighted multiple harmonic sums $\har_{m}\big((n_{i})_{d};(\xi_{i})_{d+1}\big) = 
m^{n_{d}+\ldots+n_{1}} h_{m}\big((n_{i})_{d};(\xi_{i})_{d+1}\big)$.

We will use the fact that the harmonic Frobenius can be expressed in two ways:

(a) One ``in terms of integrals'', i.e. in which the coefficients of the harmonic Frobenius are expressed in terms of $p$-adic cyclotomic multiple zeta values, which are integrals and which we want to compute.

(b) Another one ``in terms of series'', in which the coefficients of the harmonic Frobenius are certain sums of series expressed in terms of the numbers $\har_{p^{\alpha}}\big((n_{i})_{d};(\xi_{i})_{d+1}\big)$, which are explicit.

In \cite{I-2}, by writing these two expressions and observing that they are are equal, we get an expression for $p$-adic cyclotomic multiple zeta values in terms of the numbers $\har_{p^{\alpha}}\big((n_{i})_{d};(\xi_{i})_{d+1}\big)$ and vice-versa. We are going to do something similar here, not for the harmonic Frobenius but for the study of the numbers $\har_{q^{\tilde{\alpha}}}\big((n_{i})_{d};(\xi_{i})_{d+1}\big)$, as functions of $\tilde{\alpha} \in \mathbb{N}^{\ast}$. Since the harmonic Frobenius can be expressed in terms of these numbers, this will directly provide a study of the iterated harmonic Frobenius in terms of its number of iterations.

After some preliminaries (\S1) we will do this study in terms of integrals (\S2), in terms of series (\S3), and we will use the fact that these two ways give the same result (\S4). In \S5 we will go back from the harmonic Frobenius to the Frobenius.

Moreover, we will keep track of the motivic structure underlying this framework. Indeed, $p$-adic cyclotomic multiple zeta values are reductions of $p$-adic periods \cite{Yamashita}, and there is a motivic Galois action on the pro-unipotent fundamental groupoid of $\mathbb{P}^{1} - \{0,\mu_{N},\infty\}$ (\cite{Deligne Goncharov}, \S5).

The Frobenius is expressed by means of the Ihara action (\ref{eq:Ihara action}), which is the image of the motivic Galois action by a certain morphism (see \S1.1.3). In \cite{I-2}, the passage from the Frobenius to the harmonic Frobenius lifts to a passage from the Ihara action to an operation which we called the pro-unipotent harmonic action of integrals, and we also find a pro-unipotent harmonic action of series (see \S1.4). The interest of pro-unipotent harmonic actions is that, being byproducts of the motivic Galois action, they retain certain properties of motivic Galois actions; and having a computation which keeps track of the motivic Galois action is key for us. The pro-unipotent harmonic actions will be the main objects in our subsequent papers \cite{II-1, II-2, II-3} in which we will show the compatibility between our computation and the motivic Galois theory of $p$-adic cyclotomic multiple zeta values.

Establishing the definition of pro-unipotent harmonic actions requires enriching the pro-unipotent fundamental groupoid, which is a groupoid in affine schemes over $\mathbb{P}^{1} - \{0,\mu_{N},\infty\}$ by turning it into a groupoid in ultrametric complete normed algebras (\cite{I-2}, \S1).

\subsection{A few definitions}

The study will require new definitions. First, we will define an ad hoc notion of contraction mapping (Definition \ref{def contracting}). We will show that the Frobenius at base-points $(\vec{1}_{1},\vec{1}_{0})$ is a contraction in our ad hoc sense. This will shed light on the dynamics of the Frobenius which has a unique fixed point. Thus, the ultrametric framework established in \cite{I-2} will be crucial here because in this framework we will introduce a notion of contraction mapping and we will see that the Frobenius is a contraction. Keeping track of the motivic structures will also require new definitions. 

We will study of $\har_{q^{\tilde{\alpha}}}$ as a function of $\tilde{\alpha}$ ``in terms of integrals'' (\S2) by viewing $\har_{q^{\tilde{\alpha}}}$ via their expression in terms of $p$-adic cyclotomic multiple zeta values proved in the main theorem of \cite{I-2}. This will be done in two different ways, corresponding to the two types of $p$-adic cyclotomic multiple zeta values evoked in \S0.1.

(i) A way involving the fixed point of the Frobenius and the numbers $\zeta_{q}^{\KZ}$. It will lead us to introduce a pro-unipotent harmonic action of integrals at $(1,0)$, $\circ_{\har}^{\smallint_{1,0}}$ (Definition \ref{def de Rham Ihara}), which is a variant of the notion introduced in \cite{I-2}. Another point of view on this object will be explained in the appendix.

(ii) A way involving the numbers $\zeta_{p,\alpha_{0}}$. It will lead us to introduce a map  $\iter_{\har}^{\smallint}(\textbf{a},\Lambda)$ of iteration of the harmonic Frobenius of integrals at (1,0), ($\Lambda$ and $\textbf{a}$ are formal variables which represent respectively $q^{\tilde{\alpha}}$ and $\frac{\tilde{\alpha}}{\tilde{\alpha}_{0}}$) (Definition \ref{def iteration of the harmonic Frobenius of series}).

In the study of $\har_{q^{\tilde{\alpha}}}$ as a function of $\tilde{\alpha}$ in terms of series (\S3), we do not have an analogue of the fixed point and the study will lead us to introduce a map of iteration of the harmonic Frobenius of series $\iter_{\har}^{\Sigma}(\textbf{a},\Lambda)$
(Definition \ref{def iteration of the harmonic Frobenius of series}).

Finally, in \S4, we will relate \S2 and \S3 by defining a map of comparison between series and integrals, which will be injective thanks to the results of \S2 and \S3.

As in \cite{I-2} these definitions enable us to express the computation not number by number, but as a new structure on the pro-unipotent fundamental group, which is more efficient. Indeed, this structure retains certain features of the motivic Galois theory of periods, which will be crucial in our subsequent papers \cite{II-1, II-2, II-3} in which we will relate the motivic Galois theory of $p$-adic cyclotomic multiple zeta values to our formulas.

\subsection{Results}

The main result consists of three equations to express $\har_{q^{\tilde{\alpha}}}$ as a function of $\tilde{\alpha}$, and the comparison between them.

The first two equations (proved in \S2), in which the harmonic Frobenius is thought of in terms of integrals, correspond to (i) and (ii) above.
The first one (\ref{eq:first of I-3}) involves the fixed-point of the Frobenius and will be called the \emph{fixed point equation of the harmonic Frobenius of integrals at (1,0)}; the second one (\ref{eq:second of I-3}) will be the \emph{iteration equation of the harmonic Frobenius of integrals at (1,0)}.

Finally, the third equation (proved in \S3) (\ref{eq:third of I-3}), in which the harmonic Frobenius is thought of in terms of series, will be the \emph{iteration equation of the harmonic Frobenius of integrals at (1,0)}.

In (\ref{eq:second of I-3}) and (\ref{eq:third of I-3}) the dependence of $\har_{q^{\tilde{\alpha}}}$ in $\tilde{\alpha}$ is via a power series in $K[[q^{\tilde{\alpha}}]][\tilde{\alpha}]$ and in (\ref{eq:first of I-3}) it is via a power series in $K[[q^{\tilde{\alpha}}]]$. We are going to see that these expansions are equal (\ref{eq:fourth of I-3}): in particular, the coefficients of $(q^{\tilde{\alpha}})^{0}\tilde{\alpha}^{m}$ for $m\geq 1$ will vanish.

In the statement below, for any $\tilde{\alpha} \in \mathbb{Z} \cup \{\pm \infty\} - \{0\}$, $\Phi_{q,\tilde{\alpha}}$ is the generating series of the numbers $\zeta_{q,\tilde{\alpha}}$ and $\Phi_{q,-\infty}=\Phi_{q}^{\KZ}$ is the generating series of the numbers $\zeta_{q}^{\KZ}$ (Notation \ref{notation generating series}), $\har_{q,\alpha}$ is a generating series of generalized prime weighted cyclotomic multiple harmonic sums (Definition \ref{generalized pMHS}), and $\tau$ is defined in equation (\ref{eq:tau}).
\newline 
\newline \textbf{Theorem.} \emph{Let $\tilde{\alpha}_{0},\tilde{\alpha} \in \mathbb{N}^{\ast}$ such that $\tilde{\alpha}_{0} | \tilde{\alpha}$.
\newline (i) (integrals)
The pro-unipotent harmonic action of integrals at $(1,0)$, denoted by $\circ_{\har}^{\smallint_{1,0}}$, is a continuous group action, and we have:
\begin{equation}
\label{eq:first of I-3}\har_{q,\tilde{\alpha}} =  \tau(q^{\tilde{\alpha}}) \big(
\Phi_{q,-\infty}^{-1}e_{1}\Phi_{q,-\infty} \big) \circ_{\har}^{\smallint_{1,0}} \har_{q,-\infty} .
\end{equation}
The map of iteration of the Frobenius of integrals at (1,0), denoted by $\iter_{\har}^{\smallint_{1,0}}$, satisfies, at words $w=\big((n_{i})_{d};(\xi_{i})_{d+1}\big)$ such that $\frac{\tilde{\alpha}}{\tilde{\alpha}_{0}} > d$:
\begin{equation}
\label{eq:second of I-3} \har_{q,\tilde{\alpha}}(w) = \iter_{\har}^{\smallint_{1,0}}\big(\frac{\tilde{\alpha}}{\tilde{\alpha}_{0}},q^{\tilde{\alpha}_{0}}\big)\big(\Phi^{-1}_{q,\tilde{\alpha}_{0}}e_{1}\Phi_{q,\tilde{\alpha}_{0}}\big)(w)
\end{equation}
(ii) (series) The map of iteration of the Frobenius of series, denoted by $\iter_{\har}^{\Sigma}$, satisfies:
\begin{equation} \label{eq:third of I-3}\har_{q,\tilde{\alpha}} =  \iter_{\har}^{\Sigma}\big(\frac{\tilde{\alpha}}{\tilde{\alpha}_{0}},q^{\tilde{\alpha}_{0}} \big)\big(\har_{q,\tilde{\alpha}_{0}}\big) .
\end{equation}
\newline (iii) (comparison between integrals and series) We have the following equalities of formal power series with formal variables $a$ and $\Lambda$:
\begin{equation} \label{eq:fourth of I-3} \tau(\Lambda)(\Phi_{q,-\infty}^{-1}e_{1}\Phi_{q,-\infty}) \circ_{\har}^{\smallint_{1,0}} \har_{q,-\infty} = \iter_{\har}^{\smallint}\big(a,\Lambda\big)\big(\Phi^{-1}_{q,\tilde{\alpha}_{0}}e_{1}\Phi_{q,\tilde{\alpha}_{0}}\big) = \iter_{\har}^{\Sigma}\big(a,\Lambda\big)\big(\har_{q,\tilde{\alpha}_{0}}\big) .
\end{equation}}

The first terms of the equations of the theorem are written in Example \ref{example of the theorem}.

In \S5, we deduce a similar result for the iteration of the Frobenius on the affinoid analytic subspace $\mathbb{P}^{1,\an} - \underset{\xi \in \mu_{N}(K)}{\cup} B(\xi,1)$ of $\mathbb{P}^{1,\an}/K$, knowing that the fixed-point equation of the Frobenius is already given by Coleman integration. One of these equations uses the regularization of $p$-adic iterated integrals studied in \cite{I-1}. In the appendix, we explain that the pro-unipotent harmonic action of integrals in (1,0) corresponds to a certain Poisson bracket.

The main result provides a natural way to compute the fixed point $\Phi_{q,-\infty}$ i.e. $p$-adic cyclotomic multiple zeta values in the sense of Coleman integration. Indeed, we see that the fixed point of the Frobenius $\Phi_{q,-\infty} \in \Pi_{1,0}(K)$ appears naturally as a way to express the coefficients of the iteration equations, and that these iteration equations can be understood in terms of explicit sums of series. This gives a way to compute Coleman integration without directly doing Coleman integration.

The main result also allows us to replace the map of comparison from integrals to series defined in \cite{I-2} by a map which has the advantage of being injective.

From a dynamical point of view, the main result gives an asymptotic expansion at infinite order of the convergence of the iterated (harmonic) Frobenius towards its fixed point. More precise information would follow from non-vanishing results or results on the valuation of $p$-adic cyclotomic multiple zeta values, or of certain infinite sums of them. This gives a corespondence between certain arithmetical properties of $p$-adic cyclotomic multiple zeta values and dynamical properties of the Frobenius. 
A correspondence between dynamical properties of the Frobenius and analytic properties of cyclotomic multiple harmonic sums is also deduced in \S5.

The pro-unipotent harmonic action $\circ_{\har}^{\smallint_{1,0}}$ which is defined in this paper will be central in our next papers \cite{II-1, II-2, II-3} on the explicit version of the algebraic theory of $p$-adic cyclotomic multiple zeta values. We will also see there that considering the iterates of the Frobenius, instead of only the Frobenius itself, is necessary to formulate an explicit version of the algebraic theory of $p$-adic cyclotomic multiple zeta values which is purely $p$-adic and not adelic. It will also find an application in \cite{III-1}, where we will see that we can construct a structure on $\pi_{1}^{\un,\dR}(\mathbb{P}^{1} - \{0,\mu_{p^{\alpha}N},\infty\})$ which generalizes the crystalline Frobenius on $\pi_{1}^{\un,\dR}(\mathbb{P}^{1} - \{0,\mu_{N},\infty\})$ iterated $\alpha$ times. Considering two parameters $\tilde{\alpha}_{0}$ and $\tilde{\alpha}$ with $\tilde{\alpha}_{0} | \tilde{\alpha}$ and not just $\tilde{\alpha}$ will also be useful in \cite{III-1} to shed light on the computation of $p$-adic cyclotomic multiple zeta values.

\emph{Acknowledgments.} This work has been done at Universit\'{e} Paris Diderot and Universit\'{e} de Strasbourg, with, respectively, the support of ERC grant 257638 and Labex IRMIA. It has been revised at  Universit\'{e} de Gen\`{e}ve and Ben Gurion University of the Negev with, respectively, the support of NCCR SwissMAP and Ishai Dan-Cohen's ISF grant. I thank two anonymous referees whose feedback helped me improve the redaction of this paper. I also thank Stephanie Blumenstock for her help with editing.

\section{Preliminaries}

In this section we establish the framework of this paper. 
We review some definitions and properties about the pro-unipotent fundamental groupoid of $\mathbb{P}^{1} - \{0,\mu_{N},\infty\}$, some results \cite{I-2} and we add to them a few new definitions and notations. Throughout this paper, $\mathbb{N}$ resp. $\mathbb{N}^{\ast}$ will denote the set of non-negative resp. positive integers.

\subsection{The pro-unipotent fundamental groupoid of $\mathbb{P}^{1} - \{0,\mu_{N},\infty\}$}

\subsubsection{The de Rham realization}

Let $X$ be $\mathbb{P}^{1} - \{0,\mu_{N},\infty\}$ over the $p$-adic field $K$, with the notations of \S0.1. The de Rham pro-unipotent fundamental groupoid $\pi_{1}^{\un,\dR}(X)$, in the sense of \cite{Deligne}, is a groupoid in pro-affine schemes over $X$. Its base points are the points of $X$ and the non-zero tangent vectors at $\{0,\mu_{N},\infty\} \subset \mathbb{P}^{1}$, called tangential base-points. The groupoid structure is defined by the morphisms $\pi_{1}^{\un,\dR}(X_{K},z,y) \times \pi_{1}^{\un,\dR}(X_{K},y,x) \rightarrow \pi_{1}^{\un,\dR}(X_{K},z,x)$ for any base-points $x,y,z$. By \cite{Deligne} \S12.9, each $\pi_{1}^{\un,\dR}(X,y,x)$ is canonically isomorphic to the spectrum of the shuffle Hopf algebra $\mathcal{O}^{\sh,e_{0\cup\mu_{N}}}$ over the alphabet $e_{0\cup\mu_{N}} = 
\{e_{x}\text{ }|\text{ }x \in \{0\} \cup \mu_{N}(K)\}$. This isomorphism is compatible with the groupoid structure.
\newline\indent Following \cite{Deligne Goncharov}, for any $N$-th root of unity $\xi \in \mu_{N}(K)$, we denote by $\Pi_{\xi,0}= \pi_{1}^{\un,\dR}(X,\vec{1}_{\xi},\vec{1}_{0})$. Let $f \mapsto f^{(\xi)}$ be the isomorphism $\Pi_{1,0} \rightarrow \Pi_{\xi,0}$ induced by the automorphism $x \mapsto \xi x$ of $X$ by functoriality of $\pi_{1}^{\un,\dR}$.
\newline\indent Let
$K\langle \langle e_{0\cup \mu_{N}}\rangle\rangle$ be the non-commutative $K$-algebra of formal power series over the non-commuting variables $e_{x}$, $x \in \{0 \cup \mu_{N}(K)\}$. We will write an element $f \in K\langle\langle e_{0\cup \mu_{N}}\rangle\rangle$ as $f\big((e_{x})_{x}\big) = f\big((e_{x})_{x \in \{0\}\cup \mu_{N}(K)}\big)$ or $f\big(e_{0},(e_{\xi})_{\xi}\big)= f\big(e_{0},(e_{\xi})_{\xi}\big)$. The coefficient in $f$ of a word $w$ on the alphabet $e_{0 \cup \mu_{N}}$ is denoted by $f[w]$. This notation extends by linearity to linear combinations of words, and if for any $n\geq 0$ $w_{n}$ is a linear combination of words of weight $n$, we denote by $f[\sum_{n=0}^{\infty} w_{n}] = \sum_{n=0}^{\infty}f[ w_{n}]$ if this series converges. 
We have a canonical inclusion
$\Spec(\mathcal{O}^{\sh,e_{0\cup\mu_{N}}})(K) \subset K\langle \langle e_{0\cup \mu_{N}}\rangle\rangle$; namely, $\Spec(\mathcal{O}^{\sh,e_{0\cup\mu_{N}}})(K)$ is the group of elements satisfying the shuffle equation: for all words $w,w'$, $f[w]f[w'] = f[w \sh w']$ where $\sh$ is the shuffle product of words on the alphabet $e_{0\cup \mu_{N}}$.

\subsubsection{Motivic Galois action and byproducts}

The motivic version of $\pi_{1}^{\un}(\mathbb{P}^{1} - \{0,\mu_{N},\infty\})$ is constructed in \cite{Deligne Goncharov}, \S5. 

Let $G^{\omega}$ be the motivic Galois group associated with the Tannakian category of mixed Tate motives over $k_{N}$ unramified at $p$ prime to $N$, and with the canonical fiber functor $\omega$. There is a semi-direct product decomposition $G_{\omega} = \mathbb{G}_{m} \ltimes U_{\omega}$ where $U_{\omega}$ is pro-unipotent.

One has an action of $G^{\omega}$ on $\Pi_{1,0}$.  Let $V^{\omega}$ be the group of automorphisms defined in \cite{Deligne Goncharov}, \S5.10. There is a morphism $U^{\omega} \rightarrow V^{\omega}$ sending the action of $U^{\omega}$ on $\Pi_{1,0}$ to an action of $V^{\omega}$ on $\Pi_{1,0}$. This action makes $\Pi_{1,0}$ a torsor under $V^{\omega}$. Thus one can consider the isomorphism of schemes $V^{\omega} \simeq \Pi_{1,0}$, $v \mapsto v ({}_1 1_{0})$ where ${}_1 1_{0}$ is the canonical de Rham path in the sense of \cite{Deligne} \S12. This isomorphism sends the action of $V^{\omega}$ on $\Pi_{1,0}$ to the Ihara action (\cite{Deligne Goncharov}, \S5.11), namely, the group law $\circ^{\smallint_{1,0}}$ on $\Pi_{1,0}$ defined by 
\begin{equation} \label{eq:Ihara action} g \circ^{\smallint_{1,0}} f = g \big(e_{0},(e_{\xi})_{\xi}\big) \times f\big(e_{0},({g^{(\xi)}}^{-1}e_{\xi}g^{(\xi)})_{\xi}\big) .
\end{equation}
\indent The motivic Galois action of $\mathbb{G}_{m}$ on $\Pi_{1,0}$ is 
\begin{equation} \label{eq:tau} \tau: \Big( \lambda,f\big((e_{x})_{x}\big) \Big) \mapsto f\big((\lambda e_{x})_{x}\big) ,
\end{equation}

i.e. $\lambda$ acts by multiplying the term of weight $n$ in $f$ by $\lambda^{n}$, for all $n \in \mathbb{N}$. Let the collection of maps $(\tau_{n})_{n\in\mathbb{N}}$ be defined by the equality $\sum\limits_{n\in\mathbb{N}} \tau_{n}(f) \lambda^{n} = \tau(\lambda)(f)$ for all $\lambda$. Namely, $\tau_{n}$ sends $f = \sum\limits_{w\text{ word}} f[w]w$ to $\sum\limits_{\substack{w\text{ word}\\ \weight(w)=n}} f[w]w$. These formulas also define an action on $K\langle\langle e_{0\cup \mu_{N}}\rangle\rangle$ for which we will use the same notations.

\subsubsection{The crystalline realization}

Let $\sigma$ be the Frobenius automorphism of $K$. 
For $\alpha\in \mathbb{N}^{\ast}$, let $X^{(p^{\alpha})}$ be the pull back of $X$ by $\sigma$ iterated $\alpha$ times.

Let $\phi$ be the Frobenius of the crystalline pro-unipotent fundamental groupoid of $\mathbb{P}^{1} - \{0,\mu_{N},\infty\}$ (\cite{Deligne}, \S11 and \S13.6). It is a $\sigma$-linear isomorphism of groupoids $\pi_{1}^{\un,\dR}(X^{(p)}) \simlra \pi_{1}^{\un,\dR}(X)$. For any $\alpha\in \mathbb{N}^{\ast}$, the Frobenius iterated $\alpha$ times is $\phi_{\alpha}= (\sigma^{\alpha-1})^{\ast}\phi \circ \cdots \circ \sigma^{\ast}(\phi) \circ \phi$. It is a $\sigma^{\alpha}$-linear isomorphism of groupoids $\pi_{1}^{\un,\dR}(X^{(p^{\alpha})}) \simlra \pi_{1}^{\un,\dR}(X)$. When $\alpha$ is divisible by $\frak{o}=\frac{\log(p)}{\log(q)}$, then $\sigma^{\alpha}=\id$, thus $\phi_{\alpha}$ is $K$-linear in the usual sense, its source and target are the same, and it is equal to $\phi_{\frak{o}}$ iterated $\frac{\alpha}{\frak{o}}$ times: we will write $\alpha = \frak{o}\tilde{\alpha}$, and $\phi_{\frak{o}\tilde{\alpha}} = \phi_{\frak{o}}^{\tilde{\alpha}}$. We denote by $\phi_{-\alpha}=\phi_{\alpha}^{-1}$.

Let us now consider the Frobenius at base-points $(1,0)=(\vec{1}_{1},\vec{1}_{0})$: $\phi_{\alpha}: \Pi^{(p^{\alpha})}_{1,0} \simlra \Pi_{1,0}$ where $\Pi_{1,0}^{(p^{\alpha})} = \pi_{1}^{\un,\dR}(X^{(p^{\alpha})},\vec{1}_{1},\vec{1}_{0})$. The non-commutative generating series of $p$-adic cyclotomic multiple zeta values are $\Phi_{p,\alpha}=\tau(p^{\alpha})\phi_{\alpha}(1)\in \Pi_{1,0}(K)$, and $\Phi_{p,-\alpha}=\phi_{-\alpha}(1) \in \Pi^{(p^{\alpha})}_{1,0}(K)$. Let us denote again by $\sigma$ the map $K \langle \langle e_{0 \cup \mu_{N}}\rangle\rangle \rightarrow K \langle \langle e_{0 \cup \mu_{N}}\rangle\rangle$ defined by applying the Frobenius $\sigma$ of $K$ to each coefficient of a formal power series. The formal properties of the Frobenius imply the following formulas:
\begin{equation} \label{eq: formula for Frob 1} \tau(p^{\alpha})\phi_{\alpha}: f \in \Pi_{1,0}^{(p^{\alpha})}(K) \mapsto \Phi_{p,\alpha} \circ^{\smallint_{1,0}} \sigma^{-\alpha}(f) \in \Pi_{1,0}(K) ,
\end{equation}
\begin{equation} \label{eq: formula for Frob 2} \phi_{-\alpha}: f \in \Pi_{1,0}(K) \mapsto \Phi_{p,-\alpha} \circ^{\smallint_{1,0}} \tau(p^{\alpha})\sigma^{\alpha}(f) \in \Pi_{1,0}^{(p^{\alpha})}(K) .
\end{equation}

\indent One also has the other notion $\Phi_{q}^{\KZ} \in \Pi_{1,0}(K)$ of a non-commutative generating series of $p$-adic cyclotomic multiple zeta values, defined by the following equality:

\begin{equation} \label{eq:Frobenius fixed point} \phi_{\frac{\log(q)}{\log(p)}}(\Phi_{q}^{\KZ}) = \Phi_{q}^{\KZ} ,
\end{equation}

where the existence and uniqueness of a fixed point of $\phi_{\frac{\log(q)}{\log(p)}}$ in $\Pi_{1,0}(K)$ follows from the theory of Coleman integration \cite{Coleman, Besser, Vologodsky}. The (ii) of the notation below will be justified by the results of \S2.

\begin{Notation} \label{notation generating series} (i) For any $\tilde{\alpha} \in \mathbb{Z} - \{0\}$, let $\Phi_{q,\tilde{\alpha}} = \Phi_{p,\frac{\log(q)}{\log(p)}\tilde{\alpha}}$.
\newline (ii) Let $\Phi_{q,-\infty} = \Phi_{q}^{\KZ}$, and let 
$\Phi_{q,\infty}$ be the inverse of $\Phi_{q,-\infty}$ for the Ihara product  $\circ^{\smallint_{1,0}}$.
\newline (iii) For $\alpha\in \mathbb{Z} \cup \{\pm \infty\} - \{0\}$, the $p$-adic cyclotomic multiple zeta values are the numbers $\zeta_{q,\tilde{\alpha}}\big((n_{i})_{d};(\xi_{i})_{d}\big) = (-1)^{d} \Phi_{q,\tilde{\alpha}}[e_{0}^{n_{d}-1}e_{\xi_{d}}\ldots e_{0}^{n_{1}-1}e_{\xi_{1}}]$ (and similarly for $\zeta_{p,\alpha}$ and $\Phi_{p,\alpha}$.)
\end{Notation}

\subsubsection{Around the adjoint action $\Ad(e_{1})$}

We use the convention that the adjoint action $\Ad_{(.)}(x)$ on $\Pi_{1,0}$ is $f \mapsto f^{-1}xf$ (instead of the usual $f \mapsto fxf^{-1}$, due to our convention of reading the groupoid multiplication from the right to the left). The adjoint Ihara action, defined in \cite{I-2}, is the group law on $\Ad_{\Pi_{1,0}}(e_{1})$ defined by 
\begin{equation} \label{eq:adjoint Ihara action} h \circ_{\Ad}^{\smallint_{1,0}} f = f \big(e_{0},(h^{(\xi)})_{\xi}\big) . 
\end{equation}

Let $\tilde{\Pi}_{1,0}$ be the sub-group scheme of $\Pi_{1,0}$ defined by the equations $f[e_{1}] = f[e_{0}] = 0$ (cf. \S1.1.1); $\Ad(e_{1})$ induces an isomorphism of groups $(\tilde{\Pi}_{1,0}(K),\circ^{\smallint_{1,0}}) \simlra (\Ad_{\tilde{\Pi}_{1,0}(K)}(e_{1}), \circ_{\Ad}^{\smallint_{1,0}})$. 

By (\ref{eq:Ihara action}) and (\ref{eq:tau}), one has a semi-direct product $\mathbb{G}_{m} \ltimes \tilde{\Pi}_{1,0}$, which acts on $\tilde{\Pi}_{1,0}$. Similarly, by (\ref{eq:adjoint Ihara action}) and (\ref{eq:tau}), one has a semi-direct product $\mathbb{G}_{m} \ltimes \Ad_{\tilde{\Pi}_{1,0}}(e_{1})$, which acts on $\Ad_{\tilde{\Pi}_{1,0}}(e_{1})$. The map $\id \times \Ad(e_{1})$ induces an isomorphism between these two group actions. For all 
$f,g \in \Pi_{1,0}(K)$, $\lambda \in K^{\ast}$, $n \in \mathbb{N}$, we have:

\begin{equation} \label{eq:tau to tau n}g \circ^{\smallint_{1,0}} (\tau(\lambda)(f)) = \sum_{n \in \mathbb{N}} \lambda^{n} g \circ^{\smallint_{1,0}} (\tau_{n}f) .
\end{equation}
We have:

\begin{equation} \tau_{n+1} \circ \Ad(e_{1}) = \Ad(e_{1}) \circ \tau_{n} ,
\end{equation}
\begin{equation} \Ad_{g}(e_{1}) \circ^{\smallint_{1,0}}_{\Ad} \frac{\tau(\lambda)}{\lambda} \Ad_{f}(e_{1}) = \sum_{n \in \mathbb{N}}
\lambda^{n} \Ad_{g}(e_{1})  \circ_{\Ad}^{\smallint_{1,0}} \tau_{n+1}\Ad_{f}(e_{1}) .
\end{equation}

\subsection{An ultrametric structure on the $K$-points of the de Rham pro-unipotent fundamental groupoid}

As reviewed in \S1.1, each $\Pi_{y,x} = \pi_{1}^{\un,\dR}(\mathbb{P}^{1} - \{0,\mu_{N},\infty\},y,x)$ is an affine scheme over $K$, and we have a canonical embedding $\Pi_{y,x}(K) \subset K \langle \langle e_{0\cup \mu_{N}}\rangle\rangle$. We consider now an enrichment of $K \langle \langle e_{0\cup \mu_{N}}\rangle\rangle$ into a ultrametric complete normed $K$-algebra: we review facts from \cite{I-2}, and we add a few complements. In particular, in \S1.2.3 we add a notion of contraction and we apply it to the Frobenius at base-points (1,0).

\subsubsection{From affine schemes to ultrametric normed algebras over $K$}

For $n,d \in \mathbb{N}^{\ast}$, let $\Wd_{\ast,d}(e_{0\cup\mu_{N}})$, resp. $\Wd_{n,d}(e_{0\cup\mu_{N}})$ the set of words on $e_{0\cup\mu_{N}}$ that are of depth $d$, resp. of weight $n$ and depth $d$. Let $\Lambda$ and $D$ be two formal variables.
\newline\indent Let $K \langle\langle e_{0\cup\mu_{N}} \rangle\rangle_{<\infty} \subset K \langle \langle e_{0\cup\mu_{N}} \rangle \rangle$, be the subset of the elements $f$ such that, for each $d \in \mathbb{N}^{\ast}$, we have
$\underset{w \in \Wd_{\ast,d}(e_{0\cup\mu_{N}})}{\sup} |f[w]|_{p} < \infty$. We say that the elements of $K \langle\langle e_{0\cup\mu_{N}} \rangle\rangle_{<\infty}$ are bounded.
Let $K \langle \langle e_{0\cup\mu_{N}} \rangle \rangle_{o(1)} \subset K \langle \langle e_{0\cup\mu_{N}} \rangle \rangle_{<\infty}$ be the subset of the elements $f$ such that, for each $d \in \mathbb{N}^{\ast}$, we have $\underset{w \in \Wd_{n,d}(e_{0\cup\mu_{N}})}{\sup} \big|f[w] \big|_{p} \underset{n\rightarrow \infty}{\longrightarrow} 0$. We say that the elements of $K \langle\langle e_{0\cup\mu_{N}} \rangle\rangle_{o(1)}$ are summable.
\newline\indent Let $\mathcal{N}_{\Lambda,D}: f \in K \langle\langle e_{0\cup\mu_{N}} \rangle\rangle \mapsto 
\sum\limits_{(n,d) \in \mathbb{N}^{2}} \bigg(  
\underset{w \in\Wd_{n,d}(e_{0\cup\mu_{N}})}{\max} \big|f[w]\big|_{p} \bigg) \Lambda^{n}D^{d} \in  \mathbb{R}_{+}[[\Lambda,D]]$, and $\mathcal{N}_{D}: f \in K \langle\langle e_{0\cup\mu_{N}} \rangle\rangle_{<\infty} \mapsto \sum\limits_{d \in \mathbb{N}}
\bigg( \underset{w \in \Wd_{\ast,d}(e_{0\cup\mu_{N}})}{\sup} |f[w]|_{p} \bigg) D^{d} \in \mathbb{R}_{+}[[D]]$. One can check that these definitions give structures of complete normed ultrametric $K$-algebra on $K \langle\langle e_{0\cup\mu_{N}} \rangle\rangle$, $K \langle \langle e_{0\cup\mu_{N}} \rangle \rangle_{<\infty}$ and $K \langle \langle e_{0\cup\mu_{N}} \rangle \rangle_{o(1)}$ (\cite{I-2}, proposition 1.3.3).

\subsubsection{Compatibility between the ultrametric structure and the usual algebraic operations}

By \cite{I-2}, proposition 1.3.6, the Ihara product (\ref{eq:Ihara action}), the adjoint Ihara action (\ref{eq:adjoint Ihara action}), and the action $\tau$ (\ref{eq:tau}) are continuous relative to the topologies defined by $\mathcal{N}_{\Lambda,D}$ and $\mathcal{N}_{D}$ on $\Pi_{1,0}(K)$ and the $p$-adic topology on $K$. And for all $f,g \in \Pi_{1,0}(K)$, $\lambda \in K^{\times}$, we have (by \cite{I-2}, proof of proposition 1.3.6)
\begin{equation} \label{eq:Norm Ad}
\mathcal{N}_{\Lambda,D}(\Ad_{f}(e_{1})) \leqslant \Lambda D \mathcal{N}_{\Lambda,D}(f) ,
\end{equation}
\begin{equation} \label{eq:submultiplicativity}\mathcal{N}_{\Lambda,D}( g \circ^{\smallint_{1,0}} f ) \leqslant \mathcal{N}_{\Lambda,D}(g) \times \mathcal{N}_{\Lambda,D}(f) ,
\end{equation}
\begin{equation} \label{eq:norm and tau} \mathcal{N}_{\Lambda,D}(\tau(\lambda)(f))(\Lambda,D) = \mathcal{N}_{\Lambda,D}(f)(\lambda\Lambda,D) .
\end{equation}

These equations imply similar equations with $\mathcal{N}_{D}$ instead of $\mathcal{N}_{\Lambda,D}$ by passing to supremums.

Let us add another compatibility, which concerns the maps $\tau_{n}$ defined in \S1.1.2:

\begin{Lemma} \label{prop prs et N}(i) For all $n \in \mathbb{N}^{\ast}$, for all $f \in K \langle \langle e_{0\cup\mu_{N}}\rangle\rangle$, we have $\mathcal{N}_{\Lambda,D}(\tau_{n}(f))\leqslant\mathcal{N}_{\Lambda,D}(f)$, and in particular $\tau_{n}$ is a continuous linear map for the $\mathcal{N}_{\Lambda,D}$-topology.
\newline (ii) $K \langle \langle e_{0\cup\mu_{N}} \rangle\rangle_{<\infty}$ and $K \langle \langle e_{0\cup\mu_{N}} \rangle\rangle_{o(1)}$ are stable by $\tau_{n}$.
\newline (iii) \label{corollary} For all $f,g \in \Pi_{1,0}(K)$ and $n \in \mathbb{N}$ we have:
\begin{equation} \label{eq:submultiplicativity}
\mathcal{N}_{\Lambda,D}( \Ad_{(g \text{ }\circ^{\smallint_{1,0}}\text{ }\tau_{n}(f))}(e_{1})) \leqslant\Lambda D \mathcal{N}_{\Lambda,D}(g)  \mathcal{N}_{\Lambda,D}(f) .
\end{equation}
\end{Lemma}

\begin{proof} (i) and (ii) are clear from the definitions. (iii) follows from (\ref{eq:Norm Ad}), (\ref{eq:submultiplicativity}) and (i).
\end{proof}

\subsubsection{The weighted Ihara action and contraction mappings}

We define a notion of contraction mappings within the topological framework reviewed above. The exponent ${-1_{\smallint_{1,0}}}$ means the inverse for the Ihara product (\ref{eq:Ihara action}).

\begin{Definition} \label{def contracting}
Let $\kappa \in K^{\ast}$ with $|\kappa|_{p}<1$. We say that a map $\psi: \Pi_{1,0}(K)\rightarrow \Pi_{1,0}(K)$ is a $\kappa$-contraction (with respect to $\mathcal{N}_{\Lambda,D}$ and $\circ^{\smallint_{1,0}}$) if, for all $f_{1},f_{2} \in \Pi_{1,0}(K)$, we have:
\begin{equation} \label{eq: of contraction} \mathcal{N}_{\Lambda,D} \big( \psi(f_{2})^{-1_{\smallint_{1,0}}} \circ^{\smallint_{1,0}} \psi(f_{1}) \big) (\Lambda,D) \leqslant\mathcal{N}_{\Lambda,D} \big( f_{2}^{-1_{\smallint_{1,0}}} \circ^{\smallint_{1,0}} f_{1} \big)(\kappa \Lambda,D) .
\end{equation}
\end{Definition}

Indeed, let $\psi: \Pi_{1,0}(K) \rightarrow \Pi_{1,0}(K)$ be a contraction in the sense of definition \ref{def contracting}. Then, by the submultiplicativity of the Ihara product with respect to $\mathcal{N}_{\Lambda,D}$ (equation (\ref{eq:submultiplicativity})) and the fact that $K\langle\langle e_{0\cup\mu_{N}}\rangle\rangle$ is complete with respect to the distance defined by $\mathcal{N}_{\Lambda,D}$ (\S1.2), a standard proof tells us that $\psi$ is continuous (with respect to $\mathcal{N}_{\Lambda,D}$) and has a unique fixed point, equal to 
$\fix_{\psi} = \underset{a \rightarrow \infty}{\lim} \psi^{a} (f)$ for all $f \in \Pi_{1,0}(K)$. Thus the contractions in the sense of definition \ref{def contracting} satisfy the usual properties of contractions regarding fixed points.

\begin{Definition} \label{definition weighted}Let $(\lambda,g) \in \mathbb{G}_{m}(K) \times \Pi_{1,0}(K)$. We call weighted Ihara action by $g$ with parameter $\lambda$ and we denote by $(\lambda,g)\text{ } \circ^{\smallint_{1,0}}$ the map
$$ \Pi_{1,0}(K) \rightarrow \Pi_{1,0}(K), $$ 
$$ f \mapsto (\lambda,g) \circ^{\smallint_{1,0}} f = g \circ^{\smallint_{1,0}} \tau(\lambda)(f) . $$
\end{Definition}

We now relate the two previous definitions in the case where $\lambda$ is small:

\begin{Proposition} \label{prop weighted Ihara is contraction}(i) Let $(\lambda,g) \in \mathbb{G}_{m}(K) \times \Pi_{1,0}(K)$ such that $|\lambda|_{p}<1$.  \label{prop contractance} The map $(\lambda,g) \circ^{\smallint_{1,0}}$ is a $\lambda$-contraction. More precisely, the inequality (\ref{eq: of contraction}) is an equality if $\kappa=\lambda$. 
\newline\indent (ii) For all $(\lambda,g) \in \mathbb{G}_{m}(K) \times \Pi_{1,0}(K)$, the Ihara action of $g$ weighted by $\lambda$ is an automorphism of the scheme
$\Pi_{1,0} \times_{\Spec \mathbb{Q}} \Spec K$, whose inverse is
$$ f \mapsto \tau(\lambda^{-1}) (g^{-1_{\smallint_{1,0}}} \circ^{\smallint_{1,0}} f) . $$
\end{Proposition}

\begin{proof} (i) We have $\big((\lambda,g) \circ^{\smallint_{1,0}}(f_{2})\big) ^{-1_{\smallint_{1,0}}} \circ^{\smallint_{1,0}} \big((\lambda,g) \circ^{\smallint_{1,0}}(f_{1})\big) = 
\big( 	g \circ^{\smallint_{1,0}} \tau(\lambda)(f_{2})\big) ^{-1_{\smallint_{1,0}}}	g \circ^{\smallint_{1,0}} \tau(\lambda)(f_{1})
= \tau(\lambda)(f_{2})^{-1_{\smallint_{1,0}}} \circ^{\smallint_{1,0}}	g^{-1_{\smallint_{1,0}}} \circ^{\smallint_{1,0}} g \circ^{\smallint_{1,0}} \tau(\lambda)(f_{1}) = \tau(\lambda)(f_{2})^{-1_{\smallint_{1,0}}} \circ^{\smallint_{1,0}} \tau(\lambda)(f_{1}) = 
\tau(\lambda)\big( (f_{2})^{-1_{\smallint_{1,0}}} \circ^{\smallint_{1,0}} f_{1} \big)$.
\newline Thus 
$\mathcal{N}_{\Lambda,D}\big((\lambda,g) \circ^{\smallint_{1,0}}(f_{2})\big) ^{-1_{\smallint_{1,0}}} \circ^{\smallint_{1,0}} \big((\lambda,g) \circ^{\smallint_{1,0}}(f_{1})\big)(\Lambda,D) = \mathcal{N}_{\Lambda,D}\tau(\lambda)\big( (f_{2})^{-1_{\smallint_{1,0}}} \circ^{\smallint_{1,0}} f_{1} \big)(\Lambda,D)
= \mathcal{N}_{\Lambda,D}\big( (f_{2})^{-1_{\smallint_{1,0}}} \circ^{\smallint_{1,0}} f_{1} \big)(\lambda\Lambda,D)$ by equation (\ref{eq:norm and tau}).

(ii) Follows from the definitions.
\end{proof}

Knowing that we have a family of contractions, we consider their fixed points and their iterations.

\begin{Definition} Let $(\lambda,g) \mapsto \fix_{\lambda,g}$ be the fixed point map which sends $(\lambda,g) \in \{z \in K\text{ }|\text{ }0<|z|_{p}<1\} \times \Pi_{1,0}(K)$ to the unique fixed point of the weighted Ihara action $(\lambda,g) \circ^{\smallint_{1,0}}$.
\end{Definition}

It follows from the definitions that the map $\fix_{\lambda}: g \mapsto \fix_{\lambda,g}$ is an automorphism of the scheme $\Pi_{1,0} \times_{\Spec \mathbb{Q}} \Spec K$, whose inverse is
$\fix_{\lambda}^{-1}: f \mapsto f \text{ }\circ^{\smallint_{1,0}} \text{ } \tau(\lambda)(f)^{-1_{\smallint_{1,0}}}$.
The fixed point map is characterized by the equation:
\begin{equation} \label{eq:char fixed point} g\big( e_{0},(e_{\xi})_{\xi}\big) \fix_{\lambda,g}\big( \lambda e_{0},\lambda (g_{\xi}^{-1}e_{\xi}g_{\xi})_{\xi} \big) = \fix_{\lambda,g}\big( e_{0},(e_{\xi})_{\xi} \big) .
\end{equation}

Note that the inversion for the Ihara product on $\Pi_{1,0}(K)$ is characterized by
$$ g(e_{0},(e_{\xi})_{\xi}) \text{ . } g^{-1_{\smallint_{1,0}}} (e_{0}, (g_{\xi}^{-1}e_{\xi}g_{\xi})_{\xi}) = 1 . $$

\begin{Definition} \label{iteration of the weighted Ihara action}Let $a \in \mathbb{N}^{\ast}$. Let the map of iteration $a$ times of the Ihara action weighted by $\lambda$,
$\iter^{\smallint_{1,0}}_{a,\lambda}: 
\Pi_{1,0}(K) \rightarrow \Pi_{1,0}(K) $ 
be defined by $g \mapsto g^{a(\circ^{\smallint_{1,0}},\lambda)}$ where 
\begin{equation} \label{eq:inverse Ihara} g^{a(\circ^{\smallint_{1,0}},\lambda)} =  \underbrace{(\lambda,g) \circ^{\smallint_{1,0}} \ldots \circ^{\smallint_{1,0}} (\lambda,g)}_{a} \circ^{\smallint_{1,0}} 1 = g \circ^{\smallint_{1,0}} \tau(\lambda)(g) \circ^{\smallint_{1,0}} \ldots \circ^{\smallint_{1,0}} \tau(\lambda^{a-1})(g) 
\end{equation}
\end{Definition}

Thus $g^{a(\circ^{\smallint_{1,0}},\lambda)}$ is the unique element of $\Pi_{1,0}(K)$ such that we have, for all $f \in \Pi_{1,0}(K)$,
\newline $\underbrace{(\lambda,g) \circ^{\smallint_{1,0}} \ldots \circ^{\smallint_{1,0}} (\lambda,g)}_{a} \circ^{\smallint_{1,0}} f = (\lambda^{a},g^{a(\circ^{\smallint_{1,0}},\lambda)}) \circ^{\smallint_{1,0}} f $. The iteration map is expressed in terms of the usual de Rham multiplication on $\Pi_{1,0}(K)$ by 
\begin{equation} \label{eq: char iteration}
\iter^{\smallint_{1,0}}_{a,\lambda}(g) = 
g\big(e_{0},(e_{\xi})_{\xi\in \mu_{N}(K)}\big) \text{ } g\big(\lambda e_{0},(\lambda \Ad_{g_{\xi}}(e_{\xi}))_{\xi}\big) \text{ }\cdots\text{ } 
g\big( \lambda^{a-1}e_{0},(\lambda^{a-1}\Ad_{g_{\xi}^{a-1}}(e_{\xi}))_{\xi} \big) .
\end{equation}

\subsubsection{Application to the Frobenius}

We apply the previous paragraphs to study the iteration of Frobenius at the base-points (1,0) which we view as a map $\phi: \Pi_{1,0}^{(p)}(K) \rightarrow \Pi_{1,0}(K)$.

\begin{Lemma} \label{lemma Frobenius is p-contraction} The map  $\phi_{-\frac{\log(q)}{\log(p)}}: \Pi_{1,0}(K) \rightarrow \Pi_{1,0}(K)$ is a $\frac{1}{q}$-contraction.
\newline If $\Pi^{(p)}_{1,0}(K)$ is identified to $\Pi_{1,0}(K)$ by the isomorphism defined by $e_{0} \mapsto e_{0}$ and $e_{\xi^{(p^{\alpha})}} \mapsto e_{\xi}$ for all $\xi \in \mu_{N}(K)$, the map $\phi_{-1}: \Pi_{1,0}(K) \rightarrow \Pi^{(p)}_{1,0}(K)$ is a $\frac{1}{p}$-contraction.
\end{Lemma}

\begin{proof} This follows from the formula (\ref{eq: formula for Frob 1}), from proposition \ref{prop weighted Ihara is contraction} and from the fact that $\sigma$ is an isometry of $K$ for the $p$-adic metric.
\end{proof}

In the rest of this paper, for simplicity, we will deal mostly with the iterations of $\phi_{-\frac{\log(q)}{\log(p)}}$. This is sufficient, knowing that, for any $\alpha \in \mathbb{N}^{\ast}$, writing the Euclidean division $\alpha = r + u\frac{\log(q)}{\log(p)}$, we have $\phi_{\alpha} = \phi_{\frac{\log(q)}{\log(p)}}^{u}\text{ }\circ \text{ }\phi_{r}$.

\subsection{Prime weighted cyclotomic multiple harmonic sums}

The numbers $\har_{q^{\alpha}}\big( (n_{i})_{d};(\xi_{i})_{d+1} \big)$ will play a central role; here, we formally explain how to study them.

\subsubsection{The three frameworks of computation}

Multiple polylogarithms are the solutions to the Knizhnik-Zamolodchikov differential equation, which is the universal connection associated with $\pi_{1}^{\un,\dR}(\mathbb{P}^{1} - \{0,\mu_{N},\infty\})$ (in the sense of \cite{Deligne}, \S12), and whose crystalline Frobenius structure is $\phi$.

Their power series expansion at 0 is the following, which relates them to cyclotomic multiple harmonic sums:

\begin{equation} \label{eq:power series}\Li[(n_{i})_{d};(\xi_{i})_{d}](z) = \sum_{0<m_{1}<\ldots<m_{d}} \frac{(\frac{\xi_{2}}{\xi_{1}})^{m_{1}} \cdots (\frac{z}{\xi_{d}})^{m_{d}}}{m_{1}^{n_{1}} \cdots m_{d}^{n_{d}}} .
\end{equation}

We have proved as a \cite{I-2} (see \cite{I-2}, equation (0.3.7)) that, for all $\alpha \in \mathbb{N}^{\ast}$, prime weighted multiple harmonic sums are expressed in the following way (where $f \mapsto f^{(\xi)}$ is the natural map $\Pi_{1,0}(K) \rightarrow \Pi_{\xi,0}(K)$):

\begin{equation} \label{eq:formula of I-2}
\har_{p^{\alpha}}\big((n_{i})_{d};(\xi_{i})_{d+1}\big) = (-1)^{d} \sum_{l =0}^{\infty}\sum_{\xi} \xi^{-p^{\alpha}} \Ad_{{\Phi^{(\xi)}_{p,\alpha}}}(e_{\xi})[e_{0}^{l}e_{\xi_{d+1}}e_{0}^{n_{d}-1}e_{\xi_{d}}\ldots e_{0}^{n_{1}-1}e_{\xi_{1}}] .
\end{equation}

We have a total of three ways to deal with prime weighted multiple harmonic sums, which makes three frameworks for computations:

(i) via their expression in terms of $p$-adic cyclotomic multiple zeta values (\ref{eq:formula of I-2})

(ii) via their expression as coefficients of power series expansions of multiple polylogarithms (\ref{eq:power series}).

(iii) via their definition as elementary explicit iterated sums (\ref{eq:harmonic sum}).

We will symbolize these three frameworks by, respectively, the notations $\int_{1,0}$, $\int$ and $\Sigma$.

In \cite{I-2} we have expressed the ``harmonic Frobenius'' in the frameworks $\int$ and $\Sigma$ and we have compared the two expressions. Here we are going to express the ``iteration of the harmonic Frobenius'' in the frameworks $\int_{1,0}$ and $\Sigma$ and compare the two expressions. Keeping in mind the distinction between these three frameworks $\int_{1,0}$, $\int$ and $\Sigma$ will be essential in this paper and in subsequent ones. We note that the frameworks $\Sigma$ and $\int$ make sense for all weighted cyclotomic multiple harmonic sums whereas the framework $\int_{1,0}$ makes sense only for the prime weighted cyclotomic multiple harmonic sums and follows from a theorem.

\subsubsection{The generalization to negative numbers of iterations of the Frobenius}

The indices of $p$-adic cyclotomic multiple zeta values, of the form $\big((n_{i})_{d};(\xi_{i})_{d}\big)$, are distinct from the indices of cyclotomic weighted multiple harmonic sums (\ref{eq:harmonic sum}), of the form $\big((n_{i})_{d};(\xi_{i})_{d+1}\big)$.

\begin{Definition} \label{def harmonic word}A harmonic word over $e_{0\cup \mu_{N}}$, is a tuple $\big((n_{i})_{d};(\xi_{i})_{d+1}\big)$, with $d \in \mathbb{N}^{\ast}$ $(n_{i})_{d} \in (\mathbb{N}^{\ast})^{d}$, $(\xi_{i})_{d+1} \in \mu_{N}(K)^{d+1}$. We sometimes identify it with $e_{\xi_{d+1}} e_{0}^{n_{d}-1}e_{\xi_{d}} \ldots e_{0}^{n_{1}-1}e_{\xi_{1}}$. Let us denote by $\Wd_{\har}(e_{0\cup \mu_{N}})$ the set of harmonic words over $e_{0\cup \mu_{N}}$.
\end{Definition}

We now define, using (\ref{eq:formula of I-2}), an analogue of multiple harmonic sums associated with negative numbers of iterations of the Frobenius, and another analogue associated with the fixed point of the Frobenius. The notation ``$\har$'' that we are going to use will be justified by the subsequent papers \cite{II-1, II-2, II-3} and the notion of ``cyclotomic multiple harmonic values''. Below, $f[\frac{1}{1-e_{0}}w] = \sum_{l=0}^{\infty} f[e_{0}^{l}w]$.

\begin{Definition} \label{generalized pMHS}For any $w=e_{\xi_{d+1}}e_{0}^{n_{d}-1}e_{\xi_{d}}\ldots e_{0}^{n_{1}-1}e_{\xi_{1}}=\big( (n_{i})_{d};(\xi_{i})_{d+1}\big)$ we call generalized prime weighted multiple harmonic sums the following numbers.
\newline (i) For any $\alpha \in \mathbb{N}^{\ast}$, let $\displaystyle\har_{p,\alpha}(w) = \har_{p^{\alpha}}(w)$ and 
$\displaystyle\har_{p,-\alpha}(w) = (-1)^{d} \sum_{\xi} \xi^{-p^{\alpha}} \Ad_{\Phi^{(\xi)}_{p,-\alpha}}(e_{\xi})[\frac{1}{1-e_{0}}w]$ .
\newline (ii) For $\epsilon \in \{\pm 1\}$, let $\displaystyle \har_{q,\epsilon\infty}(w) = (-1)^{d} \sum_{\xi} \xi^{-1} \Ad_{\Phi^{(\xi)}_{q,\epsilon\infty}}(e_{\xi})[\frac{1}{1-e_{0}}w] $.
\newline (iii) If $p^{\alpha} = q^{\tilde{\alpha}}$, with $\tilde{\alpha} \in \mathbb{Z}$, we denote by $\har_{q,\tilde{\alpha}}=\har_{p,\alpha}$.
\end{Definition}

\subsubsection{The non-commutative generating series}

We now define generating series of prime weighted cyclotomic multiple harmonic sums, the numbers (\ref{eq:harmonic sum}) with $m=p^{\alpha}$.

We have defined in \cite{I-2} two variants of 
$K \langle \langle e_{0\cup\mu_{N}} \rangle \rangle^{\smallint}$ adapted to multiple harmonic sums: $K \langle \langle e_{0 \cup \mu_{N}} \rangle\rangle_{\har}^{\smallint} \subset K \langle\langle e_{0 \cup \mu_{N}} \rangle\rangle$ the vector subspace of the elements $f$ such that, for all words $w$ on $e_{0 \cup \mu_{N}}$, the sequence $(f[e_{0}^{l}w])_{l\in \mathbb{N}}$ is constant and $f[w'e_{0}]=0$ for all words $w'$; and $K \langle \langle e_{0 \cup \mu_{N}} \rangle\rangle_{\har}^{\Sigma} = K^{\Wd_{\har}(e_{0\cup\mu_{N}})}$. Here is the third variant: 

\begin{Definition} (i) Let $K \langle \langle e_{0\cup\mu_{N}} \rangle\rangle_{\har}^{\smallint_{1,0}}= \{f \in K\langle \langle e_{0\cup\mu_{N}} \rangle\rangle \text{ }|\text{ } \forall l>0,\text{ }\forall r>0,\text{ }\forall w \text{ word on } e_{0\cup \mu_{N}},\text{ }\\ f[e_{0}^{l}we_{0}^{r}] = 0\}$
\newline (ii) Let $K \langle \langle e_{0\cup\mu_{N}} \rangle\rangle_{\har,0}^{\smallint_{1,0}} = \{
f \in K \langle \langle e_{0\cup\mu_{N}} \rangle\rangle_{\har}^{\smallint_{1,0}} 
\text{ }|\text{ } \forall d \in \mathbb{N}^{\ast},\text{ }n_{i} \in \mathbb{N}^{\ast}\text{ } (1 \leqslant i \leqslant d),\text{ }\xi_{i} \in \mu_{N}(K)\text{ } (1 \leqslant i \leqslant d+1),\text{ } f[e_{\xi\xi_{d+1}}e_{0}^{n_{d}-1}e_{\xi\xi_{d}} \ldots e_{0}^{n_{1}-1}e_{\xi\xi_{1}}] =
\xi^{-1}f[e_{\xi_{d+1}}e_{0}^{n_{d}-1}e_{\xi\xi_{d}} \ldots e_{0}^{n_{1}-1}e_{\xi_{1}}] \}$.
\end{Definition}

The map $f \mapsto \sum\limits_{w \in \Wd_{\har}(e_{0\cup \mu_{N}})} f[w] \frac{1}{1-e_{0}}w$ clearly defines an isomorphism $K \langle\langle e_{0\cup\mu_{N}} \rangle \rangle^{\smallint_{1,0}}_{\har} \simlra K \langle \langle e_{0\cup\mu_{N}} \rangle \rangle^{\smallint}_{\har}$ of topological $K$-vector spaces, with topology defined by $\mathcal{N}_{D}$. However, we denote $K \langle \langle e_{0\cup\mu_{N}} \rangle\rangle_{\har}^{\smallint_{1,0}}$ and $K \langle \langle e_{0\cup\mu_{N}} \rangle\rangle_{\har}^{\smallint}$ differently in order to keep in mind the important distinction between the frameworks $\int_{1,0}$, $\int$, $\Sigma$.

Let us define the non-commutative generating series of the generalized prime weighted cyclotomic multiple harmonic sums.

\begin{Definition}
For any
$\tilde{\alpha} \in \mathbb{Z} \cup \{ \pm \infty \} - \{0\}$, let  
$\har_{q,\tilde{\alpha}} = \sum\limits_{w \in \Wd_{\har}(e_{0\cup \mu_{N}})}\har_{q,\tilde{\alpha}}(w) w \in K\langle\langle e_{0\cup \mu_{N}}\rangle\rangle_{\har}^{\smallint_{1,0}}$.
\end{Definition}

We note that, for $\tilde{\alpha} \in \mathbb{N}^{\ast}$, we have
$\har_{q,\tilde{\alpha}} \in K\langle\langle e_{0\cup \mu_{N}}\rangle\rangle_{\har,0}^{\smallint_{1,0}}$, because we have, for all $\xi \in \mu_{N}(K)$, $\xi^{-q^{\tilde{\alpha}}} = \xi^{-q} = \xi^{-1}$.

\subsection{The pro-unipotent harmonic actions and the harmonic Frobenius}

We review definitions from \cite{I-2} of the pro-unipotent harmonic actions $\circ_{\har}^{\smallint}$ and $\circ_{\har}^{\Sigma}$, and the harmonic Frobeniuses $(\tau(p^{\alpha})\phi^{\alpha})^{\smallint}_{\har}$ and $(\tau(p^{\alpha})\phi^{\alpha})^{\Sigma}_{\har}$.

Let $\Ad_{\tilde{\Pi}_{1,0}(K)}(e_{1})_{o(1)} = \Ad_{\tilde{\Pi}_{1,0}(K)}(e_{1}) \cap K \langle\langle e_{0\cup\mu_{N}} \rangle\rangle_{o(1)}$; by \cite{I-2}, proposition 1.3.5, it is a subgroup of $\Ad_{\tilde{\Pi}_{1,0}(K)}(e_{1})$ for the usual group structure of $\Spec(\mathcal{O}^{\sh,e_{0\cup \mu_{N}}})$, and for the adjoint Ihara product $\circ_{\Ad}^{\smallint_{1,0}}$; it is a complete topological group with the topology defined by $\mathcal{N}_{D}$, for both group structures.

\subsubsection{In the framework of integrals}

We review definitions from \cite{I-2} which will be useful in what follows.

Let $K \langle\langle e_{0 \cup \mu_{N}} \rangle\rangle^{\lim} \subset K \langle\langle e_{0 \cup \mu_{N}} \rangle\rangle$ be the vector subspace consisting of the elements $f\in K \langle\langle e_{0 \cup \mu_{N}} \rangle\rangle$ such that, for all words $w$ on $e_{0 \cup \mu_{N}}$, the sequence $(f[e_{0}^{l}w])_{l\in \mathbb{N}}$ has a limit in $K$, and $f[w'e_{0}]=0$ for all words $w'$. Let the map $\lim: K \langle \langle e_{0 \cup \mu_{N}} \rangle\rangle^{\lim} \rightarrow K \langle \langle e_{0 \cup \mu_{N}} \rangle\rangle_{\har}^{\smallint}$ be defined by, for all words $w$ over $e_{0\cup\mu_{N}}$, $(\lim f)[w] = \underset{l\rightarrow \infty}{\lim} f[e_{0}^{l}w].$ The $p$-adic pro-unipotent harmonic action of integrals (\cite{I-2}, definition 2.2.2) is the map 
$\circ^{\smallint}_{\har}: \Ad_{\tilde{\Pi}_{1,0}(K)_{o(1)}}(e_{1}) \times
(K \langle\langle e_{0 \cup \mu_{N}} \rangle\rangle_{\har}^{\smallint})^{\mathbb{N}}
\rightarrow 
(K \langle\langle e_{0 \cup \mu_{N}} \rangle\rangle_{\har}^{\smallint})^{\mathbb{N}}$ defined by 
$$\big( g, (h_{m})_{m\in\mathbb{N}} \big) \mapsto g \circ_{\har}^{\smallint} (h_{m})_{m\in\mathbb{N}} = \big( \lim \big( h_{m}(e_{0}, \big(\tau(m)(g^{(\xi)}\big)_{\xi}  \big)\big)_{m\in\mathbb{N}} $$

The harmonic Frobenius of integrals (\cite{I-2}, definition 2.3.5) is the map $(\tau(p^{\alpha})\phi^{\alpha})^{\smallint}_{\har}:
\big( K\langle\langle e_{0 \cup \mu_{N}} \rangle\rangle_{\har}^{\smallint} \big)^{\mathbb{N}}\rightarrow \big( K\langle\langle e_{0 \cup \mu_{N}} \rangle\rangle_{\har}^{\smallint}\big)^{\mathbb{N}}$ 
defined by 
$$ f \mapsto \Phi_{p,\alpha}^{-1}e_{1}\Phi_{p,\alpha} \circ_{\har}^{\smallint} \sigma^{\alpha}(f) . $$

\subsubsection{In the framework of series}

The $p$-adic pro-unipotent harmonic action of series (\cite{I-2}, proposition-definition 4.3.1) is a counterpart of $\circ^{\smallint}_{\har}$ found in terms of series. It is a map 
$$\circ_{\har}^{\Sigma}: K\langle \langle e_{0\cup \mu_{N}} \rangle\rangle_{\har,o(1)}^{\Sigma} \times (K\langle \langle e_{0\cup \mu_{N}} \rangle\rangle_{\har}^{\Sigma})^{\mathbb{N}} \rightarrow (K\langle \langle e_{0\cup \mu_{N}} \rangle\rangle_{\har}^{\Sigma})^{\mathbb{N}}  $$ 
where $ K\langle \langle e_{0\cup \mu_{N}} \rangle\rangle_{\har,o(1)}^{\Sigma}$ is defined in \cite{I-2}, definition 4.1.3. 

The harmonic Frobenius of series is 
$$ f \mapsto \har_{p^{\alpha}} \circ_{\har}^{\Sigma} \sigma^{\alpha}(f) . $$

We proved in \cite{I-2} that the harmonic Frobenius of integrals and the harmonic Frobenius of series are equal (\cite{I-2}, equations (0.3.3) and (0.3.5)). Thus, we see that the harmonic Frobenius is characterized by $\Phi_{p,\alpha}^{-1}e_{1}\Phi_{p,\alpha}$ or, equivalently, by $\har_{p^{\alpha}}$. This is why studying $\har_{p^{\alpha}}$ is equivalent to studying the harmonic Frobenius and, in the next sections, we will see $\har_{p^{\alpha}}$ as a function of $\alpha$, which amounts to studying the harmonic Frobenius as a function of $\alpha$.

\begin{Remark} We can define natural analogues of all the structures which are at base-points (1,0) and which are mentioned in this \S1, at base-points $(0,\xi)$ for any $\xi \in \mu_{N}(K)$, which are compatible with the morphism $(x \mapsto \xi x)_{\ast}$. All the results of this \S1 have natural analogs at base-points $(0,\xi)$ for any $\xi \in \mu_{N}(K)$. This will be used implicitly in the proofs of the next section.
\end{Remark}

\section{Iteration of the harmonic Frobenius of integrals at (1,0)}

We prove equation (\ref{eq:first of I-3}) in \S2.1, and we prove equation (\ref{eq:second of I-3}) in \S2.2.

\subsection{The fixed point equation of the harmonic Frobenius of integrals at (1,0)}

\subsubsection{The fixed point equation for the Frobenius of integrals}

We express the iterated Frobenius at base-points of (1,0) as a function of the number of iterations.

\begin{Proposition} (a) For all $\tilde{\alpha} \in \mathbb{N}^{\ast}$, we have
	$\Phi_{q,\infty} = \fix_{q^{\tilde{\alpha}_{0}},\Phi_{q,\tilde{\alpha}}}$ i.e.
	$$ \Phi_{q,-\tilde{\alpha}} \circ^{\smallint_{1,0}} \tau(q^{\tilde{\alpha}})\Phi_{q,-\infty} = \Phi_{q,-\infty}, $$
	$$ \tau(q^{\tilde{\alpha}})\Phi_{q,\infty} \circ^{\smallint_{1,0}} \Phi_{q,\tilde{\alpha}} = \Phi_{q,\infty} . $$
	(b) For the topology defined by $\mathcal{N}_{D}$ on $\Pi_{1,0}(K)$ we have:
	$$\displaystyle \Phi_{q,-\tilde{\alpha}} \underset{\tilde{\alpha} \rightarrow \infty}{\longrightarrow} \Phi_{q,-\infty}, $$ 
	$$\displaystyle \Phi_{q,\tilde{\alpha}} \underset{\tilde{\alpha} \rightarrow \infty}{\longrightarrow} \Phi_{q,\infty}. $$
	(c) For any $\tilde{\alpha} \in \mathbb{N}^{\ast}$, $f \in \Pi_{1,0}(K)$, $g \in \Pi^{(p^{\alpha})}_{1,0}(K)$,
	$$  \phi_{\frac{\log(q)}{\log(p)}}^{-\tilde{\alpha}}(f) = \sum_{n= 0}^{\infty} \Phi_{q,\infty} \circ^{\smallint_{1,0}} \bigg( \tau_{n} \big(\Phi_{q,-\infty} \circ^{\smallint_{1,0}} f \big) \bigg) \text{ }. \text{ } (q^{\tilde{\alpha}})^{n} , $$
	$$  \phi_{\frac{\log(q)}{\log(p)}}^{\tilde{\alpha}}(g) = \sum_{n = 0}^{\infty} \Phi_{q,\infty} \circ^{\smallint_{1,0}} \bigg( \tau_{n}\big(\Phi_{q,-\infty} \circ^{\smallint_{1,0}} g \big) \bigg) \text{ }. \text{ } (q^{\tilde{\alpha}})^{n} . $$
\end{Proposition} 

\begin{proof}
(a) Equations (\ref{eq: formula for Frob 2}) and (\ref{eq:Frobenius fixed point}) imply the first equation. The second equation is deduced from the first one by applying the inversion for the product $\circ^{\smallint_{1,0}}$, which commutes with $\tau$.
\newline (b) The first equation follows from lemma \ref{lemma Frobenius is p-contraction} and the discussion after proposition \ref{prop weighted Ihara is contraction}, and the second equation is deduced from the first one by applying the inversion for $\circ^{\smallint_{1,0}}$, knowing the structure of topological group of $(\Pi_{1,0}(K),\circ^{\smallint_{1,0}})$ for $\mathcal{N}_{D}$. Alternatively, the two equations follow from (i), the fact that 
$\tau(q^{\tilde{\alpha}})\Phi_{q,\infty} \underset{\tilde{\alpha} \rightarrow \infty}{\longrightarrow} 1$, 
$\tau(q^{\tilde{\alpha}})\Phi_{q,-\infty} \underset{\tilde{\alpha} \rightarrow \infty}{\longrightarrow} 1$
and that structure of topological group.
\newline (c) Follows from equations (\ref{eq: formula for Frob 1}), (\ref{eq: formula for Frob 2}), in which we replace $\Phi_{q,\tilde{\alpha}}$ and $\Phi_{q,-\tilde{\alpha}}$ by their expressions given by (i), and in which we express $\tau$ in terms of the maps $\tau_{n}$ defined in \S1.1.2 just after equation (1.2).
\end{proof}

In particular, by proposition \ref{prop fixed point equation at 1,0} (i), the existence and uniqueness of a Frobenius-invariant path, which follows from the theory of Coleman integration, is re-proved and made more precise in the very particular example of $\Pi_{1,0}(K)$. 

\subsubsection{Pro-unipotent harmonic action and harmonic Frobenius of integrals at (1,0)}

We move from discussing the Frobenius at (1,0) to discussing the harmonic Frobenius of integrals, in the framework $\int_{1,0}$ in the sense of \S1.3.1. In view of this result, we introduce new objects.

\begin{Definition} \label{definition sigma}Let $K\langle \langle e_{0\cup \mu_{N}} \rangle\rangle_{\widetilde{o(1)}} \subset K\langle \langle e_{0\cup \mu_{N}} \rangle\rangle$ be the set of elements $h$ such that all series of the type $\sum\limits_{l=0}^{\infty} h \big[e_{0}^{l} e_{\xi_{d+1}}e_{0}^{n_{d}-1}e_{\xi_{d}} \ldots e_{0}^{n_{1}-1}e_{\xi_{1}}\big]$, where $d$ and the $n_{i}$'s are positive integers and the $\xi_{i}$'s are $N$-th roots of unity, are convergent in $K$. Let the summation map be $S: K\langle\langle e_{0\cup \mu_{N}} \rangle\rangle_{\widetilde{o(1)}} \rightarrow K \langle\langle e_{0\cup\mu_{N}} \rangle\rangle_{\har,0}^{\smallint_{1,0}}$ defined by $h \mapsto \sum\limits_{\xi \in \mu_{N}(K)} \xi^{-1}\sum\limits_{\substack{d \in \mathbb{N}^{\ast} \\ \xi_{1},\ldots,\xi_{d+1} \in \mu_{N}(K) \\ n_{1},\ldots,n_{d} \in \mathbb{N}^{\ast}}} h^{(\xi)} \big[ \frac{1}{1-e_{0}}e_{\xi_{d+1}}e_{0}^{n_{d}-1}e_{\xi_{d}} \ldots e_{0}^{n_{1}-1}e_{\xi_{1}}\big]\text{ }e_{\xi_{d+1}}e_{0}^{n_{d}-1}e_{\xi_{d}} \ldots e_{0}^{n_{1}-1}e_{\xi_{1}}$.
\end{Definition}

We note that $K\langle \langle e_{0\cup \mu_{N}} \rangle\rangle_{\widetilde{o(1)}}$ contains $K\langle \langle e_{0\cup \mu_{N}} \rangle\rangle_{o(1)}$ defined in \S1.2. We now define a variant of the pro-unipotent harmonic action of integrals 
$\circ_{\har}^{\smallint}$. In the next statement, we denote by $\circ_{\Ad}^{\smallint_{1,0}}$
the extension of the adjoint Ihara product $\circ_{\Ad}^{\smallint_{1,0}}$ into a map $K\langle \langle e_{0\cup\mu_{N}}\rangle\rangle \times K\langle \langle e_{0\cup\mu_{N}}\rangle\rangle \rightarrow K\langle \langle e_{0\cup\mu_{N}}\rangle\rangle$ defined again by the formula of equation (\ref{eq:adjoint Ihara action}).

\begin{Definition} \label{def de Rham Ihara} Let the map
	$$ \circ_{\har,0}^{\smallint_{1,0}}: \Ad_{\tilde{\Pi}_{1,0}(K)_{o(1)}}(e_{1}) \times K \langle \langle e_{0\cup\mu_{N}} \rangle\rangle^{\smallint_{1,0}}_{\har,0}  \rightarrow  K \langle \langle e_{0\cup\mu_{N}}\rangle\rangle ^{\smallint_{1,0}}_{\har,0} $$
	characterized by the commutativity of the diagram:
	\begin{equation}  \label{eq:diagram}
	\begin{array}{ccccc}
	\Ad_{\tilde{\Pi}_{1,0}(K)_{o(1)}}(e_{1}) \times K\langle \langle e_{0\cup\mu_{N}}\rangle\rangle_{\widetilde{o(1)}} && \overset{\circ_{\Ad}^{\smallint_{1,0}}}{\longrightarrow}
	&&
	K\langle\langle e_{0\cup\mu_{N}}\rangle\rangle_{\widetilde{o(1)}} \\
	\downarrow{\id \times S} &&&& \downarrow{S} \\
	\Ad_{\tilde{\Pi}_{1,0}(K)_{o(1)}}(e_{1}) \times K\langle\langle e_{0\cup\mu_{N}} \rangle\rangle^{\smallint_{1,0}}_{\har} && \overset{\circ_{\har}^{\smallint_{1,0}}}{\longrightarrow} && K \langle \langle e_{0\cup\mu_{N}} \rangle\rangle^{\smallint_{1,0}}_{\har}
	\end{array} .
	\end{equation}
The pro-unipotent harmonic action of integrals at $(1,0)$ is the map
		$$ \circ_{\har}^{\smallint_{1,0}}: \Ad_{\tilde{\Pi}_{1,0}(K)_{o(1)}}(e_{1}) \times K \langle \langle e_{0\cup\mu_{N}} \rangle\rangle^{\smallint_{1,0}}_{\har} \rightarrow K \langle \langle e_{0\cup\mu_{N}}\rangle\rangle ^{\smallint_{1,0}}_{\har} $$
defined by the same formula as the one of $\circ_{\har}^{\smallint_{1,0}}$.
\end{Definition}

The basic properties of $\circ_{\har}^{\smallint_{1,0}}$ are summarized in the next proposition.

\begin{Proposition} \label{property of act 1,0 har} (i) $\circ_{\har}^{\smallint_{1,0}}$ is a well-defined group action of $(\Ad_{\tilde{\Pi}_{1,0}(K)_{o(1)}},\circ_{\Ad}^{\smallint_{1,0}})$ on $K \langle\langle e_{0\cup\mu_{N}} \rangle\rangle_{\har}^{\smallint_{1,0}}$, continuous for the topology defined by $\mathcal{N}_{D}$.

(ii) The isomorphism $K \langle\langle e_{0\cup\mu_{N}} \rangle\rangle_{\har}^{\smallint_{1,0}} \simlra K \langle\langle e_{0\cup\mu_{N}} \rangle\rangle_{\har}^{\smallint}$, $f \mapsto \sum_{w \in \Wd_{\har}(e_{0\cup\mu_{N}})} f[w] \frac{1}{1-e_{0}}w$ induces, for all $m \in \mathbb{N}^{\ast}$, a natural isomorphism of continuous group actions between the $m$-th term of $\circ_{\har}^{\smallint}$, namely $(g,h_{m}) \mapsto \lim (\tau(m)(g) \circ_{\Ad}^{\smallint_{0,0}} h_{m})$, and the action
$(g,h) \mapsto \tau(m)g \circ_{\har}^{\smallint_{1,0}} h$
\end{Proposition}

\begin{proof} (i) For any $g$ in $\Ad_{\tilde{\Pi}_{1,0}(K)}(e_{1})$, and any word $w$, the map $f \mapsto S( g \circ_{\Ad}^{\smallint_{1,0}} f)[w]$ factors in a natural way through the map $f \mapsto S f$. This can be seen by writing the formula for the dual of $\circ_{\Ad}^{\smallint_{1,0}}$. The fact that $\circ_{\Ad}^{\smallint_{1,0}}$ sends $\Ad_{\tilde{\Pi}_{1,0}(K)_{o(1)}}(e_{1}) \times K\langle \langle e_{0\cup\mu_{N}}\rangle\rangle_{\widetilde{o(1)}}$ to $K\langle \langle e_{0\cup\mu_{N}}\rangle\rangle_{\widetilde{o(1)}}$ follows from the shuffle equation for elements of $\tilde{\Pi}_{1,0}(K)_{o(1)}$ and from the formula for the dual of $\circ_{\Ad}^{\smallint_{1,0}}$. Finally, $S$ is surjective:  any $h$ in $K \langle \langle e_{0\cup\mu_{N}} \rangle\rangle^{\smallint_{1,0}}_{\har,0}$ is equal to $S \big( \sum\limits_{\substack{\xi_{1},\ldots,\xi_{d+1} \in \mu_{N}(K) \\ n_{1},\ldots,n_{d} \in \mathbb{N}^{\ast}}} h \big[ e_{\xi_{d+1}}e_{0}^{n_{d}-1}e_{\xi_{d}} \ldots e_{0}^{n_{1}-1}e_{\xi_{1}}\big]\text{ }e_{\xi_{d+1}}e_{0}^{n_{d}-1}e_{\xi_{d}} \ldots e_{0}^{n_{1}-1}e_{\xi_{1}} \big)$. This proves that $\circ_{\har,0}^{\smallint_{1,0}}$ is well-defined, and that one can write a formula for it, which is linear with respect to the second argument; thus, $\circ_{\har}^{\smallint_{1,0}}$ is well-defined.

Let $g_{1},g_{2} \in \Ad_{\tilde{\Pi}_{1,0}(K)_{o(1)}}(e_{1})$ and $f \in K \langle \langle e_{0\cup\mu_{N}} \rangle\rangle_{o(1)}$; we have $g_{2} \circ_{\Ad}^{\smallint_{1,0}} (g_{1} \circ_{\Ad}^{\smallint_{1,0}} f) = (g_{2} \circ_{\Ad}^{\smallint_{1,0}} g_{1}) \circ_{\Ad}^{\smallint_{1,0}} f$; applying the map $S$ and the commutativity of (\ref{eq:diagram}) gives: $g_{2} \circ_{\har}^{\smallint_{1,0}} (S(g_{1} \circ^{\smallint_{1,0}}_{\Ad} f)) = (g_{2} \circ^{\smallint_{1,0}}_{\Ad} g_{1}) \circ_{\har}^{\smallint_{1,0}} S(f)$,
and applying again the commutativity of (\ref{eq:diagram}) we deduce:
$g_{2} \circ_{\har}^{\smallint_{1,0}} (g_{1} \circ_{\har}^{\smallint_{1,0}} S (f)) = (g_{2} \circ_{\Ad}^{\smallint_{1,0}} g_{1})\circ_{\har}^{\smallint_{1,0}} S(f)$. This proves that $\circ_{\har}^{\smallint_{1,0}}$ is a group action.

The map $\circ_{\Ad}^{\smallint_{1,0}}$ is continuous and each sequence $(w_{l})_{l \in \mathbb{N}}$ such that $w_{l} = e_{0}^{l}e_{\xi_{d+1}}e_{0}^{n_{d}-1}e_{\xi_{d}} \ldots e_{0}^{n_{1}-1}e_{\xi_{1}}$ for all $l$ with $e_{\xi_{d+1}}e_{0}^{n_{d}-1}e_{\xi_{d}} \ldots e_{0}^{n_{1}-1}e_{\xi_{1}}$ independent of $l$ satisfies $\displaystyle \limsup_{l \rightarrow \infty} \depth(w_{l}) < +\infty$ and $\displaystyle \weight(w_{l}) \underset{l \rightarrow \infty}{\rightarrow} +\infty$; thus the map $S$ is continuous. By the commutativity of (\ref{eq:diagram}), this implies that $\circ_{\har}^{\smallint_{1,0}} \circ (\id \times S)$ is continuous. If a sequence $(h_{u})_{u\in\mathbb{N}}$ in $K\langle\langle e_{0\cup\mu_{N}} \rangle\rangle^{\smallint_{1,0}}_{\har}$ tends to $h \in K\langle\langle e_{0\cup\mu_{N}} \rangle\rangle^{\smallint_{1,0}}_{\har}$, we can find a sequence $(f_{u})_{u\in\mathbb{N}}$ and $f$ in $K\langle\langle e_{0\cup\mu_{N}}\rangle\rangle_{\widetilde{o(1)}}$ such that $h_{u} = S f_{u}$ for all $u$, $h=f$ and $f_{u} \rightarrow f$. We deduce that $\circ_{\har}^{\smallint_{1,0}}$ is continuous.

(ii) The relation between $\circ_{\har}^{\smallint_{1,0}}$ and $\circ_{\har}^{\smallint}$ follows from the definitions of these two objects and the following property of $\circ_{\Ad}^{\smallint_{1,0}}$, which is itself a consequence of the formula for the dual of $\circ_{\Ad}^{\smallint_{1,0}}$: for all $h \in \Ad_{\tilde{\Pi}_{1,0}(K)_{o(1)}}(e_{1})$ and $g \in \Ad_{\tilde{\Pi}_{1,0}(K)_{o(1)}}(e_{1})$, and for any word $w$ over $e_{0 \cup \mu_{N}}$, we have $\displaystyle(g \circ_{\Ad}^{\smallint_{1,0}} h)[\frac{1}{1-e_{0}}w] = \lim_{l \rightarrow \infty}(g \circ_{\Ad}^{\smallint_{1,0}} S h )[e_{0}^{l}w]$.
\end{proof}

In the next example, the indices $(n_{1})$ and $(n_{1},n_{2})$ are harmonic words in the sense of definition \ref{def harmonic word}.

\begin{Example} If $N=1$, the terms of depth one and two of $g \circ_{\har}^{\smallint_{1,0}} h$ are given as follows:
	$$ (g \circ_{\har}^{\smallint_{1,0}} h)(n_{1})= 
	h(n_{1}) + g[\frac{1}{1-e_{0}}e_{1}e_{0}^{n_{1}-1}e_{1}], $$
	\begin{multline*}
	(g \circ_{\har}^{\smallint_{1,0}} h)(n_{1},n_{2}) = 
	h(n_{1},n_{2}) + g[\frac{1}{1-e_{0}}e_{1}e_{0}^{n_{2}-1}e_{1}e_{0}^{n_{1}-1}e_{1}] 
	\\ + \sum_{r=0}^{n_{1}-1} g [\frac{1}{1-e_{0}} e_{1}e_{0}^{n_{2}-1}e_{1}e_{0}^{r} ]\text{ } h(n_{1}-r) + \sum_{r=0}^{n_{2}-1} g[e_{0}^{r}e_{1}e_{0}^{n_{1}-1}e_{1}] \text{ }h(n_{2}-r) .
	\end{multline*}
\end{Example}

We now deduce from definition \ref{def de Rham Ihara} the counterpart of the harmonic Frobenius of integrals in the framework $\int_{1,0}$.

\begin{Definition} \label{harmonic Frobenius of integrals}For any $\alpha \in \mathbb{N}^{\ast}$ divisible by $\frac{\log(q)}{\log(p)}$, let the harmonic Frobenius of integrals at $(1,0)$, iterated $\alpha$ times, be the map 
$(\tau(p^{\alpha})\phi_{\alpha})_{\har}^{\smallint_{1,0}}: K \langle \langle e_{0\cup\mu_{N}} \rangle\rangle^{\smallint_{1,0}}_{\har} \rightarrow K \langle \langle e_{0\cup\mu_{N}} \rangle\rangle^{\smallint_{1,0}}_{\har}$ defined by $f \mapsto \Ad_{\Phi_{p,\alpha}}(e_{1}) \circ_{\har}^{\smallint_{1,0}} f$.
\end{Definition}

\begin{Corollary} The isomorphism $K \langle\langle e_{0\cup\mu_{N}} \rangle\rangle_{\har}^{\smallint_{1,0}} \simlra K \langle\langle e_{0\cup\mu_{N}} \rangle\rangle_{\har}^{\smallint}$, $f \mapsto \sum_{w \in \Wd_{\har}(e_{0\cup\mu_{N}})} f[w] \frac{1}{1-e_{0}}w$ induces, for all $m \in \mathbb{N}$, an isomorphism of continuous maps between $(\tau(p^{\alpha})\phi_{\alpha})_{\har}^{\smallint_{1,0}}$ and the the $m=1$ term of $(\tau(p^{\alpha})\phi_{\alpha})_{\har}$.
\end{Corollary}

\begin{proof} Direct consequence of proposition \ref{property of act 1,0 har} and the definitions of the two harmonic Frobeniuses.
\end{proof}

\subsubsection{Fixed point equation for the harmonic Frobenius}

We can now deduce from \S2.1.1, via \S2.1.2, an expression of the iterated harmonic Frobenius of integrals at (1,0) as a function of its number of iterations, whose coefficients are expressed in terms of the fixed point of the Frobenius at (1,0): this is equation (\ref{eq:first of I-3}).

By proposition \ref{prop fixed point equation at 1,0}, we have 
$\tau(q^{\tilde{\alpha}}) \Phi_{q,\infty} \circ^{\smallint_{1,0}} \Phi_{q,\tilde{\alpha}} = \Phi_{q,\infty}$.
By definition, the inverse of $\Phi_{q,\infty}$ for $\circ^{\smallint_{1,0}}$ is $\Phi_{q,-\infty}$ thus the inverse of 
$\tau(q^{\tilde{\alpha}})(\Phi_{q,\infty})$ is 
$\tau(q^{\tilde{\alpha}})(\Phi_{q,-\infty})$. Whence $\Phi_{q,\tilde{\alpha}} = \tau(q^{\tilde{\alpha}}) \Phi_{q,-\infty}
\circ^{\smallint_{1,0}} \Phi_{q,\infty}$. 
By applying $\Ad(e_{1})$, we deduce $\Ad_{\Phi_{q,\tilde{\alpha}}}(e_{1}) = \tau(q^{\tilde{\alpha}}) \Ad_{\Phi_{q,-\infty}}(e_{1})
\circ_{\Ad}^{\smallint_{1,0}} \Ad_{\Phi_{q,\infty}}(e_{1})$, whence 
$S\Ad_{\Phi_{q,\tilde{\alpha}}}(e_{1}) = S(\tau(q^{\tilde{\alpha}}) \Ad_{\Phi_{q,-\infty}}(e_{1})
\circ_{\Ad}^{\smallint_{1,0}} \Ad_{\Phi_{q,\infty}}(e_{1}))$.
By the commutative diagram in definition \ref{def de Rham Ihara}, this amounts to 
$S\Ad_{\Phi_{q,\tilde{\alpha}}}(e_{1}) = \tau(q^{\tilde{\alpha}}) \Ad_{\Phi_{q,-\infty}}(e_{1})
\circ_{\har}^{\smallint_{1,0}} S\Ad_{\Phi_{q,\infty}}(e_{1}))$.
By equation (\ref{eq:formula of I-2}), we have 
$S\Ad_{\Phi_{q,\tilde{\alpha}}}(e_{1}) =\har_{q,\tilde{\alpha}}$ and 
$S\Ad_{\Phi_{q,\infty}}(e_{1}))=\har_{q,\infty}$. Whence equation (\ref{eq:first of I-3}).

\begin{Remark} The power series expansion of any $\har_{q^{\tilde{\alpha}}}\big((n_{i})_{d};(\xi_{i})_{d+1}\big)$ in terms of $q^{\tilde{\alpha}}$, given by equation (\ref{eq:first of I-3}) have coefficients of degrees in $\{1,\ldots,\displaystyle \min_{1\leqslant i \leqslant d} n_{i}-1\}$ equal to $0$. This follows from $\Phi^{(\xi)}_{q,-\infty}[e_{0}] = 0$ which implies $\Ad_{\Phi_{q,-\infty}^{(\xi)}}(e_{\xi}) = e_{\xi} + \text{ terms of depth }\geqslant 2$, for all $\xi \in \mu_{N}(K)$.
\end{Remark}

\begin{Example} In depth one and two and for $\mathbb{P}^{1} - \{0,1,\infty\}$, we have:
	$$ \har_{q^{\tilde{\alpha}}}(n_{1}) = 
	\har_{q^{\infty}}(n_{1}) + 
	\sum_{n = n_{1}}^{\infty} (q^{\tilde{\alpha}})^{n} \Ad_{\Phi_{q,-\infty}}(e_{1})[e_{0}^{n-n_{1}}e_{1}e_{0}^{n_{1}-1}e_{1}] . $$
\begin{multline*} \har_{q^{\tilde{\alpha}}}(n_{1},n_{2})  =  
\har_{q^{\infty}}(n_{1},n_{2}) + \sum_{n = n_{1}+n_{2}}^{\infty} (q^{\tilde{\alpha}})^{n} \Ad_{\Phi_{q,-\infty}}(e_{1})[e_{0}^{n-n_{1}-n_{2}}e_{1}e_{0}^{n_{2}-1}e_{1}e_{0}^{n_{1}-1}e_{1}] 
\\ + \sum_{r_{1}=0}^{n_{1}-1} \sum_{n = n_{2}+r_{1}}^{\infty} (q^{\tilde{\alpha}})^{n} \Ad_{\Phi_{q,-\infty}}(e_{1})[e_{0}^{n-n_{2}-r_{1}}e_{1}e_{0}^{n_{2}-1}e_{1}e_{0}^{r_{1}}] \text{ } \har_{q^{\infty}}(n_{1}-r_{1}) 
\\ + \sum_{r_{2}=0}^{n_{2}-1} (q^{\tilde{\alpha}})^{n_{1}+r_{2}} \Ad_{\Phi_{q,-\infty}}(e_{1})[e_{0}^{r}e_{1}e_{0}^{n_{1}-1}e_{1}] \text{ } \har_{q^{\infty}}(n_{2}-r_{2}) .
\end{multline*}
\end{Example}

\subsection{Iteration equation of the harmonic Frobenius of integrals at (1,0)}

We are now going to re-express the iterated harmonic Frobenius of integrals at (1,0) as a function of the number of iterations, in a different way, without involving the fixed point. 

\subsubsection{Iteration on the Frobenius at (1,0)}

As in \S2.1.2, the first step is to describe the iterated Frobenius at base-points (1,0) as a function of its number of iterations, this time without involving the fixed point but, instead, the map of iteration of the weighted Ihara action (Definition \ref{iteration of the weighted Ihara action}).

\begin{Proposition} \label{prop fixed point equation at 1,0} For all $\tilde{\alpha}_{0},\tilde{\alpha} \in \mathbb{N}^{\ast}$ such that $\tilde{\alpha}_{0}|\tilde{\alpha}$, we have
$$ \Phi_{q,\alpha} = \iter_{\frac{\tilde{\alpha}}{\tilde{\alpha_{0}}},q^{\tilde{\alpha_{0}}}}^{\smallint_{1,0}} \big( \Phi_{q,\tilde{\alpha_{0}}} \big), $$
$$ \Phi_{q,\tilde{\alpha}} = 
{\Phi_{q,-\tilde{\alpha}}}^{-1_{\smallint_{1,0}}}. $$
More generally, for all $\alpha \in \mathbb{N}^{\ast}$, we have $\Phi_{p,-\alpha} = \Phi_{p,-1} \circ^{\smallint_{1,0}} \tau(\lambda)\sigma(\Phi_{p,-1}) \cdots \circ^{\smallint_{1,0}} \tau(\lambda^{\alpha-1})\sigma^{\alpha-1}(\Phi_{p,-1})$ and 
$\Phi_{p,-\alpha} = \Phi_{p,\alpha}^{-1_{\smallint_{1,0}}}$.
\end{Proposition}

\begin{proof} This follows from (\ref{eq: formula for Frob 1}) and (\ref{eq: formula for Frob 2}), and definition \ref{definition weighted}.
\end{proof}

This generalizes a statement appearing in \cite{Furusho 2}, proof of proposition 3.1.

Before continuing on the study of the iterated Frobenius, we remark that, by proposition 2.1.1 and proposition 2.1.2, considering the coefficients of these non-commutative formal power series series, one has equations relating the $p$-adic cyclotomic multiple zeta values $\zeta_{p,\alpha}(w)$ and $\zeta_{p,\alpha'}(w')$, for any $\alpha,\alpha' \in \mathbb{Z} \cup \{\pm \infty\} - \{0\}$, as follows:

\begin{Corollary} \label{comparison of vector spaces} Let $\alpha \in \mathbb{Z}\cup \{\pm \infty\} - \{0\}$. For any $n \in \mathbb{N}^{\ast}$, let  $\mathcal{Z}_{p,\alpha,n}$ be the $\mathbb{Q}$-vector space generated by the numbers $\zeta_{p,\alpha}(w)$ with $w$ a word of weight $n$.
	\newline (i) For $\alpha \in \mathbb{N}^{\ast}$,
	$\mathcal{Z}_{p,\alpha,n}= \sigma^{-\alpha}(\mathcal{Z}_{p,1,n})$, and $\mathcal{Z}_{p,-\alpha,n}= \sigma^{\alpha}(\mathcal{Z}_{p,-1,n})$.
	\newline If $\alpha\in \mathbb{Z} - \{0\}$ is such that $p^{|\alpha|}$ is a power of $q$ then 
	$\mathcal{Z}_{p,\alpha,n}= \mathcal{Z}_{q,\infty,n} = \mathcal{Z}_{q,-\infty,n}$.
	\newline In particular, the dimension of $\mathcal{Z}_{p,\alpha,n}$ is independent of $\alpha \in \mathbb{Z} \cup \{\pm \infty\} - \{0\}$.
	\newline (ii) $\mathcal{N}_{\Lambda,D} (\Phi_{p,\alpha})$ is independent of $\alpha \in \mathbb{Z} \cup \{\pm \infty\} - \{0\}$.
\end{Corollary}

\begin{proof}
	For any positive integers $n,d$, let 
	$\mathcal{O}_{n,d}^{\sh,e_{0\cup\mu_{N}}}\subset \mathcal{O}^{\sh,e_{0\cup\mu_{N}}}$ be the subspace generated by words of weight $n$ and depth $d$. Let
	$\mathcal{O}_{n,\leqslant d}^{2}=\sum\limits_{\substack{r \geqslant 2
			\\ n_{1}+\ldots+n_{r} = n
			\\ d_{1}+\ldots+d_{r} \leqslant d}} 
	\mathcal{O}_{n_{1},d_{1}}^{\sh,e_{0\cup\mu_{N}}} \text{ }\sh \ldots \sh\text{ } \mathcal{O}_{n_{r},d_{r}}^{\sh,e_{0\cup\mu_{N}}}$. Let $\lambda \in K$ such that $|\lambda|_{p}<1$ and $a \in \mathbb{N}^{\ast}$. For each $w$ word on $e_{0\cup\mu_{N}}$, of weight $n$ and depth $d$, by equations (\ref{eq:char fixed point}), (\ref{eq:inverse Ihara}), (\ref{eq: char iteration}) we have the following congruences, where the duals refer to the duality between $\mathcal{O}^{\sh,e_{0\cup \mu_{N}}}$ and the points of the corresponding group scheme:
	$$ (1 - \lambda^{n})\fix_{\lambda}^{\vee}(w) \equiv w \mod \mathcal{O}_{n, \leqslant d}^{2} , $$
	$$ {\iter^{\smallint_{1,0}}_{a,\lambda}}^{\vee}(w) \equiv \frac{1 - \lambda^{an}}{1 -\lambda^{n}} w \mod \mathcal{O}_{n,\leqslant  d}^{2} , $$
	$$ (-1_{\smallint_{1,0}})^{\vee}(w) \equiv - w \mod \mathcal{O}_{n, \leqslant d}^{2} . $$
By induction on $(n,d)$, this implies the following equalities, where ${}_{\mathbb{Z}(\lambda)} \mathcal{O}^{\sh,e_{0\cup \mu_{N}}}_{n,\leqslant d}$ is the $\mathbb{Z}(\lambda)$-module generated by words of weight $n$ and of depth $\leqslant d$:
$$\fix_{\lambda}^{\vee}
\big(  {}_{\mathbb{Z}(\lambda)} \mathcal{O}^{\sh,e_{0\cup \mu_{N}}}_{n,\leqslant d} \big) =  \iter^{\vee}_{a_{\circ^{\smallint_{1,0}},\lambda}} \big( {}_{\mathbb{Z}(\lambda)} \mathcal{O}^{\sh,e_{0\cup \mu_{N}}}_{n,\leqslant d} \big) = 
{(-1_{\smallint_{1,0}})}^{\vee} \big( {}_{\mathbb{Z}(\lambda)} \mathcal{O}^{\sh,e_{0\cup \mu_{N}}}_{n,\leqslant d} \big) =  {}_{\mathbb{Z}(\lambda)} \mathcal{O}^{\sh,e_{0\cup \mu_{N}}}_{n,\leqslant d} , $$
and, that, for all $g \in \Pi_{\xi,0}(K)$, we have
$$ \mathcal{N}_{\Lambda,D}( \fix_{\lambda,g}) 
= \mathcal{N}_{\Lambda,D} (g^{a_{\circ^{\smallint_{1,0}},\lambda}})
= \mathcal{N}_{\Lambda,D}(g^{-1_{\smallint_{1,0}}})
= \mathcal{N}_{\Lambda,D}(g) .
$$
This implies the result via proposition \ref{prop fixed point equation at 1,0} (i) (a) and (ii).
\end{proof}

A particular case of corollary \ref{comparison of vector spaces} (i) can be found in \cite{Yamashita}, proposition 3.10, and an explicit example is given in \cite{Furusho 2}, example 2.10.

\begin{Remark} The proof of Corollary \ref{comparison of vector spaces} indicates that the equations of Corollary \ref{comparison of vector spaces} are compatible with the depth filtration and the bounds on valuations of $p$MZV$\mu_{N}$'s, which are two parameters in our computation. This can be viewed as a prerequisite for the next paragraphs in which we show a kind of compatibility between the iteration of the Frobenius and our computation of $p$MZV$\mu_{N}$'s.
\end{Remark}

\subsubsection{The iteration of the harmonic Frobenius}

As in \S2.1.2, we move from discussing the Frobenius at (1,0), in \S2.1.2, to discussing the harmonic Frobenius. We first describe how the map $\iter^{\smallint_{1,0}}_{a,\lambda}$ depends on its parameters $\lambda$ and $a$.

\begin{Definition} \label{definition lautre tau} For any $d \in \mathbb{N}^{\ast}$, let $\tau_{\ast,\leqslant d}: K \langle\langle e_{0\cup \mu_{N}} \rangle\rangle \rightarrow  K \langle\langle e_{0\cup \mu_{N}} \rangle\rangle$ be the map which sends $f = \sum\limits_{w\text{ word}} f[w]w$ to $\sum\limits_{\substack{w\text{ word} \\ \depth(w)\leqslant d}} f[w]w$.
\end{Definition}

\begin{Proposition} \label{definition of el weighted}Let $\Lambda$, $\tilde{\Lambda}$, $\textbf{a}$ be three formal variables. There exists a map
	\begin{equation} \iter^{\smallint_{1,0}}(\Lambda,\tilde{\Lambda}, \textbf{a}): \Ad_{\tilde{\Pi}_{1,0}(K)}(e_{1}) \rightarrow K \langle\langle e_{0\cup\mu_{N}} \rangle\rangle [\tilde{\Lambda},\textbf{a}](\Lambda) ,
	\end{equation}
	such that, for any $f \in  \Ad_{\tilde{\Pi}_{1,0}(K)}(e_{1})$, word $w$, $a \in \mathbb{N}^{\ast}$ such that $a > \depth(w)$ and $\lambda \in K - \{0\}$ which is not a root of unity, we have:
	$$ \iter^{\smallint_{1,0}}_{a,\lambda}(f)[w] = \iter^{\smallint_{1,0}}(\lambda,\lambda^{a},a)(f)[w] . $$
\end{Proposition}

\begin{proof} With the assumptions of the statement, let $d=\depth(w)$; knowing that $g[\emptyset] = 1$, dualizing the multiplication of the $a$ factors in equation (\ref{eq: char iteration}) gives
	\begin{multline} \label{eq:trois un deux}
	\iter^{\smallint_{1,0}}_{a,\lambda} (g) [w] =
	\\ \sum_{0 \leqslant d' \leqslant d} \text{ }\sum_{0\leqslant i_{1} < \ldots <i_{d'} \leqslant a-1} \text{ }\sum_{\substack{w_{i_{1}} \not= \emptyset ,\ldots, w_{i_{d'}} \not= \emptyset \\ w_{i_{1}}\ldots w_{i_{d'}} = w}}	
	g \big(\lambda^{i_{1}-1}e_{0},(\lambda^{i_{1}-1}
	\Ad_{{g^{(\xi)}}^{i_{1}-1}}(e_{\xi}))_{\xi}\big)[w_{i_{1}}] 
	\times \ldots 
	\times 
	\\ g \big(\lambda^{i_{d'}-1}e_{0},(\lambda^{i_{1}-1}
	\Ad_{{g^{(\xi)}}^{i_{d'}-1}}(e_{\xi}))_{\xi}\big)[w_{i_{d'}}] .
	\end{multline}
	\indent We have assumed that $a>d$; let us thus separate the indices $i_{j}\leqslant d$ and $i_{j}> d$: 
	$\displaystyle \sum_{0 \leqslant d' \leqslant d} \text{ } \sum_{0\leqslant i_{1} < \ldots <i_{d'} \leqslant a-1}  = \sum_{0 \leqslant d'' \leqslant d' \leqslant d} \text{ } \sum_{0\leqslant i_{1} < \ldots <i_{d''}\leqslant d} \text{ }\sum_{d < i_{d''+1}<\ldots < i_{d'} \leqslant a-1}$. This yields an expression of (\ref{eq:trois un deux}), as a $K$-linear combination indexed by $\{(d'',d') \text{ | } 0 \leqslant d'' \leqslant d' \leqslant d\}\times \{\text{deconcatenations of }w\text{ in }d' \text{ non-empty subwords}\}$ which is independent of $a$ but depends polynomially of $\lambda$, and with coefficients as well independent of $a$ and polynomial functions of $\lambda$, of the numbers
	\begin{multline} \label{eq: trois un trois}
	\sum_{d< i_{d''} < \ldots <i_{d'} \leqslant a-1}  g\big(\lambda^{i_{1}-1}e_{0},(\lambda^{i_{d''}-1}
	\Ad_{{g^{(\xi)}}^{i_{d''}-1}}(e_{\xi}))_{\xi}\big)[w_{i_{d''}}] 
	\times \ldots \times 
	\\
	g(\lambda^{i_{d'}-1}e_{0},(\lambda^{i_{d''}-1}
	\Ad_{{g^{(\xi)}}^{i_{d''}-1}}(e_{\xi}))_{\xi})[w_{i_{d'}}] .
	\end{multline}
	\indent Let $\epsilon^{(\xi)} = g^{(\xi)} - 1$, and $\tilde{\epsilon}^{(\xi)} = {g^{(\xi)}}^{-1} - 1$. For $0 \leqslant i_{j} \leqslant a-1$, we have (where $\tau_{\ast,\leqslant d}$ is defined in definition \ref{definition lautre tau}): $\displaystyle  \tau_{\ast,\leqslant d}(\Ad_{{g^{(\xi)}}^{i_{j} -1}}(e_{\xi})) = \sum_{\substack{m_{j},m'_{j}\in \mathbb{N} \\ m_{j} \leqslant i_{j},\text{ } m'_{j} \leqslant i_{j},\text{ } m_{j}+m'_{j}+1 \leqslant d}}
	{i_{j} \choose m_{j}}{i_{j} \choose m'_{j}}
	\tilde{\epsilon}_{\xi}^{m'_{j}} e_{\xi} \epsilon_{\xi}^{m_{j}}$. When $i_{j} > d$, the collection of conditions $\{m_{j},m'_{j}\in \mathbb{N},\text{ } m_{j} \leqslant i_{j}, \text{ }m'_{j} \leqslant i_{j},\text{ }m_{j}+m'_{j}+1 \leqslant d\}$ is equivalent to 
	$\{m_{j},m'_{j}\in \mathbb{N},\text{ }m_{j}+m'_{j}+1 \leqslant d\}$; thus, dualizing in (\ref{eq: trois un trois}) each factor
	$g(\lambda^{i_{j}-1}e_{0},(\lambda^{i_{j}-1}\Ad_{{g^{(\xi)}}^{i_{j}-1}}(e_{\xi}))_{\xi})$ tells us that (\ref{eq: trois un trois}) is a linear combination, independent of $a$ and $\lambda$, of sums
	$\displaystyle \sum_{d< i_{d''} < \ldots <i_{d'} \leqslant a-1} 
	\prod_{j=d''}^{d'} {i_{j} \choose m_{j}}{i_{j} \choose m'_{j}} \lambda^{i_{j} \weight_{j}}$, where $\weight_{j} \in \mathbb{N}^{\ast}$ arises as the weight of a certain quotient sequence of $w_{i_{j}}$, and ${i_{j} \choose m_{j}}{i_{j} \choose m'_{j}}$ are polynomials of $i_{j}$.
	\newline\indent Finally, any function of $a$ of the form
	$\displaystyle \sum_{L \leqslant  I_{1} < \ldots < I_{\delta} \leqslant a-1} \prod_{j=1}^{\delta} P_{j}(I_{j}) \lambda^{I_{j} C_{j}}$ with $L,\delta \in \mathbb{N}^{\ast}$, $C_{1},\ldots,C_{\delta} \in \mathbb{N}^{\ast}$ and $P_{1},\ldots,P_{\delta} \in K[T]$ polynomials, depends on $a$ as a polynomial function of $(a,\lambda^{a})$: one can reduce this statement to $L=0$ by splitting an iterated sum over $0 \leqslant I_{1}<\ldots < I_{\delta} \leqslant a-1$ at $L$ and by induction on $\delta$, then use, again by induction on $\delta$ that, for all $\deg_{j} \in \mathbb{N}^{\ast}$, we have $\sum\limits_{I_{j}=0}^{\deg_{j}} I_{j}^{\delta'} \lambda^{C_{j}I_{j}} = 
	\big( \lambda^{C_{j}} \frac{d}{d(\Lambda^{C_{j}})} \big)^{l} \big( \frac{\lambda^{C_{j}\deg_{j}} - 1}{\Lambda^{C_{j}}-1} \big)$.
\end{proof}

Let us now define the map of iteration of the harmonic Frobenius of integrals at (1,0), by using the above iteration map and the summation map $S$ of definition \ref{definition sigma}:

\begin{Definition}\label{la definition a citer dans le theoreme}
Let $\iter_{\har}^{\smallint_{1,0}}(a,\lambda) = S \circ \iter^{\smallint_{1,0}}(\lambda,\lambda^{a},a):  \Ad_{\tilde{\Pi}_{1,0}(K)}(e_{1}) \rightarrow K \langle\langle e_{0\cup \mu_{N}} \rangle\rangle_{\har}^{\smallint_{1,0}}$.
\end{Definition}

We can now deduce from \S2.2.1 and the previous proposition a second description of the iterated harmonic Frobenius as a function of its number of iterations: this is equation (\ref{eq:second of I-3}). 

By proposition \ref{prop fixed point equation at 1,0}, we have 
$\Phi_{q,\tilde{\alpha}} = \iter^{\smallint_{1,0}}_{\frac{\tilde{\alpha}}{\tilde{\alpha}_{0}},q^^{\tilde{\alpha}_{0}}(\Phi_{q,\tilde{\alpha}_{0}})$, thus by proposition \ref{definition of el weighted}, we have 
$\Phi_{q,\tilde{\alpha}} = \iter^{\smallint_{1,0}}(q^{\tilde{\alpha}_{0}}, q^{\tilde{\alpha}},\frac{\tilde{\alpha}}{\tilde{\alpha}_{0}})(\Phi_{q,\tilde{\alpha}_{0}})$, whence 
$S\Phi_{q,\tilde{\alpha}} = S\iter^{\smallint_{1,0}}(q^{\tilde{\alpha}_{0}}, q^{\tilde{\alpha}},\frac{\tilde{\alpha}}{\tilde{\alpha}_{0}})(\Phi_{q,\tilde{\alpha}_{0}})$. By definition \ref{la definition a citer dans le theoreme}, this amounts to $S\Phi_{q,\tilde{\alpha}} =\iter_{\har}^{\smallint_{1,0}} (\frac{\tilde{\alpha}}{\tilde{\alpha}_{0}},q^{\tilde{\alpha}})(\Phi_{q,\tilde{\alpha}_{0}})$. Finally, by equation (\ref{eq:formula of I-2}), we have $S\Phi_{q,\tilde{\alpha}} = \har_{q,\tilde{\alpha}}$.
Whence equation (\ref{eq:second of I-3}).

\section{Iteration of the harmonic Frobenius of series}

In this section, we prove equation (\ref{eq:third of I-3}) and we discuss its meaning.

\subsection{Prime weighted multiple harmonic sums as functions of the number of iterations of the Frobenius}

In this section we study how $\har_{q^{\tilde{\alpha}}}\big((n_{i})_{d};(\xi_{i})_{d+1}\big)$ depends on $\tilde{\alpha}$.

The first step is to write a $p$-adic expression of $\har_{q^{\tilde{\alpha}}}\big((n_{i})_{d};(\xi_{i})_{d+1}\big)$, obtained by considering the  $q^{\tilde{\alpha}_{0}}$-adic expansion of the indices $m_{1},\ldots,m_{d}$ of the domain of summation of $\har_{q^{\tilde{\alpha}}}\big((n_{i})_{d};(\xi_{i})_{d+1}\big)$ as in equation (\ref{eq:harmonic sum}).

\begin{Lemma} \label{proposition starting}We have, for any $d\in \mathbb{N}^{\ast}$, positive integers $n_{i}$ ($1 \leqslant i \leqslant d$), and $N$-th roots of unity $\xi_{i}$ ($1 \leqslant i \leqslant d+1$),
\begin{multline} \label{eq:511}\har_{q^{\tilde{\alpha}}}\big((n_{i})_{d};(\xi_{i})_{d+1}\big) = 
\\
\sum_{\substack{(l_{i})_{d} \in \mathbb{N}^{d}
\\ (v_{i})_{d},(u_{i})_{d},(r_{i})_{d} \in \mathbb{N}^{d} \times \mathbb{N}^{d}
\times \{1,\ldots,q^{\tilde{\alpha}_{0}}-1\}
\\ u_{i} \leqslant q^{\tilde{\alpha}_{0}(v_{i+1}-v_{i}-1)} (q^{\tilde{\alpha}_{0}} u_{i+1} + r_{i+1})-1  \text{ }\text{ if } v_{i}<v_{i+1}
\\ u_{i}\leqslant u_{i+1} \text{ }
\text{ if } v_{i}=v_{i+1} \text{ and } r_{i}<r_{i+1}
\\ u_{i} \leqslant u_{i+1} - 1 \text{ }\text{ if } v_{i}= v_{i+1} \text{ and } r_{i}\geqslant r_{i+1}
\\ q^{\tilde{\alpha}_{0}(v_{i}-v_{i+1}-1)}(q^{\tilde{\alpha}_{0}} u_{i}+r_{i}) \leqslant u_{i+1} \text{ } \text{ if } v_{i}>v_{i+1} }}
(q^{\tilde{\alpha}})^{\sum_{i=1}^{d}n_{i}} \bigg(\frac{1}{\xi_{d+1}}\bigg)^{q^{\tilde{\alpha}}}
\bigg( \prod_{j=1}^{d} \bigg(\frac{\xi_{j+1}}{\xi_{j}} \bigg)^{u_{j}+r_{j}} {-n_{j} \choose l_{j}}  \frac{(q^{\tilde{\alpha}_{0}}u_{j})^{l_{j}}}{r_{j}^{l_{j}+n_{j}}} \bigg) .
\end{multline}
\end{Lemma}

\begin{proof} Let $(m_{1},\ldots,m_{d}) \in \mathbb{N}^{d}$ such that $0<m_{1}<\ldots < m_{d}<q^{\tilde{\alpha}}$. There is a unique way to write $m_{i} = (q^{\tilde{\alpha}_{0}})^{v_{i}} (q^{\tilde{\alpha}_{0}}u_{i}+r_{i})$ \noindent with $v_{i} \in \mathbb{N}$, $u_{i} \in \mathbb{N}$, $r_{i}\in \{1,\ldots,q^{\tilde{\alpha}_{0}}-1\}$: for each $(m_{1},\ldots,m_{d})$ and each $i \in \{1,\ldots,d\}$, $v_{i}$ is the $q^{\tilde{\alpha}_{0}}$-adic valuation of $m_{i}$, and $u_{i}$ and $r_{i}$ are, respectively, the quotient and the remainder of the Euclidean division of $m_{i}q^{-\tilde{\alpha}_{0}v_{i}}$ by $q^{\tilde{\alpha}_{0}}$.
\newline\indent We have, for any $\xi \in \mu_{N}(K)$, $\xi^{(q^{\tilde{\alpha}_{0}})^{v_{i}} (q^{\tilde{\alpha}_{0}}u_{i}+r_{i})} = \xi^{(q^{\tilde{\alpha}_{0}}u_{i}+r_{i})} = \xi^{u_{i}+r_{i}}$, and we write $(q^{\tilde{\alpha}_{0}}u_{i}+r_{i})^{-n_{i}} = 
r_{i}^{-n_{i}} (\frac{q^{\tilde{\alpha}_{0}}u_{i}}{r_{i}}+1)^{-n_{i}} = \sum\limits_{l_{i} \in \mathbb{N}} {-n_{i} \choose l_{i}} (q^{\tilde{\alpha}_{0}}u_{i})^{l_{i}} r_{i}^{-n_{i}-l_{i}}$ for each $i$. This gives 
$$ \bigg(\frac{1}{\xi_{d+1}}\bigg)^{q^{\tilde{\alpha}}}  \displaystyle\prod_{i=1}^{d} (\frac{\xi_{i+1}}{\xi_{i}})^{m_{i}} m_{i}^{-n_{i}} = \bigg(\frac{1}{\xi_{d+1}}\bigg)^{q^{\tilde{\alpha}}}   \prod_{j=1}^{d} \bigg( \sum_{l_{j}\in \mathbb{N}} \bigg(\frac{\xi_{j+1}}{\xi_{j}} \bigg)^{u_{j}+r_{j}} {-n_{j} \choose l_{j}}  \frac{(q^{\tilde{\alpha}_{0}}u_{j})^{l_{j}}}{r_{j}^{l_{j}+n_{j}}} \bigg)$$
\indent Let $(v_{i})_{d} \in \{0,\ldots,\frac{\tilde{\alpha}}{\tilde{\alpha}_{0}}-1\}^{d}$,
$(u_{i})_{d} \in \{0,\ldots,q^{\tilde{\alpha}-\tilde{\alpha}_{0}}-1\}^{d}$, 
$(r_{i})_{d} \in  \{1,\ldots,q^{\tilde{\alpha}_{0}}-1\}^{d}$ 
such that, for all $i \in \{1,\ldots,d\}$ we have
$0<q^{\tilde{\alpha}_{0}v_{i}}(q^{\tilde{\alpha}_{0}}u_{i}+r_{i}) < q^{\tilde{\alpha}}$. Then, for all $i \in \{1, \ldots,d-1\}$,
\begin{equation} \label{eq:change of domain of summation} q^{\tilde{\alpha}_{0}v_{i}}(q^{\tilde{\alpha}_{0}}u_{i}+r_{i})  < q^{\tilde{\alpha}_{0}v_{i+1}}(q^{\tilde{\alpha}_{0}}u_{i+1}+r_{i+1})
\Leftrightarrow \left\{
\begin{array}{ll} u_{i} \leqslant q^{\tilde{\alpha}_{0}(v_{i+1}-v_{i}-1)} (q^{\tilde{\alpha}_{0}} u_{i+1} + r_{i+1})-1  &\text{if } v_{i}<v_{i+1}
\\ u_{i}\leqslant u_{i+1} 
& \text{if } v_{i}=v_{i+1}, r_{i}<r_{i+1}
\\ u_{i} \leqslant u_{i+1} - 1  &\text{if } v_{i}= v_{i+1},  r_{i}\geqslant r_{i+1}
\\ q^{\tilde{\alpha}_{0}(v_{i}-v_{i+1}-1)}(q^{\tilde{\alpha}_{0}} u_{i}+r_{i}) \leqslant u_{i+1} & \text{if } v_{i}>v_{i+1}
\end{array} \right. . 
\end{equation}
\end{proof}

In the expression of $\har_{q^{\tilde{\alpha}}}\big((n_{i})_{d};(\xi_{i})_{d+1}\big)$ of the lemma \ref{proposition starting}, we are going to sum over all the possible values of the parameters $u_{i}$ and $r_{i}$, in order to have an expression which depends only on the $v_{i}$'s. In view of that, let the numbers $\mathcal{B}_{m}^{l}(\xi) \in K$, for $l,m \in \mathbb{N}$ such that $0 \leqslant m \leqslant l+1$ and $\xi \in \mu_{N}(K)$, defined by the equation $\sum\limits_{n_{1}=0}^{n-1} \xi^{n_{1}}n_{1}^{l} = \xi^{n} \sum\limits_{m=0}^{l+1} \mathcal{B}_{m}^{l}(\xi) n^{m}$ for all $n \in \mathbb{N}$ (see the lemma 3.1.3 of \cite{I-1}). We denote by $\mathcal{B}_{m}^{l}=\mathcal{B}_{m}^{l}(1)$. For $l,m \in \mathbb{N}^{2}$ such that $1 \leqslant m \leqslant l+1$, we have $\mathcal{B}_{m}^{l} = \frac{1}{l+1} {l+1 \choose m} B_{l+1-m}$  where $B$ denotes Bernoulli numbers, and the others $\mathcal{B}_{m}^{l}$ are $0$. For $\xi \in \mu_{N}(K) - \{1\}$, $n,l \in \mathbb{N}^{\ast}$, we have $\mathcal{B}_{m}^{l}(\xi) \in \mathbb{Z}[\xi,\frac{1}{\xi},\frac{1}{\xi-1}]$,  and a formula for $\mathcal{B}_{m}^{l}(\xi)$ can be obtained by applying $\displaystyle\big( T \frac{d}{dT} \big)^{l}$ to the equation $\displaystyle\frac{T^{n}-1}{T-1} = \sum\limits_{n_{1}=0}^{n-1} T^{n_{1}}$, where $T$ is a formal variable.

\begin{Lemma} \label{elimination of ui and ri} Let $w= \big(
(n_{i})_{d};(\xi_{i})_{d+1} \big)$. We fix $(l_{i})_{d} \in \mathbb{N}^{d}$ and $(v_{i})_{d} \in \{0,\ldots,\frac{\tilde{\alpha}}{\tilde{\alpha}_{0}}-1\}$.
\newline\indent Let $R$ be the ring generated by $N$-th roots of unity and numbers $\frac{1}{1-\xi}$ where $\xi\not=1$ is a root of unity.
For any word $w'$ over $e_{0\cup\mu_{N}}$, there exists a polynomial $$ P_{w,w',(l_{i})_{d},(v_{i})_{d}} \in R[(Q_{j,j+1})_{1 \leqslant i \leqslant d-1},(B_{m,l,\xi})_{\substack{1\leqslant l \leqslant l_{1}+\ldots+l_{d}+d \\ 0 \leqslant m \leqslant l+1 \\ \xi \in \mu_{N}(K)}}]$$ 
with degree at most $l_{1}+\ldots+l_{d}+d$ in the variables $Q_{j,j+1}$, and with total degree at most $d$ in the variables $B_{m,l,\xi}$, which is non-zero for finitely many $w$'s, and such that we have
\begin{multline} \label{eq:5 2 1} \sum_{\substack{(u_{i})_{d} \in \mathbb{N}^{d}\\ (r_{i})_{d} \in \{1,\ldots,q^{\tilde{\alpha}_{0}} - 1 \}
\\ 0<q^{\tilde{\alpha}_{0}v_{1}}(q^{\tilde{\alpha}_{0}}u_{1}+r_{1}) < \ldots < q^{\tilde{\alpha}_{0}v_{d}}(q^{\tilde{\alpha}_{0}}u_{d}+r_{d}) < q^{\tilde{\alpha}}}}
\bigg( \prod_{j=1}^{d} \big(\frac{\xi_{j+1}}{\xi_{j}} \big)^{u_{j}+r_{j}}  \frac{(q^{\tilde{\alpha}_{0}})^{l_{i}+n_{i}}u_{i}^{l_{i}}}{r_{j}^{l_{j}+n_{j}}} \bigg)
\\ = \sum_{w'\text{ word on }e_{0\cup\mu_{N}}} P_{w,w',(l_{i})_{d},(v_{i})_{d}} \bigg((q^{\tilde{\alpha}_{0}(|v_{j+1}-v_{j}|-1)})_{1 \leqslant j \leqslant d},(\mathcal{B}_{m}^{l}(\xi))_{\substack{0 \leqslant m \leqslant \sum_{i=1}^{d}l_{i}+d+1 \\ \xi \in \mu_{N}(K)}} \bigg) \har_{q^{\tilde{\alpha_{0}}}}(w') .
\end{multline}
\end{Lemma}

\begin{proof} If $d=1$ we can apply the definition of the numbers $\mathcal{B}_{m}^{l}(\xi)$, $\xi \in \mu_{N}(K)$ mentioned above.
\newline\indent If $d>1$, let $i \in \{1,\ldots,d\}$ such that $v_{i} = \min (v_{1},\ldots,v_{d})$; 
we fix $u_{i-1},r_{i-1}$ and $u_{i+1},r_{i+1}$.  By (\ref{eq:change of domain of summation}), one has natural functions $f_{1},f_{2},f_{3},f_{4}$ such that we can write
\begin{multline*} \sum\limits_{ q^{\tilde{\alpha}_{0}v_{i-1}}(q^{\tilde{\alpha}_{0}}u_{i-1}+r_{i-1}) < q^{\tilde{\alpha}_{0}v_{i}}(q^{\tilde{\alpha}_{0}}u_{i}+r_{i}) < q^{\tilde{\alpha}_{0}v_{i+1}}(q^{\tilde{\alpha}_{0}}u_{i+1}+r_{i+1})} 
\\ = \sum\limits_{\substack{
f_{1}(u_{i-1},r_{i-1},u_{i+1},r_{i+1}) \leqslant u_{i} \leqslant f_{2}(u_{i-1},r_{i-1},u_{i+1},r_{i+1}) \\ r_{i}<r_{i+1} }} 
+ \sum\limits_{\substack{
f_{3}(u_{i-1},r_{i-1},u_{i+1},r_{i+1}) \leqslant u_{i} \leqslant f_{4}(u_{i-1},r_{i-1},u_{i+1},r_{i+1}) \\ r_{i} \geqslant r_{i+1} }}  .  
\end{multline*}
Using that equality we can apply the result in depth 1 i.e. the definitions of the numbers $\mathcal{B}_{m}^{l}(\xi)$ to express the sum 
$$ \sum\limits_{ q^{\tilde{\alpha}_{0}v_{i-1}}(q^{\tilde{\alpha}_{0}}u_{i-1}+r_{i-1}) < q^{\tilde{\alpha}_{0}v_{i}}(q^{\tilde{\alpha}_{0}}u_{i}+r_{i}) < q^{\tilde{\alpha}_{0}v_{i+1}}(q^{\tilde{\alpha}_{0}}u_{i+1}+r_{i+1})} \big(\frac{\xi_{i+1}}{\xi_{i}} \big)^{u_{i}+r_{i}}  \frac{(q^{\tilde{\alpha}_{0}})^{l_{i}+n_{i}}u_{i}^{l_{i}}}{r_{i}^{l_{i}+n_{i}}} . $$
This gives an expression for the left-hand side of (\ref{eq:5 2 1}) as a sum over $d-1$ variables $u_{i}$ and $d-1$ variables $r_{i}$, which is of a similar type. Continuing this procedure, we obtain an expression depending on the $r_{i}$'s via localized multiple harmonic sums, in the sense of \cite{I-2} \S3.2:
$\displaystyle  \sum\limits_{0<r_{1}<\ldots<r_{d}<q^{\tilde{\alpha}_{0}}} \frac{\big( \frac{\rho_{2}}{\rho_{1}} \big)^{r_{1}} \ldots \big(\frac{\rho_{d+1}}{\rho_{d}}\big)^{r_{d}} \big(\frac{1}{\rho_{d+1}}\big)^{q^{\tilde{\alpha}_{0}}}}{r_{1}^{\tilde{n}_{1}} \ldots r_{d}^{\tilde{n}_{d}}}$ with a positive integer $d$, for any $N$-th roots of unity $\rho_{i}$ ($1 \leqslant i \leqslant d+1$), $\tilde{n}_{i} \in \mathbb{Z}$ ($1 \leqslant i \leqslant d$) which are not necessarily positive. This can be expressed in terms of $q^{\tilde{\alpha}_{0}}$ and usual prime weighted multiple harmonic sums $\har_{q^{\tilde{\alpha}_{0}}}$ by the main result of \cite{I-2}, \S3.2. When we obtain products of numbers $\har_{q^{\tilde{\alpha_{0}}}}(w)$, we can linearize that expression by using the quasi-shuffle relation of \cite{Hoffman}.
\end{proof}

In the expression of  $\har_{q^{\tilde{\alpha}}}\big((n_{i})_{d};(\xi_{i})_{d+1}\big)$ obtained by combining lemma 3.1.1 and lemma 3.1.2, we are going to sum over all the possible values of the parameters $v_{i}$ in proposition \ref{elimination of ui and ri}. In view of that, we need another lemma. 

\begin{Lemma} \label{lemma in iteration of series} Let $d,M \in \mathbb{N}^{\ast}$; let $A_{1},\ldots,A_{d} \in R[T]$ be polynomials, with $R\subset \mathbb{Q}$, and $T_{1},\ldots,T_{d}$ formal variables. There exist coefficients $C_{\delta_{1},\ldots,\delta_{d}}(M) \in R[T]$ such that we have 
$$ \sum_{0\leqslant \tilde{v}_{1} <\ldots< \tilde{v}_{d} \leqslant M-1} \prod_{i=1}^{d} T_{i}^{\tilde{v}_{i}} A_{i}(\tilde{v}_{i})
= \sum_{\substack{
\forall i,\text{ } l_{1}+ \ldots + l_{i} 
\leqslant \deg (A_{1})+\ldots+ \deg A_{i}
\\ (U_{1},\ldots,U_{d}) \in \mathbb{Q}(T_{1},\ldots,T_{d})^{d} \text{s.t.} 
\\ U_{1} = T_{1} \\ \forall i \in \{1,\ldots,d-1\}, 
U_{i+1} \in \big\{ U_{i}T_{i+1} , -T_{i+1} \big\}
\\ \delta_{1} ,\ldots , \delta_{d} \geqslant 0 
}}
C_{\delta_{1},\ldots,\delta_{d}}(M) \prod_{i=1}^{d} \frac{U_{i}^{\delta_{i}}}{(U_{i}-1)^{\delta_{i}+1}} . $$
\end{Lemma}

\begin{proof} (a) For any $n \in \mathbb{N}^{\ast}$, $m \in \mathbb{N}$, we have
$\big( T \frac{\partial}{\partial T} \big) \frac{T^{m}}{(T-1)^{n}} = m \frac{T^{m}}{(T-1)^{n}} - n \frac{T^{m+1}}{(T-1)^{n+1}}$. Thus, by induction on $\alpha$, for all $\alpha \in \mathbb{N}^{\ast}$, $\big( T \frac{\partial}{\partial T}\big)^{\alpha} \frac{T^{m}}{(T-1)^{n}} = \sum\limits_{l=0}^{\alpha} (-n)(-n-1) \ldots (-n-l+1) \frac{T^{m+l}}{(T-1)^{n+l}}$. Moreover we have 
$\sum\limits_{0 \leqslant \tilde{v} \leqslant W-1} T^{\tilde{v}} \tilde{v}^{\alpha} 
= 
\big( T \frac{\partial}{\partial T} \big)^{\alpha} \big( \sum\limits_{0 \leqslant \tilde{v} \leqslant W-1} T^{\tilde{v}} \big)
=
\big( T \frac{\partial}{\partial T} \big)^{\alpha} \big( \frac{T^{\tilde{v}}-1}{T-1} \big)$. Whence
$\sum\limits_{0\leqslant \tilde{v} \leqslant M-1} T^{v} v^{\alpha} 
= \sum\limits_{l=0}^{\alpha} (-1)^{l} l! \frac{ T^{l}}{(T-1)^{l+1}} (T^{M}-1)$.
This gives the result for $d=1$ by linearity with respect to $A_{1}$.
\newline\indent (b) Let us prove the result by induction on $d$: assume that $A_{1} = \sum\limits_{\alpha_{1}=0}^{\deg A_{1}} u_{\alpha_{1}} \tilde{v}^{\alpha_{1}}$ with $u_{\alpha_{1}} \in \mathbb{Q}$. We have, for all $\alpha_{1} \in \{0,\ldots, \deg A_{1} \}$: 
\begin{multline*}
\sum\limits_{0\leqslant \tilde{v}_{1} <\ldots< \tilde{v}_{d} \leqslant M-1} T_{1}^{\tilde{v}_{1}} \tilde{v}_{1}^{\alpha_{1}} \prod_{i=2}^{d} T_{i}^{\tilde{v}_{i}} A_{i}(\tilde{v}_{i}) =
\\ \sum\limits_{l=0}^{a_{1}} (-1)^{l} l! \frac{1}{(T_{1}-1)^{l+1}}
\sum\limits_{0\leqslant \tilde{v}_{2} <\ldots< \tilde{v}_{d} \leqslant M-1} \bigg( (T_{1}T_{2})^{\tilde{v}_{2}}\prod_{i=3}^{d} T_{i}^{\tilde{v}_{i}} A_{i}(\tilde{v}_{i}) - \prod_{i=2}^{d} T_{i}^{\tilde{v}_{i}} A_{i}(\tilde{v}_{i}) \bigg) .
\end{multline*} 
Whence the result by induction and by linearity.
\end{proof}

Combining lemmas 3.1.1, 3.1.2, 3.1.3, we can now sum over all the $u_{i}$'s, $r_{i}$'s and $v_{i}$'s and write an expression of  $\har_{q^{\tilde{\alpha}}}\big((n_{i})_{d};(\xi_{i})_{d+1}\big)$ as a function of $\tilde{\alpha}$ as we wanted.

\begin{Proposition} \label{elimination of vi}Let a harmonic word $w= \big((n_{i})_{d};(\xi_{i})_{d+1}\big)$. Let us fix $(l_{i})_{d} \in \mathbb{N}^{d}$.
\newline\indent Let $R$ be the ring of the proposition \ref{elimination of ui and ri}. For every word $w'$ over $e_{0\cup\mu_{N}}$, there exists a polynomial $P_{w,w',(l_{i})_{d}} \in  R\big[\tilde{Q},A,(B_{m,l,\xi})_{\substack{1\leqslant l \leqslant l_{1}+\ldots+l_{d}+d \\ 1 \leqslant m \leqslant l+1 \\ \xi \in \mu_{N}(K)}}\big]$ with degree at most $\sum\limits_{i=1}^{d} l_{i}+d$ in $\tilde{Q}$, and with total degree at most $d$ in the variables $B_{m,l,\xi}$, such that we have
\begin{multline} \label{eq:5 2 2}
\sum_{(v_{i})_{d} \in \{1,\ldots,q^{\tilde{\alpha}_{0}}-1\}}
P_{w,w',(l_{i})_{d},(v_{i})_{d}} \Big((q^{\tilde{\alpha}_{0}(|v_{j+1}-v_{j}|-1)})_{1 \leqslant j \leqslant d},(\mathcal{B}_{m}^{l}(\xi))_{\substack{0 \leqslant m \leqslant \sum_{i=1}^{d}l_{i}+d+1 \\ \xi \in \mu_{N}(K)}} \Big)
\\ =  P_{w,w',(l_{i})_{d}} \Big(q^{\tilde{\alpha}},\frac{\tilde{\alpha}}{\tilde{\alpha}_{0}}, (\mathcal{B}_{m}^{l}(\xi))_{\substack{0 \leqslant m \leqslant \sum_{i=1}^{d}l_{i}+d+1 \\ \xi \in \mu_{N}(K)}} \Big) .
\end{multline}
\end{Proposition}

\begin{proof} The set $\{0,\ldots,\frac{\tilde{\alpha}}{\tilde{\alpha}_{0}}-1\}^{d}$ admits a partition, which depends only on $d$, indexed by the set of couples $(E,\sigma)$, where $E$ is a partition of $\{1,\ldots,d\}$ and $\sigma$ is a permutation of $\{1,\ldots,\sharp E\}$, defined as follows: for each $(v_{1},\ldots,v_{d}) \in [0,k-1]^{d}$, and each such $(E,\sigma)$, we say that
$(v_{1},\ldots,v_{d}) \in (E,\sigma)$ if and only if, for all, $i,i',a$: $\left\{
\begin{array}{ll}
v_{i} = v_{i'} \text{ for } i,i' \in P_{\sigma(a)}
\\ v_{i} < v_{i'} \text{ for } i \in P_{\sigma(a)}, i' \in P_{\sigma(a+1)}
\end{array} \right.$. By the proof of proposition \ref{elimination of ui and ri}, the function $(v_{i})_{d} \mapsto P_{w,w',(l_{i})_{d},(v_{i})_{d}}$ is constant on each term of that partition (since $\xi^{q} = \xi$ for all $\xi \in \mu_{N}(K)$, we have $\xi^{q^{\tilde{\alpha}_{0}v}}=\xi$ for all $v \in \mathbb{N}^{\ast}$). We split the left-hand side of (\ref{eq:5 2 2}) as $\sum\limits_{(v_{i})_{d} \in \{1,\ldots,\frac{\tilde{\alpha}}{\tilde{\alpha}_{0}}-1\}} = \sum\limits_{(E,\sigma)} \sum\limits_{(v_{i})_{d} \in (E,\sigma)}$ and we compute each subsum $\sum\limits_{(v_{i})_{d} \in (E,\sigma)}$. By multilinearity we can assume that $P_{w,w',(l_{i})_{d},(v_{i})_{d}}$ is a monomial in the $q^{\tilde{\alpha}_{0}(|v_{j+1}-v_{j}-1|)}$'s. Thus the subsum is a function of the type
\begin{multline}
\label{eq: 5 2 3}
\sum_{0 \leqslant \tilde{v}_{1} < \ldots \tilde{v}_{d'} \leqslant M-1}
T_{i_{1}}^{\tilde{v}_{1}}\ldots T_{i_{r}}^{\tilde{v}_{r}} ,
\end{multline}
applied to $\tilde{v}_{i} = v_{\sigma(i+1)}-v_{\sigma(i)}-1$ and  $T_{i}=q^{\tilde{\alpha}_{0}}$, where $d',M \in \mathbb{N}^{\ast}$, $I = \{i_{1},\ldots,i_{r} \} \subset \{1,\ldots,d'\}$ with $i_{1} < \ldots < i_{r}$, and $T_{i_{1}},\ldots,T_{i_{r}}$ formal variables. Moreover, $\sum\limits_{i\leqslant i+1<\ldots<j} 1$ is a polynomial function of $(i,j)$ with coefficients in $\mathbb{Z}$. Thus we can express (\ref{eq: 5 2 3}) by lemma \ref{lemma in iteration of series}. This provides the result.
\end{proof}

\subsection{The relation of iteration of the harmonic Frobenius of series}

Using the result of \S3.1 we can now formalize the iteration of the harmonic Frobenius from the point of view of series. We refer to
$K\langle\langle e_{0\cup\mu_{N}}\rangle\rangle_{\har}^{\Sigma} = \prod\limits_{\text{w harmonic word}} K.w$ defined in \cite{I-2}, \S3.

\begin{Definition} \label{def iteration of the harmonic Frobenius of series} Let the map of iteration of the harmonic Frobenius of series be the map 
$\iter_{\har}^{\Sigma}(\Lambda,\Lambda^{a},a): K\langle\langle e_{0\cup\mu_{N}}\rangle\rangle_{\har}^{\Sigma} \rightarrow K [\Lambda,\Lambda^{a},a] \langle\langle e_{0\cup \mu_{N}} \rangle\rangle_{\har}^{\Sigma}$ defined by, for any word $w$,  
$$ \iter_{\har}^{\Sigma}(f)[w] = \sum\limits_{w'\text{ word on }e_{0\cup\mu_{N}}} P_{w,w',(l_{i})_{d}} \Big(\Lambda^{a},a,(\mathcal{B}_{m}^{l}(\xi))_{\substack{0 \leqslant m \leqslant \sum_{i=1}^{d}l_{i}+d+1 \\ \xi \in \mu_{N}(K)}} \Big) f[w'] . $$
\end{Definition}

Let us now finish the proof of equation (\ref{eq:third of I-3}).
By lemma \ref{proposition starting}, proposition \ref{elimination of ui and ri} and proposition \ref{elimination of vi}, the only thing to check is the convergence of the series, which are infinite sums over $(l_{i})_{d} \in \mathbb{N}^{d}$. This follows from the following facts:
\newline\indent (a) For any $n \in \mathbb{N}^{\ast}$, it follows from $p^{v_{p}(n)}\leqslant n$ that $v_{p}(\frac{1}{n}) \geqslant -\frac{\log(n)}{\log(p)}$; moreover, for any $n \in \mathbb{N}^{\ast}$, we have $v_{p}(B_{n}) \geqslant -1$ (this is part of von Staudt-Clausen's theorem). Thus for all $l,m$, we have $v_{p}(\mathcal{B}_{m}^{l}) \geqslant - 1 - \frac{\log(l+1)}{\log(p)}$, and, given that $|\xi|_{p} = |1-\xi|_{p}=1$ for all $\xi \in \mu_{N}(K) - \{1\}$, we have, for all $l,m$, and $\xi \in \mu_{N}(K)- \{1\}$, $v_{p}(\mathcal{B}_{m}^{l}(\xi)) \geqslant 0$
\newline\indent (b) If $T_{1},T_{2}$ are formal variables and $m \in \mathbb{N}^{\ast}$, we have $\displaystyle \frac{T_{1}^{m}-1}{T_{1}-1} - \frac{T_{2}^{m}-1}{T_{2}-1} = \frac{T_{1}^{m}-T_{1}}{T_{1}-1} - \frac{T_{2}^{m}-T_{2}}{T_{2}-1}$.
\newline\indent (c) For any $z \in K$ such that $v_{p}(z) \not= 0$, we have $v_{p}(\frac{1}{z-1}) > 0$ if $v_{p}(z)> 0$, and $v_{p}(\frac{1}{z^{-1}-1}) > v_{p}(z^{-1})$ if $v_{p}(z) < 0$.

\begin{Remark} The equation (\ref{eq:third of I-3}) is related to the formula $\har_{p^{\alpha}\mathbb{N}} = \har_{p^{\alpha}} \circ_{\har}^{\Sigma} \har_{\mathbb{N}}^{(p^{\alpha})}$ proved in \cite{I-2}, where $\circ_{\har}^{\Sigma}$ is the pro-unipotent harmonic action of series introduced in \cite{I-2}, \S4.3: restricting that equation to $\har_{m}$ with $m$ a power of $p$ gives a functional equation satisfied by the map $\alpha \mapsto \har_{p^{\alpha}}$, which expresses
$\har_{p^{\alpha+\beta}}$ in terms of $\har_{p^{\alpha}}$ and $\har^{(p^{\alpha})}_{p^{\beta}}$.
\end{Remark}

\begin{Remark} As in \cite{I-2}, the computation which leads to the above result remains true for the generalization of cyclotomic multiple harmonic sums obtained by replacing the factors $\displaystyle\frac{1}{m_{i}^{n_{i}}}$, $1\leqslant i \leqslant d$ in (\ref{eq:harmonic sum}) by, more generally, factors $\chi_{i}(m_{i})$ where $\chi_{i}$ are locally analytic group endomorphisms of the multiplicative group $K^{\ast}$, which are analytic on disks of radius $p^{-\alpha}$.
\end{Remark}

\begin{Remark} The main theorem gives formulas for $p$-adic cyclotomic multiple zeta values which depend on an additional parameter, a number of iterations of the Frobenius different from the one under consideration. Here is another way to obtain formulas with parameters. The computation of regularized iterated integrals in \cite{I-1}, \S3 can be done by replacing the Euclidean division by $p^{\alpha}$ in $\mathbb{N}$ by the Euclidean division by $p^{\beta}$ with $\beta\geq \alpha$. This gives, for example, $\zeta_{p,\alpha}(n)
= \displaystyle \frac{p^{\alpha n}}{n-1} \lim_{|m|_{p} \rightarrow 0} 
\frac{1}{p^{\beta}m} \sum_{\substack{0<m_{1}<p^{\beta}m \\ p^{\alpha} \nmid m}} \frac{1}{m_{1}^{n}}$, and this gives formulas in which the prime weighted multiple harmonic sums are replaced by the following generalization, where $(l_{1},\ldots,l_{d}) \in \mathbb{N}^{d}$, $I,I' \subset \{1,\ldots,d\}$ and $\beta$: 
$$ p^{\alpha\sum_{i=1}^{d} n_{i}+\beta\sum_{i=1}^{d}l_{i}}
\sum_{\substack{0=m_{0}<m_{1}<\ldots<m_{d}<p^{\beta}
\\ \text{ for }j \in I, m_{j-1} \equiv m_{j} [p^{\alpha}]
\\  \text{ for }j \in I', m_{j} \equiv 0 [p^{\alpha}] }}                \frac{(\frac{\xi_{2}}{\xi_{1}})^{m_{1}}\ldots (\frac{\xi_{d+1}}{\xi_{d}})^{m_{d}} (\frac{1}{\xi_{d+1}})^{p^{\beta}}}{m_{1}^{n_{1}+l_{1}}\ldots m_{d}^{n_{d}+l_{d}}} . $$
\end{Remark}

\begin{Example} Let us consider the case of $\mathbb{P}^{1} - \{0,1,\infty\}$ $(N=1)$, for which we have $q=p$, $\tilde{\alpha}=\alpha$, $\tilde{\alpha}_{0}=\alpha_{0}$, and depth one and two. Equation (\ref{eq:third of I-3}) is, for all $n \in \mathbb{N}^{\ast}$,
$$ \har_{p^{\alpha}}(n) = 
-\sum\limits_{l\geqslant 0} {-n \choose l} \har_{p^{\alpha_{0}}}(l+n) \sum\limits_{u=1}^{l+1} \mathcal{B}_{u}^{l}  \frac{1}{1-p^{\alpha_{0}(u+n)}} - 
\sum\limits_{u\geqslant n+1} p^{\alpha u} \frac{1}{1-p^{\alpha_{0}u}} \sum\limits_{l\geqslant u-n-1} {-n \choose l} \har_{p^{\alpha_{0}}}(l+n) \mathcal{B}_{u-n}^{l} . $$
\indent For all $n_{1},n_{2} \in \mathbb{N}^{\ast}$, $\har_{p^{\alpha}}(n_{1},n_{2})$ is the sum of the following terms, where the variables
$v_{1},v_{2}$ are those defined in lemma \ref{proposition starting} and where, for a set $E$, $1_{E}$ means the characteristic function of $E$:
\newline\indent $\bullet$ the term "$v_{1}=v_{2}$":
$\sum\limits_{\substack{u\geqslant 1 \\ l_{1},l_{2} \geqslant 0 \\ l_{1}+l_{2} \geqslant u-1}} \frac{p^{\alpha(u+n_{1}+n_{2})}-1}{p^{\alpha_{0}(u+n_{1}+n_{2})}-1} \prod_{i=1}^{2} {-n_{i} \choose l_{i}}
\bigg( \mathcal{B}_{u}^{l_{1},l_{2}} \prod_{i=1}^{2}\har_{p^{\alpha_{0}}}(n_{i}+l_{i}) + \mathcal{B}_{u}^{l_{1}+l_{2}} \har_{p^{\alpha_{0}}}(n_{1}+l_{1},n_{2}+l_{2}) \bigg);
$
\newline 
\indent $\bullet$ the term "$v_{1}<v_{2}$":
$\sum\limits_{\substack{M_{1}, M_{2} \geqslant -1\\ u,t\geqslant 1}} \bigg( 
\frac{1_{t \not= u+n_{2}}}{p^{\alpha_{0}(n_{2}+u-t)}-1} \big( \frac{ p^{\alpha(n_{1}+n_{2}+u)}-p^{\alpha_{0}(n_{1}+n_{2}+u)}}{p^{\alpha_{0}(n_{2}+n_{1}+u)}-1} - \frac{ p^{\alpha(n_{1}+t)}-p^{\alpha_{0}(n_{1}+t)}}{p^{\alpha_{0}(n_{1}+t)}-1} \big)
\\ + 1_{t=u+n_{2}}\big( \frac{\alpha(p^{\alpha(n_{1}+n_{2}+u)} }{p^{\alpha_{0}(n_{1}+n_{2}+u)}-1} + \frac{1-p^{\alpha(n_{1}+n_{2}+u)}}{(p^{\alpha_{0}(n_{1}+n_{2}+u)}-1)^{2}} \big) \bigg) \times 
\bigg( \mathcal{B}_{t}^{M_{1}+t} \mathcal{B}_{u}^{M_{2}+u}
\newline \sum_{j=0}^{\min(t,M_{2}+u)} {t \choose j} {-n_{1} \choose M_{1}+t} {-n_{2} \choose M_{2}+u-j} \har_{p^{\alpha_{0}}}(n_{1}+M_{1}+t)\har_{p^{\alpha_{0}}}(n_{2}+M_{2}+u-t)\bigg);$
\newline\indent $\bullet$ the term "$v_{1}>v_{2}$": by the change of variable $(m_{1},m_{2}) \mapsto (p^{\alpha}-m_{1},p^{\alpha}-m_{2})$, it is 
$\sum\limits_{\substack{0<m_{1}<m_{2}<p^{\alpha} \\ v_{p}(n_{1})>v_{p}(n_{2})}} \frac{(p^{\alpha})^{n_{1}+n_{2}}}{n_{1}^{n_{1}}n_{2}^{n_{2}}}
= \sum\limits_{\substack{l_{1},l_{2} \geqslant 0 \\ 0<n_{1}<n_{2}<p^{\alpha} \\ v_{p}(n_{1})<v_{p}(n_{2})}} {-n_{1} \choose l_{1}}{-n_{2} \choose l_{2}} \frac{(p^{\alpha})^{l_{1}+l_{2}+n_{1}+n_{2}}}{n_{1}^{n_{1}}n_{2}^{n_{2}}} .$
\end{Example}

\subsection{Interpretation in terms of cyclotomic multiple harmonic sums viewed as functions of the upper bound of their domain of summation}

The main result above gives a description of $\har_{q^{\tilde{\alpha}}}$ as a function of $q^{\tilde{\alpha}}$ (and $\tilde{\alpha}$) regarded as a $p$-adic integer. Let us extend the question and consider the study of $\har_{m}$ as a function of $m$, for any $m \in \mathbb{N}^{\ast}$ regarded as a $p$-adic integer. We are going to remove the factor $m^{\weight}$ in $\har_{m}$, i.e. consider the (non-weighted) cyclotomic multiple harmonic sums $\frak{h}_{m} \big((n_{i})_{d};(\xi_{i})_{d+1}\big)$ of equation (\ref{eq:harmonic sum}).

Intuitively, $\frak{h}_{m}$ is a highly discontinuous function of $m$, but we have proved by the main theorem that $\har_{q^{\tilde{\alpha}}}$ has a power series expansion in terms of $q^{\tilde{\alpha}}$. The goal of the next proposition is to write a decomposition of $\frak{h}_{m}$ in a way which explains the relation between these two phenomena, by using the $q$-adic expansion of $m$, in order to clarify the dependence of $\har_{m}$ in $m$.

Below we use the following definition: an increasing connected partition of a subset of $\mathbb{N}$ is a partition of that set into sets $J_{i}$ of consecutive integers, such that each element of $J_{i}$ is less than each element of $J_{i'}$ when $i<i'$.

\begin{Proposition} \label{n variable 1}
	(i) Let $m \in \mathbb{N}^{\ast}$, and let its $q$-adic expansion be $m = a_{y_{d'}}q^{y_{d'}} + a_{y_{d'-1}}q^{y_{d'-1}} + \ldots + a_{y_{1}}q^{y_{1}}$, with $y_{d'}> \ldots > y_{1}$, and $a_{y_{d'}},\ldots,a_{y_{1}} \in \{1, \ldots,q-1\}$. Let $\nu_{j'} = 
	a_{y_{d'}}q^{y_{d'}} + \ldots + a_{y_{d'-j'+1}}q^{y_{d'-j'+1}}$ for $j' \in \{1,\ldots,d'\}$. We have
\begin{multline} \frak{h}_{m} \big((n_{i})_{d};(\xi_{i})_{d+1} \big) = \sum_{\substack{\frak{n} = \{\nu_{j''_{1}},\ldots,\nu_{j''_{d''}}\} \subset \{\nu_{1},\ldots,\nu_{d'}\}
\\ \jmath: \frak{n} \hookrightarrow \{1,\ldots,d\}\text{ injective }
\\ J_{0} \amalg \ldots \amalg J_{d'} = \{1,\ldots,d\} - \jmath(\frak{n}) \text{, satisfying }(\ast)}}
\prod_{j''=1}^{d''} \frac{1}{\nu_{j''}^{n_{\jmath(\nu_{j''})}}} \prod_{j'=0}^{d'}
\sum_{l_{j_{j'}^{\max}},\ldots,l_{j_{j'}^{\min}}\geqslant 0}
\\
\bigg( \prod_{u=j_{j'}^{\min}}^{j_{j'}^{\max}} { -n_{u} \choose l_{u}} \bigg) \big( \sum_{l=d'}^{d'-j'+1}a_{y_{l}}q^{y_{l}} \big)^{\sum_{u=j_{j'}^{\min}}^{j_{j'}^{\max}} l_{u}} \frak{h}_{a_{y_{d'-j'}}q^{y_{d'-j'}}}
\big( (n_{j}+l_{j})_{j_{j',min}<j<j_{j',max}};(\xi)_{j_{\min}< j< j_{\max}+1}\big) ,
\end{multline}
where $(\ast)$ is that 
$J_{0}\text{ }\amalg \text{ }\ldots \text{ }\amalg\text{ }J_{d'}$ is an increasing connected partition of $\{1,\ldots,d\} - \jmath(\frak{n})$, such that each 
$J_{\jmath(\nu_{j''})}\text{ }\amalg\text{ }\ldots\text{ }\amalg\text{ }J_{\jmath(\nu_{j''+1})-1}$, $j''=1,\ldots,d''$, is an increasing connected partition of
$(\{1,\ldots,d\} - \jmath(\frak{n})) \text{ }\cap\text{ } ]\jmath(\nu_{j''}),\jmath(\nu_{j''+1})[$.
\newline\indent (ii) Let $n \in \mathbb{N}^{\ast}$, whose decomposition in base $q$ is of the form $a q^{y}$, with $a \in \{1,\ldots,q-1\}$ and $y \in \mathbb{N}^{\ast}$. Let $\nu_{j'} = j'q^{y}$ for $j' \in \{1,\ldots,a-1\}$. We have
\begin{multline} \frak{h}_{a q^{y}} \big((n_{i})_{d};(\xi_{i})_{d+1} \big) = \sum_{\substack{\frak{n} = \{\nu_{j''_{1}},\ldots,\nu_{j''_{d''}}\} \subset \{\nu_{1},\ldots,\nu_{d'}\}
\\ \jmath: \frak{n} \hookrightarrow \{1,\ldots,d\}\text{ injective }
\\ J_{0} \amalg \ldots \amalg J_{d'} = \{1,\ldots,d\} - \jmath(\frak{n}) \text{, satisfying }(\ast)}}
\prod_{j''=1}^{d''} \frac{1}{\nu_{j''}^{n_{\jmath(\nu_{j''})}}} \prod_{j'=0}^{d'}
\sum_{l_{j_{j'}^{\max}},\ldots,l_{j_{j'}^{\min}}\geqslant 0}
\\ \bigg( \prod_{u=j_{j'}^{\min}}^{j_{j'}^{\max}} { -n_{u} \choose l_{u}} \bigg) \big( j'q^{y} \big)^{\sum_{u=j_{j'}^{\min}}^{j_{j'}^{\max}} l_{u}} \frak{h}_{q^{y_{d'-j'}}}
\big((n_{j_{j'}} + l_{j_{j'}})_{j_{\min}<j'<j_{\max}};  (\xi_{j_{j'}})_{j_{\min}<j'<j_{\max}+1}\big) .
\end{multline}
\end{Proposition} 

\begin{proof} (i) and (ii) We apply the "formula of splitting" of multiple harmonic sums of \cite{I-2}, \S3 at $\{\nu_{1},\ldots,\nu_{r}\}$; this gives $\displaystyle \frak{h}_{m}\big((n_{i})_{d};(\xi_{i})_{d+1}\big) = \sum_{\substack{\frak{n} = \{\nu_{j_{1}},\ldots,\nu_{j_{d''}}\} \subset \{\nu_{1},\ldots,\nu_{d'}\}
\\ \jmath: \frak{n} \hookrightarrow \{1,\ldots,d\}\text{ injective }
\\ J_{0} \amalg \ldots \amalg J_{d'} = \{1,\ldots,d\} - \jmath(\frak{n}) \text{, satisfying }(\ast)}}
\prod_{j''=1}^{d''} \frac{1}{n_{\jmath(\nu_{j''})}^{n_{\jmath(\nu_{j''})}}} \prod_{j'=0}^{d'} \frak{h}_{\nu_{j'},\nu_{j'+1}}(w|_{J_{j'}})$ and we express each factor $ \frak{h}_{\nu_{j'},\nu_{j'+1}}(w|_{J_{j'}}) $ in terms of $\frak{h}_{\nu_{j'+1}-\nu_{j'}}$ by the $p$-adic formula of shifting of multiple harmonic sums of \cite{I-2}, \S4.1 writing $J_{j'} = [j_{j'}^{\min},j_{j'}^{\max}]$.
\end{proof}

\begin{Example} ($N=1$ and $d=1$) For all $n \in \mathbb{N}^{\ast}$, we have:
	\begin{multline}
	\frak{h}_{m}(n) = \frac{\frak{h}_{a_{y_{d'}}+1}(n)}{(p^{y_{d'}})^{n}}
	+ \sum_{i=1}^{d'-1} \sum_{l\geqslant 0} \frac{\frak{h}_{a_{y_{i}}+1}(n+l)}{(p^{y_{i}})^{n+l}} {-n \choose l} \big(a_{y_{d'}}p^{y_{d'}} + \ldots + a_{y_{i+1}} p^{y_{i+1}} \big)^{l} 
	\\ + \sum_{\substack{ 1 \leqslant j \leqslant d' 
			\\ 0 \leqslant a'_{y_{j}} \leqslant a_{y_{j}}-1}}
	\sum_{l\geqslant 0} {-n \choose l} \frak{h}_{y_{j}}(n+l) \bigg(\sum_{m=j+1}^{d'} a_{y_{m}} p^{y_{m}} + a'_{j} p^{y_{j}}\bigg)^{l} .
	\end{multline}
\end{Example}

In the formulas of the proposition, there are terms which are analytic functions of a power of $q$ by the main theorem, and certain factors which are ``polar'' in function of the $q$-adic expansion of $m$. This sheds light on the dependence of $\frak{h}_{m}$ on $m$.

The reason why studying $\frak{h}_{m}$ as a function of $m$ is a natural comes from \cite{I-1}, in which this question appeared implicitly. We have studied the map sending $m$ to the coefficient of degree $m$ in the power series expansion at $0$ of the overconvergent $p$-adic multiple polylogarithm  $\Li_{p,\alpha}^{\dagger}[w]$, for $w$ any word on $e_{0\cup \mu_{N}}$. We have proved that it can be extended to a locally analytic map on $\mathbb{Z}_{p}^{(N)}= \varprojlim \mathbb{Z}/Np^{u}\mathbb{Z}$ (\cite{I-1}, \S3). This map is a linear combination of multiple harmonic sums over the ring generated by $p$-adic cyclotomic multiple zeta values. Thus we can interpret them as a ``regularization'' of multiple harmonic sums. See also Appendix A of \cite{I-1}.

\section{Comparison between equations on integrals and equations on series}

We prove equation (\ref{eq:fourth of I-3}) and we discuss more generally the comparison between integrals and series.

\subsection{Uniqueness of the expansion of $\har_{q^{\tilde{\alpha}}}$ as a function of $\tilde{\alpha}$ and $q^{\tilde{\alpha}}$}

\begin{Proposition} \label{prop uniqueness}Let $\delta \in \mathbb{N}^{\ast}$, and a map $S: \mathbb{N}^{\ast} \cap [\delta,+\infty[ \rightarrow K$ such that we have, 
for all $a \in \mathbb{N}^{\ast} \cap [\delta,+\infty[$,  $S(a) = \sum\limits_{n \in \mathbb{N}} \sum\limits_{m=0}^{M} c_{n,m} (q^{a})^{n} a^{m}$, where $M \in \mathbb{N}^{\ast}$, and
$(c_{l,m})_{\substack{0 \leqslant n \\ 0 \leqslant m \leqslant M}} \in K^{\mathbb{N} \times \{0,\ldots,M\}}$ such that $\sum\limits_{n \in \mathbb{N}} \sum\limits_{m=0}^{M} |c_{n,m} q^{n}|_{p}< \infty$. 
If $S(a)=0$ for all $a\in \mathbb{N}^{\ast} \cap [\delta,+\infty[$, then we have $c_{n,m}=0$ for all $(n,m) \in \mathbb{N} \times \{0,\ldots,M\}$.
\end{Proposition}

\begin{proof} Let $a_{0} \in \mathbb{N}^{\ast} \cap [\delta,+\infty[$ and $u \in \mathbb{N}$. By taking $a = a_{0}+p^{u}$ in the equation $\sum\limits_{n \in \mathbb{N}} \sum\limits_{m=0}^{M} c_{n,m} (q^{a})^{n} a^{m}=0$ and by taking the limit $u \rightarrow \infty$, we get $\sum\limits_{m=0}^{M} c_{0,m} a_{0}^{m}=0$. Since this is true for infinitely many $a_{0}$, we get $c_{0,m}=0$ for all $m$. This implies that, for all $a$, $\sum\limits_{n \geqslant 1} \sum\limits_{m=0}^{M} c_{n,m} (q^{a})^{n} a^{m}= q^{a} \big( \sum\limits_{n \geqslant 1} \sum\limits_{m=0}^{M} c_{n,m} (q^{a})^{n-1} a^{m} \big)=0$, thus $\sum\limits_{n \in \mathbb{N}} \sum\limits_{m=0}^{M} c_{n+1,m} (q^{a})^{n} a^{m} = 0$. Whence the result: by induction on $n$, we have a contradiction if there exists $(n,m)$ such that $c_{n,m}\not=0$.
\end{proof}

Let us now prove (\ref{eq:fourth of I-3})). By \cite{I-1}, we have  $\Phi_{q,\tilde{\alpha}} \in K \langle \langle e_{0 \cup \mu_{N}} \rangle\rangle_{o(1)}$ for any $\tilde{\alpha} \in \mathbb{N}^{\ast}$. By Corollary \ref{comparison of vector spaces}, this implies that $|\Phi|_{q} = \sum\limits_{w \text{ word on }e_{0\cup \mu_{N}}} \underset{\tilde{\alpha} \in \mathbb{Z} \cup \{\pm\infty\}- \{0\}}{\sup}|\Phi_{q,\tilde{\alpha}}[w]|_{p}\text{ }w$ is a well-defined element of $K\langle \langle e_{0\cup \mu_{N}} \rangle\rangle_{o(1)}$. We have a similar bound for the coefficients of the expansion of each $\har_{q,\tilde{\alpha}}[w]$ obtained in \S3.3. Thus we can apply proposition \ref{prop uniqueness} to the power series expansion of each $\har_{q,\tilde{\alpha}}[w]$ in (\ref{eq:first of I-3}), (\ref{eq:second of I-3}), (\ref{eq:third of I-3}) to deduce that they are the same.

\begin{Example} \label{example of the theorem}The term of depth one ($d=1$), in the case of $\mathbb{P}^{1} - \{0,1,\infty\}$ (i.e. in the case $N=1$) for which $p=q$, $\tilde{\alpha} = \alpha$ and $\tilde{\alpha}_{0}=\alpha_{0}$, of the equations of the Theorem is the following, respectively (with $\mathcal{B}_{b}^{L+b} = \frac{1}{L+b+1}{L+b+1 \choose L}B_{L+1}$ for $1 \leqslant b \leqslant L+1$):
\newline\indent The fixed point equation of the harmonic Frobenius of integrals at (1,0) (equation (\ref{eq:first of I-3})):
\begin{equation} \har_{p^{\alpha}}(n) =  
\sum_{b = 1}^{\infty} (p^{\alpha})^{b+n} \Ad_{\Phi_{p,-\infty}}(e_{1}) \big[ e_{0}^{b}e_{1}e_{0}^{n-1}e_{1} \big] + 
\Ad_{\Phi_{p,\infty}}(e_{1}) \big[ \frac{1}{1-e_{0}}e_{1}e_{0}^{n-1}e_{1} \big] .
\end{equation}
\indent The relation of iteration of the harmonic Frobenius
of integrals at (1,0) (equation (\ref{eq:second of I-3})):
\begin{equation} \har_{p^{\alpha}}(n) = \sum\limits_{b = 1}^{\infty} \frac{(p^{\alpha})^{n}(p^{\alpha_{0}})^{b}}{p^{\alpha_{0}}-1} \Ad_{\Phi_{p,\alpha_{0}}}(e_{1})\big[e_{0}^{b}e_{1}e_{0}^{n-1}e_{1} \big] -  \frac{(p^{\alpha})^{n}}{p^{\alpha_{0}}-1} \Ad_{\Phi_{p,\alpha_{0}}}(e_{1})\big[\frac{1}{1-e_{0}}e_{1}e_{0}^{n-1}e_{1} \big] .
\end{equation}
\indent The relation of iteration of the harmonic Frobenius
of series (equation (\ref{eq:third of I-3})):
\begin{equation} \har_{p^{\alpha}}(n) = \sum\limits_{b = 1}^{\infty} \frac{p^{\alpha(n+b)}-1}{p^{\alpha_{0}(n+b)}-1} \sum\limits_{L= -1}^{\infty} \mathcal{B}^{L+b}_{b} \har_{p^{\alpha_{0}}}(n+b+L) .
\end{equation}
\indent The comparison between these three expressions (\ref{eq:fourth of I-3})
\begin{equation} \frac{(p^{\alpha_{0}})^{b}}{p^{\alpha_{0}}-1} \Ad_{\Phi_{p,\alpha_{0}}}(e_{1})\big[e_{0}^{b}e_{1}e_{0}^{n-1}e_{1} \big] = \Ad_{\Phi_{p,\infty}}(e_{1}) \big[ e_{0}^{b}e_{1}e_{0}^{n-1}e_{1} \big]= \frac{1}{p^{\alpha_{0}(n+b)}-1} \sum\limits_{L= -1}^{\infty} \mathcal{B}^{L+b}_{b} \har_{p^{\alpha_{0}}}(n+b+L) .
\end{equation}
\end{Example}

Generalizing this example to higher depths gives a new way to compute $p$-adic cyclotomic multiple zeta values.

\subsection{The map of comparison for all number of iterations}

In \cite{I-2}, definition 5.1.3, we have defined the map of comparison, from integrals to series, $\comp^{\Sigma\smallint}: K\langle\langle e_{0\cup \mu_{N}}\rangle\rangle_{o(1)}^{N} \rightarrow K\langle\langle e_{0\cup \mu_{N}}\rangle\rangle_{\har,o(1)}$, by
$(\comp^{\Sigma\smallint}((g_{\xi})_{\xi \in \mu_{N}(K)}))[e_{\xi_{d+1}}e_{0}^{n_{d}-1}e_{\xi_{d}} \ldots e_{0}^{n_{1}-1}e_{\xi_{1}}] = (-1)^{d} \sum\limits_{\xi \in \mu_{N}(K)} \xi^{-p^{\alpha}}  g_{\xi}[\frac{1}{1-e_{0}}e_{\xi_{d+1}}e_{0}^{n_{d}-1}e_{\xi_{d}}\ldots e_{0}^{n_{1}-1}e_{\xi_{1}}]$.

In the context of this paper, it is natural to define a variant of the map of comparison from integral to series, which takes into account the properties of the iterated Frobenius viewed as a function of its number of iterations, and which has the additional advantage of being injective.

\begin{Definition} Let the map
$\comp^{\Sigma\smallint}_{\iter}: \Ad_{\tilde{\Pi}_{1,0}(K)_{o(1)}}(e_{1}) \rightarrow ( K \langle\langle e_{0\cup\mu_{N}} \rangle\rangle_{\text{har}}^{\smallint_{1,0}})^{q^{\mathbb{N}^{\ast}}}$ defined by
$f \mapsto \big( \tau(q^{\tilde{\alpha}})(f) \circ_{\har}^{\smallint_{1,0}} \comp^{\Sigma\smallint} f \big)_{\tilde{\alpha} \in \mathbb{N}^{\ast}}$.
\end{Definition}

Equation (\ref{eq:first of I-3}) can be restated as $$ \comp^{\Sigma\smallint}_{\iter}(\Ad_{\Phi_{q,-\infty}}(e_{1})) = (\har_{q,\tilde{\alpha}})_{\alpha \in \mathbb{N}^{\ast}} $$

The key property of $\comp^{\Sigma\smallint}_{\iter}$, which a priori does not hold for $\comp^{\Sigma\smallint}$, is:

\begin{Proposition} \label{prop injectivity comp iter} $\comp^{\Sigma\smallint}_{\iter}$ is injective.
\end{Proposition}

\begin{proof} Let a word $w=e_{\xi_{d+1}}e_{0}^{n_{d}-1}e_{\xi_{d}} \ldots e_{0}^{n_{1}-1}e_{\xi_{1}}$. For 
$n \geqslant \weight(w)$, let us consider the coefficient of $(q^{a})^{n}$ in $\comp^{\Sigma\smallint,\text{iter}}(f)[w]$. It is equal to $f[e_{0}^{n-(n_{1}+\ldots+n_{d})}e_{\xi_{d+1}}e_{0}^{n_{d}-1}e_{\xi_{d}} \ldots e_{0}^{n_{1}-1}e_{\xi_{1}}] $ + terms of lower depth, where the depth is the one of coefficients of $f$. This gives the result by an induction on the depth.
\end{proof}

\section{Iteration of the Frobenius on $\mathbb{P}^{1} - \cup_{\xi} B(\xi,1)$}

In the previous sections, we have considered the Frobenius of $\pi_{1}^{\un,\crys}(\mathbb{P}^{1} - \{0,\mu_{N},\infty\})$ at base-points (1,0) and the harmonic Frobenius. We now consider the Frobenius of $\pi_{1}^{\un,\crys}(\mathbb{P}^{1} - \{0,\mu_{N},\infty\})$ on the affinoid subspace $U^{\an}=\mathbb{P}^{1,\an} - \cup_{\xi^{N}=1} B(\xi,1)$ of $\mathbb{P}^{1,\an}$ over $K$, where $B(\xi,1)$ is the open ball of center $\xi$ and of radius $1$. As in \S2, we will have a fixed-point equation (\S5.1) and an iteration equation (\S5.2). Additionally, we will have a third equation coming from the study of regularized iterated integrals in \cite{I-1} (\S5.3).

\subsection{Fixed point equation}

The fixed point equation of the Frobenius is known thanks to the theory of Coleman integration. It amounts to the definition of $p$-adic multiple polylogarithms $\Li_{q}^{\KZ}$ as Coleman integrals, in \cite{Furusho 1} \cite{Furusho 2} for $N=1$ and in \cite{Yamashita} for any $N$: we have

\begin{equation} \label{eq:the fixed point equation} \phi_{\frac{\log(q)}{\log(p)}}( \Li_{q}^{\KZ} ) = \Li_{q}^{\KZ} .
\end{equation}

\indent Restricted to $U^{\an}$, the fixed point equation amounts to the following equations (\cite{I-1}, proposition 2.1.3), they involve the overconvergent $p$-adic multiple polylogarithms $\Li_{p,\alpha}^{\dagger}[w]$ (\cite{I-1}, definition 1.2.5), which are overconvergent analytic functions on $U^{\an}$:

\begin{equation} \label{eq:first} \Li_{p,\alpha}^{\dagger}(z) = 
\Li_{q}^{\KZ}(z)\big(p^{\alpha}e_{0},(p^{\alpha}e_{\xi})_{\xi} \big) \text{ }\text{ }
{\Li_{q}^{\KZ}}^{(p^{\alpha})}(z^{p^{\alpha}})\big(e_{0},(\Ad_{\Phi^{(\xi)}_{p,\alpha}}(e_{\xi})_{\xi} \big)^{-1};
\end{equation}

\begin{equation} \label{eq:second} \Li_{p,-\alpha}^{\dagger}(z) = 
{\Li_{q}^{\KZ}}^{(p^{\alpha})}(z^{p^{\alpha}})(e_{0},(e_{\xi^{(p^{\alpha})}})_{\xi})
\text{ }\text{ } \Li_{q}^{\KZ}(z)
\big(p^{\alpha}e_{0},p^{\alpha}(\Ad_{\Phi^{(\xi^{p^{\alpha}})}_{p,-\alpha}}(e_{\xi^{(p^{\alpha})}})_{\xi}\big)^{-1} .
\end{equation}

where ${\Li_{q}^{\KZ}}^{(p^{\alpha})}$ is the analog of $\Li_{q}^{\KZ}$ on $X^{(p^{\alpha})}$ equal to the pull-back of $X=(\mathbb{P}^{1} - \{0,\mu_{N},\infty\})\text{ }/\text{ }K$ by $\sigma^{\alpha}$ where $\sigma$ is the Frobenius automorphism of $K$. When $\alpha$ is a multiple of $\frac{\log(q)}{\log(p)}$, $X^{(p^{\alpha})} = X$ and ${\Li_{q}^{\KZ}}^{(p^{\alpha})} = \Li_{q}^{\KZ}$, and when $\alpha=\frac{\log(q)}{\log(p)}$, equations 
(\ref{eq:first}), (\ref{eq:second}) are directly equivalent to (\ref{eq:the fixed point equation}).

\begin{Notation} For any $\tilde{\alpha} \in \mathbb{Z} \cup \{\pm \infty\} - \{0\}$, let $\Li_{q,\tilde{\alpha}}^{\dagger}=\Li_{p,\alpha}^{\dagger}$ and $\Li_{q,-\tilde{\alpha}}^{\dagger}=\Li_{p,-\alpha}^{\dagger}$, with $p^{\alpha} = q^{\tilde{\alpha}}$.
\end{Notation}

The Frobenius on $U^{\an}$ is characterized by the couple $(\Li_{p,\alpha}^{\dagger},\Phi_{p,\alpha})$. We already know by \S2 a description of $\Phi_{p,\alpha}$ as a function of $\alpha$. If we combine it with equations (\ref{eq:first}), (\ref{eq:second}), we deduce a description of $\Li_{p,\alpha}^{\dagger}$ as a function of $\alpha$. This gives a description of the iterated Frobenius on $U^{\an}$ as a function of its number of iterations. We leave the details to the reader. One can also write an analogue for $U^{\an}$ of the notion of contraction mapping at base-points (1,0) of definition \ref{def contracting} and of the fact that the Frobenius at (1,0) is a contraction of lemma \ref{lemma Frobenius is p-contraction}.

Let us just consider the convergence of the iterated Frobenius towards the fixed point when the number of iterations tends to $\infty$, in the unit ball $\B(0,1)$.

\begin{Proposition} \label{prop convergence Li} For all $z \in K$ such that $|z|_{p}<1$, we have, for the $\mathcal{N}_{\Lambda,D}$-topology:
	$$ \tau(q^{-\tilde{\alpha}})\Li_{q,\tilde{\alpha}}^{\dagger}(z) 
	\underset{\tilde{\alpha} \rightarrow \infty}{\longrightarrow} \Li_{q}^{\KZ}(z)\big(e_{0},(e_{\xi})_{\xi} \big); $$
	$$ \tau(q^{-\tilde{\alpha}})\Li_{q,-\tilde{\alpha}}^{\dagger}(z)
	\underset{\tilde{\alpha} \rightarrow \infty}{\longrightarrow}
	\Li_{q}^{\KZ}(z)\big( e_{0},(\Ad_{\Phi_{q,-\infty}^{(\xi)}}(e_{\xi}))_{\xi} \big) . $$
	Moreover, these convergences are uniform on all the closed disks of center $0$ and radius $\rho < 1$. 
\end{Proposition}

\begin{proof} $\tau(q^{-\tilde{\alpha}})\Li_{q,\tilde{\alpha}}^{\dagger}(z)$ is the product of $ \Li_{q}^{\KZ}(z)\big(e_{0},(e_{\xi})_{\xi} \big)$ by 
	\begin{equation} \label{eq:factor} \tau(q^{-\tilde{\alpha}})\Li_{p,X_{K}^{(q^{\tilde{\alpha}})}}^{\KZ}(z^{q^{\tilde{\alpha}}}) \big(e_{0},(\Ad_{\Phi^{(\xi)}_{q,\tilde{\alpha}}}(e_{\xi}))_{\xi} \big)^{-1} .
	\end{equation}
	The coefficient of (\ref{eq:factor}) at a word $w$ is of the form, 
	$q^{-\tilde{\alpha}\weight(w)}\sum\limits_{(w_{1},w_{2})} \sum\limits_{m =0}^{\infty} \frak{h}_{q^{\alpha}m}(w_{1}) \frac{z^{q^{\tilde{\alpha}}m}}{(q^{\tilde{\alpha}}m)^{L}} \zeta_{q,\alpha}(w_{2})$ where  $L \in \mathbb{N}^{\ast}$ and $w_{1},w_{2}$ are in a finite set depending only on $w$, determined by the combinatorics of the composition of non-commutative formal power series.
	\newline\indent For all $m \in \mathbb{N}^{\ast}$ we have $-v_{p}(m) \geqslant - \frac{\log(m)}{\log(p)}$. Applying this to the $m_{i}$'s in (\ref{eq:harmonic sum}) we deduce $v_{p}(\har_{q^{\tilde{\alpha}}m}(w)) \geqslant -\weight(w) \frac{\tilde{\alpha}\log(q) + \log(n)}{\log(p)}$.  For all $C,C' \in \mathbb{R}^{+\ast}$, and $z \in K$ such that $|z|_{p}<1$, we have $q^{\tilde{\alpha}}n v_{p}(z) - C \tilde{\alpha} - C' \log(n) \underset{\tilde{\alpha}\rightarrow \infty}{\longrightarrow} +\infty$ and this convergence is uniform with respect to $n$.
	Indeed, let $n_{0}$ be an integer such that for all $n \geqslant n_{0}$, we have $C'\log(n) \leqslant \frac{n}{2} v_{p}(z)$; then $n_{0}$ is independent of $\tilde{\alpha}$ and we have, for all $n \geqslant n_{0}$, $q^{\tilde{\alpha}}n v_{p}(z) - C \tilde{\alpha} - C' \log(n) \geqslant \frac{q^{\tilde{\alpha}}n}{2} v_{p}(z) - C\tilde{\alpha}$. Because of the bounds of valuations of cyclotomic $p$-adic multiple zeta values of \cite{I-1}, \S4, the sequence $\big(\mathcal{N}_{\Lambda,D}(\Phi_{q,-\alpha})\big)_{\tilde{\alpha} \in \mathbb{N}^{\ast}}$ is bounded. Thus, $\tau(q^{-\tilde{\alpha}})\Li_{q,\alpha}^{\dagger}(z)$ converges to $\Li_{q}^{\KZ}(z)$. Moreover, we can see that $n_{0}$ can be chosen independently from $z$ in a closed disk of center $0$ and radius $\rho < 1$. 
	\newline\indent The proofs of the statements concerning $\tau(q^{-\tilde{\alpha}})\Li_{q,-\alpha}^{\dagger}(z)$ are similar.
\end{proof}

\begin{Remark} The convergence in proposition \ref{prop convergence Li}, (ii) does not a priori extend to a uniform convergence on $U^{\an}$. Indeed, otherwise, in fact, the map $\Li_{q}^{\KZ}$ would be rigid analytic on $U^{\an}$. By the main result of Appendix A of \cite{I-1}, this would imply that, for any word $w$, the multiple harmonic sums functions $m \mapsto \har_{m}(w)$ restricted to classes of congruences modulo $N$ should be continuous as a function $m \in \mathbb{N} \subset \varprojlim \mathbb{Z}/Np^{l}\mathbb{Z}$.
This seems to contradict the results of \S3.3. More generally, we expect that the lack of regularity of the maps $m \mapsto \har_{m}(w)$ can be at least partially reflected in the mode of convergence of the sequences $\Li_{q,\tilde{\alpha}}^{\dagger}[w]$ when $\alpha \rightarrow \infty$.
\end{Remark}

\begin{Remark} One can deduce a result similar to proposition \ref{prop convergence Li} on the ball $B(\infty,1)$ by applying the automorphism $z \mapsto \frac{1}{z}$ of $\mathbb{P}^{1} - \{0,\mu_{N},\infty\}$ and the functoriality of $\pi_{1}^{\un,\crys}$.
\end{Remark}

\subsection{Iteration equation}

We now write an equation for the iteration of the Frobenius on the subspace $\mathbb{P}^{1,\an} - \cup_{\xi\in \mu_{N}(K)} B(\xi,1)$. We restrict the statement to positive numbers of iterations for simplicity, but a similar result holds for negative numbers of iterations. If $f(z)=\sum\limits_{m=0}^{\infty} c_{m}z^{m}$ is a power series with coefficients in $K$, let $f^{(p^{\alpha})}(z) = \sum\limits_{m=0}^{\infty} \sigma^{\alpha}(c_{m})z^{m}$ where $\sigma$ is the Frobenius automorphism of $K$. 

\begin{Proposition} For any, $\alpha_{0},\alpha \in \mathbb{N}^{\ast}$ with $\alpha_{0}$ dividing $\alpha$, we have
\begin{multline*}
\Li_{p,\alpha}^{\dagger}  \big(e_{0},(e_{\xi})_{\xi}\big) =
\\  \Li_{p,\alpha_{0}}^{\dagger}\big(e_{0},(e_{\xi})_{\xi}\big) 
{\Li_{p,\alpha_{0}}^{\dagger}}^{(p^{\alpha_{0}})}\big(e_{0},(\Ad_{\Phi_{p,\alpha_{0}}^{(\xi)}}(e_{\xi})_{\xi}\big) \cdots  
{\Li_{p,\alpha_{0}}^{\dagger}}^{(p^{\frac{\alpha}{\alpha_{0}}-1})}\big(e_{0},(\Ad_{{\Phi_{p,\alpha_{0}}^{\alpha-1}}^{(\xi)}}(e_{\xi})_{\xi}\big) .
\end{multline*}
\end{Proposition}

\begin{proof} The crystalline Frobenius of $\pi_{1}^{\un}(\mathbb{P}^{1} - \{0,\mu_{N},\infty\})$, restricted to the rigid analytic sections on $\mathbb{P}^{1,\an} - \cup_{\xi\in \mu_{N}(K)} B(\xi,1)$, is given, with the conventions of \cite{I-1}, by the formula
\begin{equation}
\tau(p^{\alpha})\phi^{\alpha}: f\big(e_{0},(e_{\xi})_{\xi}\big)(z) \longmapsto \Li_{p,\alpha}^{\dagger}\big(e_{0},(e_{\xi})_{\xi}\big)(z) \times  f^{(p^{\alpha})}(z^{p^{\alpha}}) \big(e_{0}, \Ad_{\Phi^{(\xi)}_{p,\alpha}}(e_{\xi})_{\xi} \big) .
\end{equation}
This implies the result.
\end{proof}

\subsection{Another iteration equation via regularized iterated integrals}

The computation of overconvergent $p$-adic multiple polylogarithms in \cite{I-1}, which was centered around a notion of regularization of iterated integrals, gives us another point of view on how they depend on $\alpha$. Below, for a power series $S \in K[[z]]$, we denote by $S[z^{m}]$ the coefficient of $z^{m}$ in $S$ for all $m \in \mathbb{N}$. Again we restrict the statement to positive numbers of iteration of the Frobenius for simplicity, but a similar statement holds for negative number of iterations.

\begin{Proposition} \label{version precise du theoreme I 3 b}For any word $w$ on $e_{0\cup \mu_{N}}$ and any $m_{0} \in \mathbb{N}$, there exists a sequence
$\big(c^{(l,\xi,n)}[w](m_{0})\big)_{\substack{l\in \mathbb{N} \\ \xi \in \mu_{N}(K) \\ n\in\mathbb{N}}}$ of elements of $K$ such that, for any $\tilde{\alpha} \in \mathbb{N}$ such that $q^{\tilde{\alpha}}>m_{0}$ and $m \in \mathbb{N}^{\ast}$ satisfying $|m - m_{0}|_{p} \leqslant q^{-\tilde{\alpha}}$, we have
$$ \Li_{q,\tilde{\alpha}}^{\dagger} [w][z^{m}] = \sum_{n = 0}^{\infty} (q^{\tilde{\alpha}})^{n} \bigg( \sum_{l \in \mathbb{N}} \sum_{\xi} c^{(l,\xi,n)}[w](m_{0}) \xi^{-m} (m-m_{0})^{l} \bigg) . $$
\end{Proposition}

\begin{proof} In \cite{I-1}, \S3 we have defined a notion of regularized iterated integrals attached to any sequence of differential forms among 
$\displaystyle p^{\alpha}\frac{dz}{z}$, $\displaystyle \frac{p^{\alpha}dz}{z-\xi}$, $\xi \in \mu_{N}(K)$, $\displaystyle \frac{d(z^{p^{\alpha}})}{z^{p^{\alpha}}-\xi^{p^{\alpha}}}$, $\xi \in \mu_{N}(K)$. 
We have computed these regularized iterated integrals by induction on the depth, and this gives us information on how they depend on $\alpha$.
Namely, each regularized iterated integral is a rigid analytic function on $\mathbb{P}^{1,\an} - \cup_{\xi\in \mu_{N}(K)} B(\xi,1)$ which has a power series expansion $\sum\limits_{m=0}^{\infty}c_{m}z^{m}$ satisfying the following property: for any $m_{0} \in \{0,\ldots,p^{\alpha}-1\}$, there exists a sequence $(c^{(l,\xi)}(m_{0}))_{\substack{l \in \mathbb{N} \\ \xi \in \mu_{N}(K)}}$ of elements of $K$ such that, for all $m \in \mathbb{N}$ with $|m-m_{0}|_{p}\leqslant p^{-\alpha}$, we have $c_{m}= \sum\limits_{l=0}^{\infty} \sum_{\xi} c^{(l,\xi)}(m_{0})\xi^{-m}(m-m_{0})^{l}$.
\newline\indent In \cite{I-1}, Appendix B, we have showed that the numbers $c^{(l,\xi)}(m_{0})$ have an expression as certain sums of series involving multiple harmonic sums, and particularly prime weighted multiple harmonic sums.
\newline\indent In \cite{I-1} \S4, we have showed an expression of each $\Li_{p,\alpha}^{\dagger}[w]$ as a linear combination of regularized $p$-adic iterated integrals over the ring of $p$-adic cyclotomic multiple zeta values $\zeta_{p,\alpha}(w')$.
\newline\indent Combining these facts with the results of \S1 and \S2 on how $\zeta_{q,\tilde{\alpha}}$ and $\har_{q^{\tilde{\alpha}}}$ depend on $\tilde{\alpha}$, we deduce the result.
\end{proof}

\appendix

\section{A Poisson bracket corresponding to the pro-unipotent harmonic action of integrals at (1,0)}

We have seen that, by their definitions, the pro-unipotent harmonic actions (\S1.4 and definition \ref{def de Rham Ihara}) are connected to the Ihara product (equation (\ref{eq:Ihara action})). However, often in the literature, what is used is not the Ihara product but the corresponding Lie bracket, called the Ihara bracket, which is a Poisson bracket. In this section we explain that the pro-unipotent harmonic action of integrals at (1,0) (Definition \ref{def de Rham Ihara}) corresponds naturally to a Poisson Lie bracket.

\subsection{The Ihara bracket and the adjoint analogue}

Let $V^{\omega}$ be the group of automorphisms defined in \cite{Deligne Goncharov}, \S5.10. The Ihara bracket is the Lie bracket of $\Lie(V^{\omega})$, regarded via the isomorphism $\Lie(V^{\omega}) \simeq \Lie(\Pi_{1,0})$, $v \mapsto v( {}_1 \Pi_{0})$; namely, it is given by the following formula (\cite{Deligne Goncharov}, \S5.12-\S5.13) 
$$ \{f,g\} = [f,g] + D_{f}(g) + D_{g}(f) , $$
where $D_{f}$ is the derivation which sends $e_{0} \mapsto 0$ and $e_{\xi} \mapsto [f^{(\xi)},e_{\xi}]$ for any $\xi \in \mu_{N}(K)$. The Ihara bracket is a Poisson bracket; namely it satisfies the equality $\{fg,h\} = \{f,h\}g+f\{g,h\}$.

Let $\tilde{V}^{\omega}$ be the preimage of $\tilde{\Pi}_{1,0}$ by the isomorphism $V^{\omega} \simlra \Pi_{1,0}$, $v \mapsto v( {}_1 \Pi_{0})$.

We have defined in \cite{I-2}, definition 1.1.3, the adjoint Ihara product (equation (\ref{eq:adjoint Ihara action})) and we have proved in \cite{I-2}, proposition 1.1.4 that $\Ad(e_{1})$ is an isomorphism of groups from $(\tilde{\Pi}_{1,0}(K),\circ^{\smallint_{1,0}}) \simlra (\Ad_{\tilde{\Pi}_{1,0}(K)}(e_{1}),
\circ_{\Ad}^{\smallint_{1,0}})$.
\newline\indent We have viewed $\Lie(\tilde{V}^{\omega})$ as a subset of $K \langle \langle e_{0\cup \mu_{N}}\rangle\rangle$ and we can again view $\Lie(\Ad_{\tilde{V}^{\omega}}(e_{1}))$ as a subset of $K \langle \langle e_{0\cup \mu_{N}}\rangle\rangle$. The derivative of $f \mapsto f^{-1}e_{1}f$, the map $f \mapsto [e_{1},f]$ is the isomorphism of Lie algebras defined by $\Lie(\tilde{V}^{\omega}) \simlra \Lie(\Ad_{\tilde{V}^{\omega}}(e_{1}))$.

\begin{Proposition} \label{prop bracket adjoint} The Lie bracket of  $\Lie(\Ad_{\tilde{V}^{\omega}}(e_{1}))$
is $\{f,g\}_{\Ad} = d_{f}(g) - d_{g}(f)$ where $d_{f}$ is the derivation sending $e_{0} \mapsto 0$, $e_{\xi} \mapsto f^{(\xi)}$ for all $\xi \in \mu_{N}(K)$.
\newline\indent Let the product $.^{\smallint_{1,0}}_{\Ad}$ on $\Ad_{\tilde{\Pi}_{1,0}(K)}(e_{1})$ be defined by $(g^{-1}e_{1}g).^{\smallint_{1,0}}_{\Ad} (f^{-1}e_{1}f) = (gf)^{-1}e_{1}(gf)$, and its Lie version be $[g,e_{1}] .^{\smallint_{1,0}}_{\Ad} [f,e_{1}] = [gf,e_{1}]$.
\newline\indent 
Then, $\{.,.\}_{\Ad}$ is a Poisson bracket, namely, it satisfies 
$\{a ._{\Ad}^{\smallint_{1,0}} b,c\} = \{a,c\}.^{\smallint_{1,0}}_{\Ad}b+a.^{\smallint_{1,0}}_{\Ad}\{b,c\}$.
\end{Proposition}

\begin{proof} Let us proceed as in the proof of proposition 5.13 in \cite{Deligne Goncharov}. \newline We write $g=1+a\epsilon$.  When $\epsilon \rightarrow 0$, we have 
$(1+a\epsilon) \circ_{\Ad}^{\smallint_{1,0}} f = f + \epsilon d_{g}(f) + O(\epsilon^{2})$.					
\newline Thus, the action of $\Lie(V^{\omega})$ on $\Ad(\tilde{\Pi}_{1,0})(e_{1})$ by $\circ_{\Ad}^{\smallint_{1,0}}$ is by $g \mapsto d_{g}$. This map is injective.
\newline By the injectivity of $d$, we only have to show that $[d_{f},d_{g}] = d_{d_{f}(g) - d_{g}(f)}$. Since they are derivations, it is sufficient to prove that these two maps agree on $e_{1}$, and this follows directly from their definitions. 
\newline The fact that it is a Poisson bracket follows from the isomorphism of Lie algebras $\Lie(\tilde{V}^{\omega}) \simlra \Lie(\Ad_{\tilde{V}^{\omega}}(e_{1}))$.
\end{proof}

\begin{Definition} We call $\{.,.\}_{\Ad}$ the adjoint Ihara bracket.
\end{Definition}

\subsection{Harmonic analogue of the Ihara bracket}

We also have defined in definition \ref{def de Rham Ihara} the pro-unipotent harmonic action of integrals at (1,0), $\circ_{\har}^{\smallint_{1,0}}$, by pushing forward the adjoint Ihara action by the map $\comp^{\Sigma\smallint}$ (see \S4.2), which amounts to the map $S$ of definition 2.1.2. We are now going to push-forward $\{.,.\}_{\Ad}$ by a linear and injective version of the map $S$ of definition 2.1.2. 
Below, $K [[\Lambda]]\langle\langle e_{0\cup\mu_{N}} \rangle\rangle_{\har}^{\smallint_{1,0}}$ is defined like $K \langle\langle e_{0\cup\mu_{N}} \rangle\rangle_{\har}^{\smallint_{1,0}}$ with coefficients in $K [[\Lambda]]$ instead of $K$.

\begin{Definition} (i) Let $S_{\lambda}: K\langle\langle e_{0\cup \mu_{N}} \rangle\rangle_{\widetilde{o(1)}} \rightarrow K [[\Lambda]]\langle\langle e_{0\cup\mu_{N}} \rangle\rangle_{\har}^{\smallint_{1,0}}$ defined by 
\newline $h \mapsto \sum\limits_{\substack{d \in \mathbb{N}^{\ast} \\ \xi_{1},\ldots,\xi_{d+1} \in \mu_{N}(K) \\ n_{1},\ldots,n_{d} \in \mathbb{N}^{\ast}}} h \big[ \frac{1}{1- \Lambda e_{0}}e_{\xi_{d+1}}e_{0}^{n_{d}-1}e_{\xi_{d}} \ldots e_{0}^{n_{1}-1}e_{\xi_{1}}\big]\text{ }e_{\xi_{d+1}}e_{0}^{n_{d}-1}e_{\xi_{d}} \ldots e_{0}^{n_{1}-1}e_{\xi_{1}}$.
\newline (ii) \label{la def 1} Let $L_{\har}(K) = S_{\lambda} \Lie \Ad_{\tilde{V}^{\omega}(K)}(e_{1})$.
\end{Definition}

\begin{Proposition} $L_{\har}(K)$ has a canonical Lie bracket $\{.,.\}_{\har}$ defined by 
$$ \{ S_{\Lambda}f,S_{\Lambda}g\}_{\text{har}} = S_{\Lambda} (\{f,g\}_{\Ad}) . $$
It is a Poisson bracket, i.e. we have $\{a ._{\har}^{\smallint_{1,0}} b,c\} = \{a,c\}.^{\smallint_{1,0}}_{\har}b+a .^{\smallint_{1,0}}_{\har}\{b,c\}$, 
where the product $.^{\smallint_{1,0}}_{\har}$ is defined by $S_{\Lambda}g ._{\har}^{\smallint_{1,0}} S_{\Lambda} f = S_{\Lambda} (g ._{\Ad}^{\smallint_{1,0}} f)$.
\end{Proposition}

\begin{proof} Similar to the proof of proposition A.1.1.
\end{proof}

\begin{Definition} We call $\{.,.\}_{\har}$ the harmonic Ihara bracket.
\end{Definition}

\begin{Remark} (i) The harmonic Ihara bracket $\{.,.\}_{\har}$ corresponds to the group law $\tilde{\circ}_{\har}^{\smallint_{1,0}}$ defined by $ S_{\Lambda}g \tilde{\circ}_{\har}^{\smallint_{1,0}} S_{\Lambda}f = S_{\Lambda}(g \circ_{\Ad}^{\smallint_{1,0}} f)$. Because of the injectivity of $S_{\Lambda}g$, the group law $\tilde{\circ}_{\har}^{\smallint_{1,0}}$ can be thought of as another version of the pro-unipotent harmonic action of integrals $\circ_{\har}^{\smallint_{1,0}}$ of definition \ref{def de Rham Ihara}.
	
(ii) Another way to define a harmonic variant of the Ihara bracket would be to restrict to summable points $\Ad_{\tilde{\Pi}_{1,0}}(K)_{o(1)}(e_{1}) \subset \Ad_{\tilde{\Pi}_{1,0}}(K)_{o(1)}(e_{1})$ and to use  $\comp_{\iter}^{\Sigma\smallint}$, which is injective by proposition 4.2.3, instead of $\comp_{\Lambda}^{\Sigma\smallint}$.
\end{Remark}


\begin{thebibliography}{50}
\bibitem[Bes]{Besser} A. Besser - \emph{Coleman integration using the Tannakian formalism}, Math. Ann. 322 1 (2002), 19-48.
\bibitem[BF]{Besser Furusho} A. Besser, H. Furusho - \emph{The double shuffle relations for $p$-adic multiple zeta values}, AMS Contemporary Math., 416 (2006), 9-29.
\bibitem[CL]{CLS} B. Chiarellotto, B. Le Stum - \emph{F-isocristaux unipotents}, Compositio Math. 116 (1999), 81-110.
\bibitem[Co]{Coleman} R. Coleman - \emph{Dilogarithms, regulators and $p$-adic $L$-functions}, Invent. Math., 69, 2 (1982), 171-208.
\bibitem[D]{Deligne} P. Deligne - \emph{Le groupe fondamental de la droite projective moins trois points}, Galois Groups over $\mathbb{Q}$ (Berkeley, CA, 1987), Math. Sci. Res. Inst. Publ. 16, Springer-Verlag, New York (1989).
\bibitem[DG]{Deligne Goncharov} P. Deligne, A. B. Goncharov, \emph{Groupes fondamentaux motiviques de Tate mixtes}, Ann. Sci. Ecole Norm. Sup. 38, 1 (2005), 1-56.
\bibitem[F1]{Furusho 1} H. Furusho - \emph{$p$-adic multiple zeta values I -- p-adic multiple polylogarithms and the p-adic KZ equation}, Invent. Math., 155, 2 (2004), 253-286.
\bibitem[F2]{Furusho 2} H. Furusho - \emph{$p$-adic multiple zeta values II -- tannakian interpretations}, Amer. J. Math, 129, 4, (2007), 1105-1144.
\bibitem[FJ]{Furusho Jafari} H. Furusho, A. Jafari - \emph{Regularization and generalized double shuffle relations for p-adic multiple zeta values}, Compositio Math. 143 (2007), 1089-1107.
\bibitem[H]{Hoffman} M. Hoffman - \emph{Quasi-shuffle products}, Journal of Algebraic Combinatorics, 11 (2000), 49-68.
\bibitem[J I-1]{I-1} D. Jarossay, \emph{A bound on the norm of overconvergent $p$-adic multiple polylogarithms}, J. Number Theory.
\bibitem[J I-2]{I-2} D. Jarossay, \emph{Pro-unipotent harmonic actions and a computation of $p$-adic cyclotomic multiple zeta values}, arXiv:1501.04893 (submitted).
\bibitem[J II-1]{II-1} D. Jarossay, \emph{Adjoint cyclotomic multiple zeta values and cyclotomic multiple harmonic values}, arXiv:1412.5099 (submitted).
\bibitem[J II-2]{II-2} D. Jarossay, \emph{The adjoint quasi-shuffle relations of $p$-adic cyclotomic multiple zeta values recovered by explicit formulas}, arXiv:1601.01158.
\bibitem[J II-3]{II-3} D. Jarossay, \emph{Interpretation of cyclotomic multiple harmonic values as periods}, arXiv:1601.01159.
\bibitem[J III-1]{III-1} D. Jarossay, \emph{Definition and computation of ramified $p$-adic cyclotomic multiple zeta values}, arXiv:1708.08009.
\bibitem[U1]{Unver 1} S. \"{U}nver - \emph{$p$-adic multi-zeta values}, J. Number Theory, 108 (2004), 111-156.
\bibitem[U2]{Unver 3} S. \"{U}nver, \emph{A note on the algebra
of $p$-adic multi-zeta values}, Commun. Number Theory Phys., 9 (2015), no. 4, 689-705
\bibitem[U3]{Unver 2} S. \"{U}nver - \emph{Cyclotomic p-adic multi-zeta values in depth two}, Manuscripta Math., 149, 3-4 (2016), 405-441.
\bibitem[S1]{Shiho 1} A. Shiho - \emph{Crystalline fundamental groups. I. Isocristals on log crystalline site and log convergent site}, J. Math. Soc. Univ Tokyo 7, 4 (2000), 509-656.
\bibitem[S2]{Shiho 2} A. Shiho - \emph{Crystalline fundamental groups. II. Log convergent cohomology and rigid cohomology},  
J. Math. Soc. Univ. Tokyo 9, 1 (2002), 1-163.
\bibitem[V]{Vologodsky} V. Vologodsky, \emph{Hodge structure on the fundamental group and its application to $p$-adic integration}, Moscow Math. J. 3, 1 (2003), 205-247.
\bibitem[Y]{Yamashita} G. Yamashita, \emph{Bounds for the dimension of $p$-adic multiple $L$-values spaces}, Documenta Math., Extra Volume Suslin (2010), 687-723.
\end{thebibliography}
\end{document}